\documentclass{amsart}
\usepackage{latexsym,amscd,amssymb, graphicx, amsthm,amsmath}  
\usepackage{blkarray}
\usepackage{tikz-cd, xcolor}
\usetikzlibrary{calc}
\usepackage[margin=1in]{geometry}
\usepackage{ytableau}
\usepackage{placeins}
\usepackage[shortlabels]{enumitem}
\usetikzlibrary{decorations.markings}

\theoremstyle{definition}
\newtheorem{theorem}{Theorem}[section]
\newtheorem{proposition}[theorem]{Proposition}
\newtheorem{definition}[theorem]{Definition}
\newtheorem{corollary}[theorem]{Corollary}
\newtheorem{lemma}[theorem]{Lemma}

\newtheorem{problem}[theorem]{Problem}
\newtheorem{remark}[theorem]{Remark}
\newtheorem{example}[theorem]{Example}

\newcommand{\equicc}[4][1 cm]{
\draw (#3,#4) circle (#1);
\pgfmathparse{#1/1 cm+0.25};
\edef\oc{\pgfmathresult cm};
  \foreach \i in {1,...,#2} {
    \coordinate (N\i) at ($(#3,#4)+ (90-\i*360/#2:#1)$);
    \fill[black] (N\i) circle (0.05 cm);
    \draw ($(#3,#4)+ (90-\i*360/#2:\oc)$) node{$\i$};
  }
  \foreach \i in {2,...,#2} {
    \pgfmathparse{\i-1}
    \edef\j{\pgfmathresult}
  }
}

\newcommand{\equiccnolabels}[4][1 cm]{
\draw (#3,#4) circle (#1);
\pgfmathparse{#1/1 cm+0.25};
\edef\oc{\pgfmathresult cm};
  \foreach \i in {1,...,#2} {
    \coordinate (N\i) at ($(#3,#4)+ (90-\i*360/#2:#1)$);
    \fill[black] (N\i) circle (0.05 cm);
  }
  \foreach \i in {2,...,#2} {
    \pgfmathparse{\i-1}
    \edef\j{\pgfmathresult}
  }
}

\ytableausetup{boxframe=0.1em}
\begin{document}
\title{Rotation invariant webs for three row flamingo Specht modules}

\author{Jesse Kim}
\address
{Department of Mathematics \newline \indent
University of California, San Diego \newline \indent
La Jolla, CA, 92093-0112, USA}
\email{jvkim@ucsd.edu}

\begin{abstract}

We introduce a new rotation-invariant web basis for a family of Specht modules $S^{(d^3, 1^{n-3d})}$, indexed by normal plabic graphs satisfying a degree condition and resembling $A_2$ webs. We show that the $\mathfrak{S}_n$ action on our basis can be understood combinatorially via a set of skein relations. From this basis, we obtain a cyclic sieving result for a $q$-analog of the hook length formula for $\lambda$. Our construction extends the jellyfish invariants of Fraser, Patrias, Pechenik, and Striker and is closely related to the weblike subgraphs of Lam.
\end{abstract}

\maketitle

\section{Introduction}

This paper introduces a new {\em web basis} for a family of $\mathfrak{S}_n$ modules indexed by partitions $(d,d,d,1^{n-3d})$. The systematic study of web bases began with work of G. Kuperberg \cite{Kuperberg} in order to study the space of invariant tensors for simple Lie algebras and their quantum groups, though examples which are now considered web bases predate the term.  What exactly constitutes a web basis differs somewhat between authors, we will use a list of properties laid out by C. Fraser, R. Patrias, O. Pechenik, and J. Striker in \cite{FPPS}. The properties they give are:

\begin{enumerate}[(1)]
\item Each basis element is indexed by a planar diagram with $n$ boundary vertices, embedded in a disk.
\item There is a topological criterion allowing identification of basis diagrams.
\item The action of the long cycle $c= (12\dots n)$ on the basis is by rotation of diagrams.
\item The action of the long element $w_0 \in n(n-1)\dots 1$ on the basis is by reflection of diagrams.
\item There is a finite list of `skein relations' describing the action of a simple transposition $s_i$ on a basis diagram.
\end{enumerate}

The simplest web basis is the Temperley-Lieb basis for two-row rectangle shapes $(d,d)$, indexed by noncrossing perfect matchings of $2d$ vertices and studied by a variety of authors \cite{TL, RTW, Kuperberg, Rho10}. One useful property of the Temperley-Lieb basis is that it makes computation of the action of $\mathfrak{S}_n$ easy: to act by a permutation on a basis element, simply permute the matching, potentially introducing crossings, then resolve each crossing by replacing it with an uncrossing in both possible ways. This crossing resolution is called a skein relation, shown below.

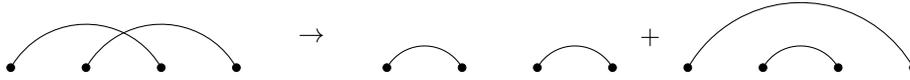
\begin{center}
\begin{tikzpicture}
\draw[fill=black] (0,0) circle (.02in);
\draw[fill=black] (1,0) circle (.02in);
\draw[fill=black] (2,0) circle (.02in);
\draw[fill=black] (3,0) circle (.02in);

\draw (2,0) arc (30:150:1.154);
\draw (3,0) arc (30:150:1.154);

\node at (4,.4) {$\rightarrow$};

\draw[fill=black] (5,0) circle (.02in);
\draw[fill=black] (6,0) circle (.02in);
\draw[fill=black] (7,0) circle (.02in);
\draw[fill=black] (8,0) circle (.02in);

\draw (6,0) arc (30:150:.577);
\draw (8,0) arc (30:150:.577);

\node at (8.5,.4) {$+$};

\draw[fill=black] (9,0) circle (.02in);
\draw[fill=black] (10,0) circle (.02in);
\draw[fill=black] (11,0) circle (.02in);
\draw[fill=black] (12,0) circle (.02in);

\draw (11,0) arc (30:150:.577);
\draw (12,0) arc (30:150:1.73);
\end{tikzpicture}
\end{center}

Kuperberg introduced similar bases for invariant spaces for rank-two Lie algebras. In this paper we will primarily be interested in the type $A_2$; the Temperley-Lieb basis is type $A_1$. In type $A_2$, Kuperber's basis is indexed by bipartite trivalent planar graphs with no faces of degree less than 6, called nonelliptic $SL_3$ webs. For each nonelliptic $SL_3$ web, the corresponding element of $V^{\otimes n}$, where $V$ is the three-dimensional defining representation of $SL_3$, is defined either recursively, in terms of tensor product and contraction, or combinatorially, in terms of a weighted sum over all proper edge coloring of the web with 3 colors. These form a basis for the space of $SL_3$ invariants of $V^{\otimes n}$, where $SL_3$ acts diagonally on $V^{\otimes n}$. The symmetric group $\mathfrak{S}_n$ acts on this invariant space by permuting tensor factors.  As an $\mathfrak{S}_n$ module, the $SL_3$ invariant space of $V^{\otimes n}$ is irreducible and isomorphic to the Specht module of shape $(d,d,d)$ where $n=3d$, and is 0 if $n$ is not a multiple of 3. Various ways to generalize this construction to $SL_n$ have been studied \cite{fontaine,CKM}, though a rotation invariant version is known only for $n$ up to 4 \cite{GPPSS}.

Rhoades generalized the Temperley-Lieb basis in a different direction, giving a web basis for shapes $(d,d,1^{n-2d})$ indexed by noncrossing set partitions with parts of size at least two \cite{Rhoskein}. In Rhoades' action, crossing resolution involved four skein relations based on the sizes of the crossing blocks. In previous work with Rhoades \cite{meR}, we showed that this action of noncrossing set partitions could be found within the top degree component of the fermionic diagonal coinvariant ring. One main goal of this paper is to generalize Kuperberg's $SL_3$ webs in an analogous way to Rhoades' generalization of the Temperley-Lieb basis. 

To do so, we build upon work of  R. Patrias, O. Pechenik, and J. Striker. In \cite{PPS}, they introduce {\em jellyfish invariants}, giving an alternate construction of Rhoades' basis within the homogeneous coordinate ring of a 2-step partial flag variety. Along with C. Fraser, they further develop these jellyfish invariants, reinterpreting them within the homogeneous coordinate ring of a Grassmanian \cite{FPPS}. They also extend their definitions to give jellyfish invariants living within a Specht modules indexed by a partition of shape $(d^r, 1^{n-rd})$. They dub these {\em flamingo Specht modules} as the partitions indexing them appear to stand on one leg. In the case $r>2$, they do not give a basis for this module. Instead, they give a linearly independent set indexed by noncrossing set partitions and a spanning set indexed by all set partitions. 

Our first result is to extend their linearly independent set to a basis of $S^{(d^r, 1^{n-rd})}$ in the case $r>2$. We do so by replacing the noncrossing condition with a weaker one based on ideas introduced by P. Pylyavskyy in \cite{Pylyavsky}, which we call {\em $r$-weakly noncrossing}. This basis fails to have the rotation and reflection invariance desired of a web basis, however. Our second result is to remedy the lack of rotation invariance in the case $r=3$ by introducing a second basis indexed by a certain rotationally invariant set $AW(n,d)$ of normal plabic graphs we call {\em augmented $SL_3$ webs}, as they closely resemble Kuperberg's $SL_3$ webs with extra edges. This resolves a conjecture made in \cite{me}. Plabic graphs were first introduced by A. Postnikov \cite{Postnikov} in order to study the totally nonnegative Grassmanian; we use the combinatorial machinery developed for them to show that our indexing set has the correct enumeration.  To define our basis, we use a modification of proper edge colorings for $SL_3$ webs which we call {\em consistent labellings}.  Consistent labellings are closely related to the weblike subgraphs introduced by T. Lam in order to define $SL_3$ web immanants and later used by C. Fraser, T. Lam, and I. Le to introduce a higher rank version of Postnikov's boundary measurement map \cite{Lam, FLL}, and we make this connection explicit.

Through consideration of the combinatorics of conistent labellings, we obtain skein relations for augmented $SL_3$ webs. These skein relations give a combinatorial description of the action of an adjacent transposition on an augmented web. The first skein relation, the crossing reduction rule shown below, shows how to expand the application of an adjacent transposition to an augmented web in the augmented web basis. The gray region represents an unknown number of edges connecting to other  vertices of the web not depicted.
\begin{center}
\begin{tikzpicture}[scale = .5]
\draw (-1,0)--(3,0);
\draw[thick] (0,0)--(0,-4);
\draw[fill = gray] (0,-4)--(-1,-4) arc (-180:-90:1);
\draw[thick] (2,0)--(2,-4);
\draw[fill = gray] (2,-4)--(3,-4) arc (0:-90:1);
\node at (0,.25) {$i$};
\node at (2,.25) {$i+1$};
\node at (3.5, -2) {$=$};
\node at (-1,-2) {$s_i \cdot$};
\draw[fill=black] (0,-4) circle (4pt);
\draw[fill=black] (2,-4) circle (4pt);
\end{tikzpicture}
\begin{tikzpicture}[scale = .5]
\draw (-1,0)--(3,0);
\draw[thick] (0,0)--(0,-4);
\draw[fill = gray] (0,-4)--(-1,-4) arc (-180:-90:1);
\draw[thick] (2,0)--(2,-4);
\draw[fill = gray] (2,-4)--(3,-4) arc (0:-90:1);
\node at (0,.4) {$i$};
\node at (2,.4) {$i+1$};
\node at (3.5, -2) {$-$};
\draw[fill=black] (0,-4) circle (4pt);
\draw[fill=black] (2,-4) circle (4pt);
\end{tikzpicture}
\begin{tikzpicture}[scale = .5]
\draw (-1,0)--(3,0);
\draw[thick] (0,0)--(1,-1)--(1,-3)--(0,-4);
\draw[thick, ->] (1,-1)--(1,-2);
\draw[fill = gray] (0,-4)--(-1,-4) arc (-180:-90:1);
\draw[thick] (2,0)--(1,-1)--(1,-3)--(2,-4);
\draw[fill = gray] (2,-4)--(3,-4) arc (0:-90:1);
\node at (0,.4) {$i$};
\node at (2,.4) {$i+1$};
\node at (3.5, -2) {$-\frac{1}{2}$};
\draw[fill=black] (0,-4) circle (4pt);
\draw[fill=black] (2,-4) circle (4pt);
\draw[fill=white] (1,-3) circle (4pt);
\draw[fill=black] (1,-1) circle (4pt);
\end{tikzpicture}
\begin{tikzpicture}[scale = .5]
\draw (-1,0)--(3,0);
\draw[thick] (0,0)--(0,-4);
\draw[fill = gray] (0,-4)--(-1,-4) arc (-180:-90:1);
\draw[thick] (2,0)--(0,-4);
\draw[fill = gray] (2,-4)--(3,-4) arc (0:-90:1);
\node at (0,.4) {$i$};
\node at (2,.4) {$i+1$};
\node at (3.5, -2) {$-\frac{1}{2}$};
\draw[fill=black] (0,-4) circle (4pt);
\draw[fill=black] (2,-4) circle (4pt);
\end{tikzpicture}
\begin{tikzpicture}[scale = .5]
\draw (-1,0)--(3,0);
\draw[thick] (0,0)--(2,-4);
\draw[fill = gray] (0,-4)--(-1,-4) arc (-180:-90:1);
\draw[thick] (2,0)--(2,-4);
\draw[fill = gray] (2,-4)--(3,-4) arc (0:-90:1);
\node at (0,.4) {$i$};
\node at (2,.4) {$i+1$};
\draw[fill=black] (0,-4) circle (4pt);
\draw[fill=black] (2,-4) circle (4pt);
\end{tikzpicture}
\end{center}
Note that not all terms on the right hand side of this relation are necessarily augmented webs, as black vertices of degree less than 3 or faces of degree 4 may be created. The remaining skein relations, which can be found in Section 6, show how to expand such terms when they arise. 

One application of our rotationally invariant basis in the case $r=3$ is that it gives us a cyclic sieving result on the indexing set. Let $X_{n,d}(q)$ denote the $q$-analog of the hook length formula for $\lambda = (d^3, 1^{n-3d})$, 
\[
X_{n,d}(q) := q^{3(d-1)+\binom{n-3(d-1)}{2}}\frac{[n]!_q}{\prod_{(i,j) \in \lambda} [h_{ij}]_q}
\]
where $h_{ij}$ denotes the hook length of a box $(i,j)$ in the Young diagram for $\lambda$. We show that the triple $(AW(n,d), C, X_{n,d}(q))$ exhibits the cyclic sieving phenomenon when $n$ is odd, and a signed version of cyclic sieving holds when $n$ is even. Specializing to the case $n=3d$ recovers a cyclic sieving result on $SL_3$ webs studied by T.K. Petersen, Pylyasky, and Rhoades in \cite{PPR}.

The rest of the paper is organized as follows. In Section 2, we review necessary background information. In Section 3 we give a definition of $r$-weakly noncrossing set partitions and give a basis of $S^{(d^r, 1^{n-rd})}$ which extends the jellyfish invariant basis. In Section 4, we define {\em augmented webs} as a certain subset of normal plabic graphs and give a combinatorial bijection between them and $3$-weakly noncrossing set partitions. This bijection will draw on ideas developed by J. Tymoczko and H. Russell to give a bijection between $SL_3$ webs and objects called {\em $m$-diagrams}, a special case of our 3-weakly noncrossing set partitions \cite{Russell, Tymoczko}. In Section 5, we define an $SL_3$-invariant polynomial attached to each normal plabic graph. We show that this definition extends jellyfish invariants and that the set of invariants attached to augmented webs satisfy properties (3) and (4) of a web basis. In Section 6, we show that skein relations hold for our plabic graph invariants. We use these skein relations to show that augmented web invariants are indeed a basis for $S^{(d^3, 1^{n-3d})}$. In Section 7, we show that our augmented web invariants can be interpreted in terms of weblike subgraphs. In Section 8, we discuss the cyclic sieving result for augmented webs which arises from our rotationally invariant basis. In Section 9, we discuss some possible future directions for this work.

\section{Background}

\subsection{Specht modules}

The irreducible representations of the symmetric group $\mathfrak{S}_n$ are indexed by integer partitions of $n$. The irreducible indexed by a partition $\lambda \vdash n$ is denoted $S^{\lambda}$ and is called the {\em Specht module} of shape $\lambda$. We give a construction of $S^{\lambda}$ following the approach of \cite{FPPS}. See \cite{Sagan} for a more detailed treatment of Specht modules.

Let $\lambda \vdash n$, and let $\lambda'$ be its transpose. Consider a matrix of $n\lambda_1'$ variables,
\[
M = \begin{bmatrix}
x_{1,1} & x_{1,2} & \cdots & x_{1,n}\\
x_{2,1} & x_{2,2} & \cdots & x_{2,n}\\
\vdots &  &\ddots& \vdots\\
x_{\lambda_1', 1} &x_{\lambda_1',2} & \cdots &x_{\lambda_1', n}
\end{bmatrix}
\]
The symmetric group $\mathfrak{S}_n$ acts on this matrix, and thus on $\mathbb{C}[M]$, by permuting columns of $M$. Let $\pi = \{\pi_1, \pi_2, \dots, \pi_{\lambda_1}\}$ denote a set partition of $n$ with shape $\lambda'$, i.e. the sizes of each part of the partition are given by the rows of $\lambda'$. For each such set partition $\pi$, let $p_\pi$ be the polynomial
\[
p_\pi = \prod_{i=1}^{\lambda_1} M_{[\lambda_i']}^{ \pi_i}
\]
where $M_{[\lambda_i']}^{ \pi_i}$ denotes the matrix minor of $M$ whose rows are indexed by $[\lambda_i']$ and columns are indexed by $\pi_i$.

Then the {\em Specht module} $S^\lambda$ is the span of these polynomials as $\pi$ ranges over all set partitions of shape $\lambda$.

\subsection{Jellyfish invariants}

Jellyfish invariants were introduced in \cite{PPS} and further developed in \cite{FPPS} in order to study the Specht module $S^{(d^r, 1^{n-rd})}$. An $(n,d,r)$-jellyfish invariant is a certain element of $S^{(d^r, 1^{n-rd})}$ attached to each ordered set partition of $[n]$ with $d$ blocks and all blocks at least size $r$. We include the basic definitions from \cite{FPPS} below, see their paper for examples and further exposition.

An {\em ordered set partition} of $n$ is a set partition with a total order on its blocks. Two blocks $A,B$ of an ordered set partition {\em cross} if there exist $a_1,a_2 \in A$ and $b_1,b_2 \in B$ such that $a_1 < b_1 < a_2< b_2$ or $b_1 < a_1 < b_2< a_2$. An ordered set partition is {\em noncrossing} if no two of its blocks cross. Let $\mathcal{OP}(n,d,r)$ denote the set of all ordered set partitions with exactly $d$ blocks and blocks of size at least $r$, and let $\mathcal{NCOP}(n,d,r)$ denote the set of all such partitions which are also noncrossing.

\begin{definition}
Let $\pi = \{\pi_1, \dots, \pi_d\} \in \mathcal{OP}(n,d,r)$ be an ordered set partition. Define the set of {\em $r$-jellyfish tableaux}, $\mathcal{J}_r(\pi)$ to be the set of generalized tableau $T$ with $d$ columns and $n-(d-1)r$ rows with the following constriants:
\begin{enumerate}
\item $T_{ij} \in [n]$ \textrm{ or } $T_{ij}$ \textrm{ is nonempty}
\item If $i\in [r], T_{ij}$ is nonempty.
\item If $i>r$, there exists exactly one $j$ such that $T_{ij}$ is nonempty
\item The nonempty entries in column $j$ are exactly the elements of $\pi_j$ in increasing order.
\end{enumerate}

For each $T \in \mathcal{J}_r(\pi)$, define a polynomial 
\[
J(T) = \prod_{i=1}^d M_{R_i(T)}^{\pi_i}
\]
where $R_i(T)$ is the set of rows containing an entry in $\pi_i$.

For each $\pi \in \mathcal{OP}(n,d,r)$, the {\em $r$-jellyfish invariant}, denoted $[\pi]_r$ is
\[
[\pi]_r = \sum_{T \in \mathcal{J}_r(\pi)} \textrm{sign}(T)J(T)
\]
where $\textrm{sign}(T)$ denotes the sign of the reading word of $T$.
\end{definition}

Fraser, Patrias, Pechenik, and Striker prove the following about $r$-jellyfish invariants:

\begin{theorem}[{{\cite[Theorem 4.24]{FPPS}}}]
\label{inspecht}
For each ordered set partition $\pi \in \mathcal{OP}(n,d,r)$, the invariant $[\pi]_r$ lies in the flamingo Specht module $S^{(d^r, 1^{n-rd})}$. 
\end{theorem}

\begin{theorem}[{{\cite[Proposition 5.11]{FPPS}}}]
\label{perminvariancejelly}
For any ordered set partition $\pi \in \mathcal{OP}(n,d,r)$ and any permutation $w \in \mathfrak{S}_n$, we have
\[
w \cdot [\pi]_r = \textrm{sign}(w)[w\cdot \pi]_r
\]
\end{theorem}
Note that this implies the span of jellyfish invariants is closed under the action of $\mathfrak{S}_n$, and must therefore be equal to $S^{(d^r, 1^{n-rd})}$.

\begin{theorem}[{{\cite[Theorem 5.13]{FPPS}}}]
\label{linindjelly}
For each noncrossing set partition $\gamma \in \mathcal{NC}(n,d,r)$, order the blocks in any way to create a corresponding ordered set partition $\pi_\gamma$. Then the set $\{[\pi_\gamma]_r : \gamma \in \mathcal{NC}(n,d,r)\}$ is linearly independent.
\end{theorem}

 Fraser, Patrias, Pechenik, and Striker thus give a spanning set of $S^{(d^r, 1^{n-rd})}$ indexed by all set partitions, and a linearly independent subset indexed by noncrossing set partitions. Thus, it is possible to choose a subset $S$ of set partitions such that $S$ indexes a basis and $S$ contains all noncrossing set partitions. We will show how to do so in Section 3.

\subsection{Noncrossing Tableaux}

Noncrossing tableaux were introduced by P. Pylyavskyy in \cite{Pylyavsky} to give a non-crossing counterpart to standard Young tableaux. Formally, noncrossing tableaux are set partitions; Pylyavskyy chose the name noncrossing tableaux to distinguish them from the more standard definition of noncrossing set partitions given in the previous subsection. As we will be using noncrossing tableaux in the context of set partitions, we will instead refer to these as {\em weakly} noncrossing set partitions. We will use a modification of this weaker condition to interpolate between strongly noncrossing set partitions and all set partitions.

\begin{definition}
Let $A = \{a_1 <a_2 < \cdots < a_{|A|}\}$ and $B= \{b_1 <b_2 < \cdots < b_{|B|}\}$ be two disjoint subsets of $[n]$ with $|A| \leq |B|$. We say $A$ and $B$ are {\em weakly noncrossing} if for all $1\leq i \leq |A|-1$ we do not have
\[
a_i < b_i < a_{i+1} < b_{i+1}
\]
or
\[
b_i < a_i < b_{i+1} < a_{i+1}
\]
and additionally, if $|A| <|B|$, we do not have
\[
b_{|A|} < a_{|A|} < b_{|B|}
\]

A set partition $\pi$ is {\em weakly noncrossing} if blocks of $\pi$ are pairwise weakly noncrossing.
\end{definition}

Pylyavskyy showed that weakly noncrossing set partitions of shape $\lambda \vdash n$ are in bijection with standard Young tableaux of shape $\lambda$.

\subsection{Plabic graphs}

Plabic graphs were introduced by Postnikov in order to study the totally nonnegative Grassmanian. A textbook treatment can be found in \cite{FWZ}. We will only need combinatorial results about plabic graphs, which we list here.

A {\em plabic graph} $G$ is a planar graph embedded in a disk, possibly with loops and multiple edges between vertices, with interior vertices colored black and white and boundary vertices labelled clockwise 1 through $n$. A {\em normal} plabic graph is a plabic graph for which white vertices are degree three, boundary vertices only connect to black vertices, and same colored vertices do not share an edge. For this paper, we consider only normal plabic graphs and state results only as they apply to normal plabic graphs, rather than including the full generality.

Given a normal plabic graph $G$ , the {\em trip} at  $i$ is the walk in $G$ starting at boundary vertex $i$ which turns right at every black vertex and left at every white vertex until it reaches the boundary at avertex we denote $\textrm{trip}(i)$. The function defined by $i \mapsto \textrm{trip}(i)$ is a permutation of $[n]$ and is called the {\em trip permutation} of $G$. The {\em exceedances} of $G$ are the exceedances of this permutation, i.e. those trips for which $\textrm{trip}(i) > i$.

Two normal plabic graphs are {\em normal move equivalent} if one can be obtained from the other via a seuquence of {\em normal urban renewal moves} and {\em normal flip moves}, which we now define. The normal urban renewal move is the move shown below, where filled in arcs represent any number of edges leading elsewhere in the graph
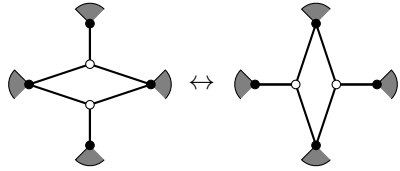
\begin{center}
\begin{tikzpicture}[scale = .27]
\draw[thick] (0,0)--(0,-2)--(-3,-3)--(0,-4)--(0,-6);
\draw[thick] (0,0)--(0,-2)--(3,-3)--(0,-4)--(0,-6);
\draw[fill = gray] (-3,-3)--(-3.71,-2.29) arc (135:225:1);
\draw[fill = gray] (3,-3)--(3.71,-2.29) arc (45:-45:1);
\draw[fill = gray] (0,0)--(-.71,.71) arc (135:45:1);
\draw[fill = gray] (0,-6)--(-.71,-6.71) arc (225:315:1);
\draw[fill = black] (0,0) circle (6pt);
\draw[fill = white] (0,-2) circle (6pt);
\draw[fill = black] (-3,-3) circle (6pt);
\draw[fill = black] (3,-3) circle (6pt);
\draw[fill = white] (0,-4) circle (6pt);
\draw[fill = black] (0,-6) circle (6pt);

\node at (5.5,-3) {$\leftrightarrow$};
\end{tikzpicture}
\begin{tikzpicture}[scale = .27]
\draw[thick] (-3,-3)--(-1,-3)--(0,-6)--(1,-3)--(3,-3);
\draw[thick] (-3,-3)--(-1,-3)--(0,0)--(1,-3)--(3,-3);
\draw[fill = gray] (-3,-3)--(-3.71,-2.29) arc (135:225:1);
\draw[fill = gray] (3,-3)--(3.71,-2.29) arc (45:-45:1);
\draw[fill = gray] (0,0)--(-.71,.71) arc (135:45:1);
\draw[fill = gray] (0,-6)--(-.71,-6.71) arc (225:315:1);
\draw[fill = black] (0,0) circle (6pt);
\draw[fill = white] (-1,-3) circle (6pt);
\draw[fill = black] (-3,-3) circle (6pt);
\draw[fill = black] (3,-3) circle (6pt);
\draw[fill = white] (1,-3) circle (6pt);
\draw[fill = black] (0,-6) circle (6pt);
\end{tikzpicture}
\end{center}
The normal flip move is 

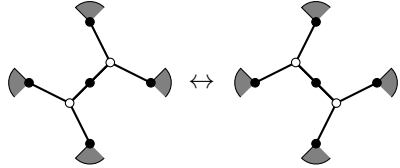
\begin{center}
\begin{tikzpicture}[scale = .27]
\draw[thick] (0,0)--(1,-2)--(0,-3)--(-1,-4)--(0,-6);
\draw[thick] (3,-3)--(1,-2)--(0,-3)--(-1,-4)--(-3,-3);
\draw[fill = gray] (-3,-3)--(-3.71,-2.29) arc (135:225:1);
\draw[fill = gray] (3,-3)--(3.71,-2.29) arc (45:-45:1);
\draw[fill = gray] (0,0)--(-.71,.71) arc (135:45:1);
\draw[fill = gray] (0,-6)--(-.71,-6.71) arc (225:315:1);
\draw[fill = black] (0,0) circle (6pt);
\draw[fill = white] (1,-2) circle (6pt);
\draw[fill = black] (-3,-3) circle (6pt);
\draw[fill = black] (0,-3) circle (6pt);
\draw[fill = black] (3,-3) circle (6pt);
\draw[fill = white] (-1,-4) circle (6pt);
\draw[fill = black] (0,-6) circle (6pt);

\node at (5.5,-3) {$\leftrightarrow$};
\end{tikzpicture}
\begin{tikzpicture}[scale = .27]
\draw[thick] (0,0)--(-1,-2)--(0,-3)--(1,-4)--(0,-6);
\draw[thick] (-3,-3)--(-1,-2)--(0,-3)--(1,-4)--(3,-3);
\draw[fill = gray] (-3,-3)--(-3.71,-2.29) arc (135:225:1);
\draw[fill = gray] (3,-3)--(3.71,-2.29) arc (45:-45:1);
\draw[fill = gray] (0,0)--(-.71,.71) arc (135:45:1);
\draw[fill = gray] (0,-6)--(-.71,-6.71) arc (225:315:1);
\draw[fill = black] (0,0) circle (6pt);
\draw[fill = white] (-1,-2) circle (6pt);
\draw[fill = black] (-3,-3) circle (6pt);
\draw[fill = black] (0,-3) circle (6pt);
\draw[fill = black] (3,-3) circle (6pt);
\draw[fill = white] (1,-4) circle (6pt);
\draw[fill = black] (0,-6) circle (6pt);
\end{tikzpicture}
\end{center}

A normal plabic graph is {\em reduced} if it is not normal move equivalent to any plabic graph which contains a {\em forbidden configuration}, i.e. a face of degree two or a leaf vertex not adjacent to the boundary.

A {\em bad feature} of a normal plabic graph $G$ is one of the following:
\begin{itemize}
\item A roundtrip: A cycle in $G$ which turns left at every white vertex and right at every black vertex.
\item An essential self-intersection: A trip in $G$ which passes through the same edge twice.
\item A bad double-crossing: Two trips in $G$ which both pass through edge $e_1$ then edge $e_2$ in that order.
\end{itemize}

\begin{theorem}[{{\cite[Theorem 7.8.6]{FWZ}}}]
A normal plabic graph is reduced if and only if it does not contain any bad features.
\end{theorem}

The more common use of this theorem is to test whether a plabic graph is reduced or not. The plabic graphs we are interested in, however, will be clearly reduced as their normal move equivalence class will have size 1. We will instead apply it to understand the structure of the trips of our plabic graphs.

\subsection{$SL_3$ Webs} $SL_3$ webs, or $A_2$ webs, were introduced by Kuperberg to study $SL_3$ invariant tensors and the representation theory of the quantum group $U_q(\mathfrak{sl}_3)$ \cite{Kuperberg}. The following is based on both Kuperberg's work as well as a nice discussion of the topic by S. Fomin and P. Pylyavskyy in \cite{FP}. A sign string of length $n$ is a string containing $n$ letters, all each $+$ or $-$, e.g. $(++--++-)$. Given a sign string $s = s_1s_2\cdots s_n$, an $SL_3$ web of type $s$ is a bipartite plabic graph with $n$ boundary vertices in which every interior vertex has degree 3 and boundary vertex $i$ is adjacent to a black vertex if $s_i = +$ and a white vertex if $s_i = -$. This is a slightly anachronistic version of the definition, as plabic graphs were defined after $SL_3$ webs, but the comparison will be useful for us later.

$SL_3$ webs have representation theoretic meaning. Let $V$ be the three-dimensional defining representation of $SL_3$, with basis $\{e_1, e_2, e_3\}$, and let $V^*$ denote its dual with dual basis $\{e_1^*, e_2^*, e_3^*\}$. An $SL_3$ web with sign string $(+++--++)$ e.g. represents an element of the space 
\[
(V \otimes V \otimes V \otimes V^* \otimes V^* \otimes V \otimes V)^{SL_3}
\]
of $SL_3$ invariant elements of $(V \otimes V \otimes V \otimes V^* \otimes V^* \otimes V \otimes V)$ where $V$ is the three-dimensional defining representation of $SL_3$, $+$ correspond to copies of $V$ and $-$ correspond to cpoies of $V^*$. 

The unique $SL_3$ web of sign string $(+++)$ 

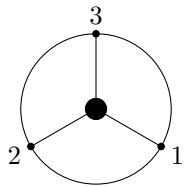
\begin{center}
\begin{tikzpicture}
\equicc[1cm]{3}{0}{0}
\draw (N1)--(0,0)--(N2);
\draw (N3)--(0,0);
\draw[fill = black] (0,0) circle (4pt);
\end{tikzpicture}
\end{center}
represents the tensor
$\sum_{\sigma \in \mathfrak{S}_3} \textrm{sign}(\sigma) e_{\sigma(1)} \otimes e_{\sigma(2)} \otimes e_{\sigma(3)} $
and the unique $SL_3$ web of sign string $(---)$ 

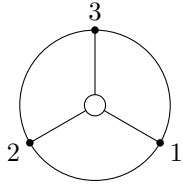
\begin{center}
\begin{tikzpicture}
\equicc[1cm]{3}{0}{0}
\draw (N1)--(0,0)--(N2);
\draw (N3)--(0,0);
\draw[fill = white] (0,0) circle (4pt);
\end{tikzpicture}
\end{center}

represents the tensor $\sum_{\sigma \in \mathfrak{S}_3} \textrm{sign}(\sigma) e_{\sigma(1)}^* \otimes e_{\sigma(2)}^* \otimes e_{\sigma(3)}^*
$. Concatenation of webs represents tensor product, and an edge between vertices represents tensor contraction.

We can also give a purely combinatorial desciption of the invariant each web represents in terms of proper edge colorings. A {\em proper edge coloring} $\ell$ of an $SL_3$ web $W$ is a labelling of the edges by the numbers $1,2,3$ such that no label appears more than once around each vertex. For each labelling, we get a simple basis tensor $T_\ell$ by taking the basis vector or dual basis vector $e_j$ or $e_j^*$ (depending on th sign string) at boundary vertex $i$ whose incident edge has label $j$, and a sign $\textrm{sign}(\ell)$ given by $(-1)^{cc(\ell)}$, where $cc(\ell)$ denotes the number of interior vertices for which $1,2,3$ appear in counterclockwise order in the labelling $\ell$ . The $SL_3$ invariant associated to $W$, which we denote $[W]_{SL_3}$ is
\[
[W]_{SL_3} = \sum_{\textrm{proper labellings } \ell} \textrm{sign}(\ell)T_\ell 
\]

A web is called {\em nonelliptic} if it contains no faces of degree 4 or less. The invariants for the set of all noneeliptic webs form a basis for the space of $SL_3$ invariant tensors.

\subsection{Cyclic sieving}

The cyclic sieving phenomenon was introduced by V. Reiner, D. Stanton, and D. White in order to unify a number of enumerative results in combinatorics \cite{RSW}. See their paper and a survey by B. Sagan for further reference \cite{Sagansurvey}.

\begin{definition}
Let $X$ be a finite set equipped with an action of the finite cyclic group $C \cong \mathbb{Z}/n\mathbb{Z}$ with generator $c$, let $X(q)$ be a polynomial, and let $\zeta$ be an $n^{th}$ root of unity. The triple $(X, C, X(q)$ is said to exhibit the {\em cyclic sieving phenomenon} if $|X^{c^d}| = X(\zeta^d)$ for any integer $d>0$, where $X^{c^d}$ denotes the set of all elements of $X$ fixed by $c^d$.
\end{definition}

One way of obtaining cyclic sieving results is via the following, which can be found in Sagan's survey \cite{Sagansurvey} and follows from a result of Springer \cite{Springer}. 
\begin{theorem}[{{\cite[Theorem 8.2]{Sagansurvey}, \cite{Springer}}}]
\label{sieving}
Let $W$ be a finite complex reflection group and let $C \leq W$ be cyclically generated by a regular element $g$. Let $V$ be a $W$-module with a basis $B$ such that $gB = B$. Then the triple
\[
(B, C, F^{V}(q))
\]
exhibits the cyclic sieving phenomenon, where $F^V(q)$ denotes the fake degree polynomial for $V$.
\end{theorem}
See \cite{Sagansurvey} for a complete definition of the fake degree polynomial, we will only need the following.
\begin{proposition}
Let $\lambda$ be a partition of $n$ and let $S^\lambda$ be the corresponding Specht module. The fake degree polynomial $F^{S^\lambda}(q)$ is given by
\[
F^{S^\lambda}(q) = q^{b(\lambda)}\frac{[n]!_q}{\prod_{(i,j) \in \lambda} [h_{ij}]_q}
\]
where $b(\lambda) = 0\lambda_1 + \lambda_2 + 2\lambda_3 + \cdots$ and $h_{ij}$ denotes the hook length of box $(i,j)$ in the Young diagram of $\lambda$.
\end{proposition}

\section{Weakly-noncrossing jellyfish invariants are a basis}

In this section, we extend the linearly independent set given in \cite{FPPS} to a basis by introducing a weaker version of the noncrossing condition for ordered set partitions. The weaker version is similar to the noncrossing tableau defined by P. Pylyavskyy in \cite{Pylyavsky}. Our version will differ in that it will depend on $r$.

\begin{definition}
Let $A = \{a_1,a_2, \dots, a_{|A|}\}$ and $B = \{b_1, b_2, \dots, b_{|B|}\}$ be two subsets of $[n]$ each of size $\geq r$. We say that $A$ and $B$ are $r$-weakly noncrossing if the following holds:
\begin{enumerate}
\item For each $i = 1\dots, r-2$, The arc $(a_i,a_{i+1})$ does not cross the arc $(b_i, b_{i+1})$.
\item For any $j_1,j_2 \geq r$, the arc $(a_{r-1}, a_{j_1})$ does not cross the arc $(b_{r-1}, b_{j_2})$.
\end{enumerate}
An (ordered) set partition is $r$-weakly noncrossing if its blocks are pairwise $r$-weakly noncrossing.
\end{definition}
One can think of this definition as being noncrossing in the sense of Pylyavskyy in the first $r-2$ entries, and noncrossing in the strong sense in the remaining entries.

Let $WNC(n,d,r)$ denote the set of all set partitions of $[n]$ into $d$ blocks each of size at least $r$ which are $r$-weakly noncrossing.

We first show that the set of $r$-weakly noncrossing set partitions is the correct size:
\begin{proposition}
\label{tableauxbiject}
There is a bijection between standard Young tableaux of shape $(d^r, 1^{n-rd})$ and $r$-weakly noncrossing set partitions in $WNC(n,d,r)$.
\end{proposition}

\begin{proof}
We show that both sets are in bijection with a set of rectangular tableaux filled with a subset of $[n]$:

\begin{definition}
Let $T(n,d,r)$ denote the set of all tableaux of shape $ \lambda = (d^r)$ filled with with integers in $[n]$ such that
\begin{enumerate}
\item Entries increase along rows and down columns.
\item No element of $[n]$ appears more than once.
\item For any $i$ which does not appear in the tableaux, the number of entries $j<i$ appearing in row $r-1$ strictly exceeds the number of entries $j<i$ appearing in row $r$.
\end{enumerate}
\end{definition}

\begin{example} Consider the two tableaux below.
\begin{center}
\begin{ytableau}
1&3&6&7\\
2&4&8&11\\
9&13&14&16\\
\end{ytableau}
\qquad\qquad\qquad
\begin{ytableau}
1&3&6&7\\
*(lightgray)2&*(lightgray)4&11&13\\
*(lightgray)5&*(lightgray)8&14&16\\
\end{ytableau}
\end{center}
The tableau on the left is an element of $T(16,4,3)$. The tableau on the right is not in $T(16,4,3)$ because 9 does not appear as a filling and there are the same number of fillings less than 9 in the second and third rows, highlighted in gray.

\end{example}

The bijection between $SYT(d^r,1^{n-rd})$ and $T(n,d,r)$ is as follows. Let $t$ be a standard Young tableaux of shape $(d^r,1^{n-rd})$.
\begin{enumerate}
\item
 If removing the blocks in rows larger than $r$ (which we will refer to as the {\em tail}) produces a tableaux in $T(n,d,r)$, do so. 
 \item Otherwise, let $i$ be the maximal element among the tail for which the number of elemenents $j<i$ in row $r-1$ equals the number of elements $j<i$ in row $r$. Remove the first block of row $r$ and all blocks below it, shift all blocks in row $r$ filled with $j<i$ one space to the left, and place a block filled with $i$ in the newly formed opening. 
 \end{enumerate}

The maximality of $i$ will guarantee that the third property of $T(n,d,r)$ is satisfied for elements larger than $i$, and the shifting left will guarantee it is satisfied for elements smaller than $i$. Call the resulting tableau $f(t)$.

To reverse this process, let $t'$ be a tableaux in $T(n,d,r)$. We obtain a standard Young tableaux of shape $(d^r, 1^{n-rd})$ as follows.
\begin{enumerate}
\item If the smallest element which does not appear in $t'$ is larger than the entry in the first box of row $r$, simply append all integers not already appearing in the tableau in increasing order as the tail.
\item Otherwise, let $i$ be the minimal filling in row $r$ which is smaller than the filling of the box one space up and to the right, or the largest element of row $r$ if no such element exists. Remove the box filled with $i$, shift all boxes to the left of it one space right, and insert the remaining entries in increasing order to form the tail. 
\end{enumerate}

The right shift will preserve the standard Young tableau property due to the minimality of $i$. Call the resulting tableau $g(t')$.

To verify that these two maps are indeed inverses, let $t \in SYT(d^r,1^{n-rd})$. If removing the tail of $t$ produces a tableau in $T(n,d,r)$, then it is clear that $g(f(t)) = t$. Otherwise, let $i$ be the element inserted into row $r$ to obtain $f(t)$. Before this insertion, there were the same number of elements less than $i$ in row $r-1$ and row $r$ of the tableau, so the filling one space up and to the right of $i$ in $f(t)$ must be larger than $i$. Additionally, all boxes $j$ to the left of $i$ were shifted over, and since we started with a standard Young tableau, the filling one space up and to the right of them is smaller than $j$. Therefore, $i$ is the filling removed by $g$, and $g(f(t)) = t$.

A similar argument shows that $f \circ g$ is also the identity. Indeed, if $i$ was the element removed from row $r$ by $g$, then there must be the same number of elements $j<i$ in row $r-1$ and $r$ of $g(t')$, and no other element larger than $i$ can have this propery as $t'$ is in $T(n,d,r)$. 

\begin{example} An example of this bijection is given below for $n=16, d=4, r=3$.  12 is the maximal filling of the tail for which the second and third rows contain the same number of lesser fillings.
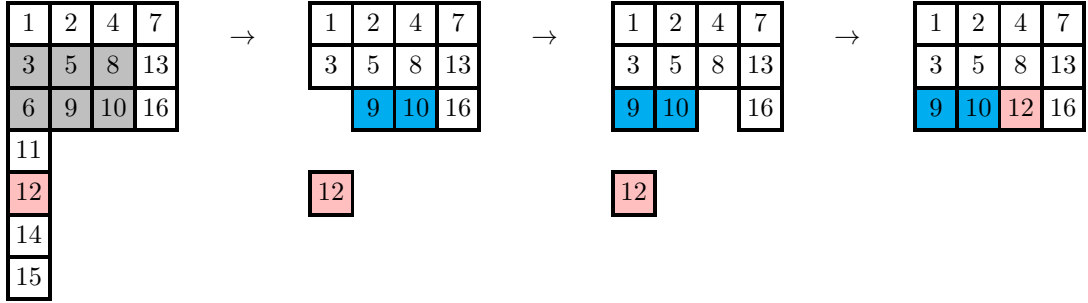
\begin{center}

\begin{ytableau}
1&2&4&7\\
*(lightgray) 3&*(lightgray) 5&*(lightgray) 8&13\\
*(lightgray) 6& *(lightgray) 9&*(lightgray) 10&16\\
11\\
*(pink)12\\
14\\
15\\
\end{ytableau}\qquad$\rightarrow$\qquad
\begin{ytableau}
1&2&4&7\\
3&5&8&13\\
\none&*(cyan)9&*(cyan)10&16\\
\none\\
*(pink)12\\
\none\\
\none\\
\end{ytableau}\qquad$\rightarrow$\qquad
\begin{ytableau}
1&2&4&7\\
3&5&8&13\\
*(cyan)9&*(cyan)10&\none&16\\
\none\\
*(pink)12\\
\none\\
\none\\
\end{ytableau}\qquad$\rightarrow$\qquad
\begin{ytableau}
1&2&4&7\\
3&5&8&13\\
*(cyan)9&*(cyan)10&*(pink)12&16\\
\end{ytableau}
\end{center}
\end{example}

The bijection between $T(n,d,r)$ and $WNC(n,d,r)$ is essentially repeated applications of the standard Catalan bijection between two row rectangular standard Young tableaux and noncrossing matchings. Given a tableau $t \in T(n,d,r)$, for $i=1, \dots, r$, let $R_i(t)$ denote the entries in row $i$ of $t$. Place the numbers 1 through $n$ in a line, and for each $i =1 , \dots, r-1$, draw $d$ arcs between elements of $R_i$ and $R_{i+1}$ such that
\begin{enumerate}
\item Elements of $R_i$ are the left endpoints of arcs, and elements of $R_{i+1}$ are the right endpoints of arcs.
\item There do not exist two arcs $(a,b)$ and $(c,d)$ such that $a<c<b<d$.
\end{enumerate}
The standard Catalan bijection argument guarantees that this is uniquely possible.
Then, for each positive integer $m$ at most $n$ not appearing in $t$, there is a unique shortest arc $(a,b)$ created at step $r-1$ such that $a<m<b$. The third condition of $T(n,d,r)$ guarantees that such an arc exists, and the noncrossing condition above guarantees it is unique. Draw the arc $(a,m)$. Finally, create a set partition $\pi$ by placing all integers connected by arcs into the same block. Then $\pi \in WNC(n,d,r)$. To see that the noncrossing condition is satisfied, note that for $i= 1, \cdots, r-2$ if $a_i$ and $a_{i+1}$ are the $i^{th}$ and $(i+1)^{th}$ smallest elements of a block of $\pi$, then they must necessarily be connected by an arc created at step $i$ in the above process.

The inverse of this bijection is simple, given a set partition $\pi \in T(n,d,r)$ place the smallest element of each block in increasing order in row 1, the second smallest in row 2, and so on, up to row $r-1$. Finally, place the largest element of each block in row $r$. 

\begin{example} Consider the tableau shown below.

\begin{center}
\begin{ytableau}
1&2&4&7\\
3&5&8&11\\
9&10&12&16\\
\end{ytableau}
\end{center}

We get the following arc diagram. Arcs created by matching the first two rows are shown in red, arcs created by matching the second and third rows are shown in green, and arcs created by connecting missing entries are shown in blue.

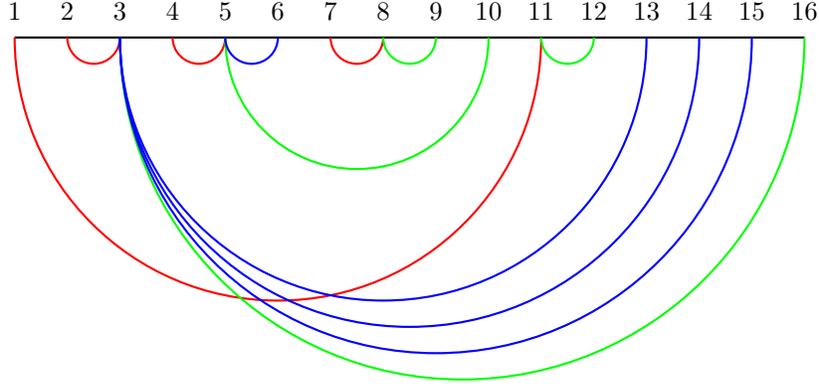
\begin{center}
\begin{tikzpicture}[scale = .7]
\coordinate (p1) at (0,0);
\coordinate (p2) at (1,0);
\coordinate (p3) at (2,0);
\coordinate (p4) at (3,0);
\coordinate (p5) at (4,0);
\coordinate (p6) at (5,0);
\coordinate (p7) at (6,0);
\coordinate (p8) at (7,0);
\coordinate (p9) at (8,0);
\coordinate (p10) at (9,0);
\coordinate (p11) at (10,0);
\coordinate (p12) at (11,0);
\coordinate (p13) at (12,0);
\coordinate (p14) at (13,0);
\coordinate (p15) at (14,0);
\coordinate (p16) at (15,0);

\foreach \i in {1,...,16}{
\node at ($(p\i)-(0,-.5)$) {$\i$};
}
\draw[thick] (p1)--(p16);
\draw[thick,red]  (p11) arc (0:-180: 5);
\draw[thick,red]  (p3) arc (0:-180:.5);
\draw[thick,red]  (p5) arc (0:-180:.5); 
\draw[thick,red]  (p8) arc (0:-180:.5); 
\draw[thick,green]  (p12) arc (0:-180:.5); 
\draw[thick,green]  (p10) arc (0:-180:2.5); 
\draw[thick,green]  (p9) arc (0:-180:.5); 
\draw[thick,green]  (p16) arc (0:-180:6.5); 

\draw[thick,blue]  (p6) arc (0:-180:.5); 
\draw[thick,blue]  (p13) arc (0:-180:5); 
\draw[thick,blue]  (p14) arc (0:-180:5.5); 
\draw[thick,blue]  (p15) arc (0:-180:6); 
\end{tikzpicture}
\end{center}

The resulting set partition is $\{\{1,11,12\}, \{2,3,13,14,15,16\}, \{4,5,6,10\}, \{7,8,9\}\}$.

\end{example}

\end{proof}

The second half of the proof of Proposition~\ref{tableauxbiject} also gives the following corollary, which we will need later:

\begin{corollary}
\label{setsunique}
A $r$-weakly noncrossing set partition $\gamma$ is uniquely determined by the $r-1$ sets 
\[\{m \mid m \textrm{ is the } i^{th} \textrm{ smallest element of some block of } \gamma\}
\] for $i=1, \dots, r-1$, along with the set 
\[\{m \mid m \textrm{ is the largest element of some block of } \gamma\}\].
\end{corollary}
\begin{proof}
The information in these sets determines the elements of each row of the tableau in $T(n,d,r)$ as defined in the proof of Proposition~\ref{tableauxbiject}. Placing elements in increasing order within each row recovers the tableau, and thus the set partition.
\end{proof}

\begin{example}
Suppose $n=7$, $d=3$, $r=2$, and thus $\lambda = (2,2,2,1)$ . There are fourteen standard Young tableaux of shape $\lambda$, and fourteen 2-weakly noncrossing set partitions. The bijection between them is shown below, with the intermediary tableau in $T(7,3,2)$ and arc diagram shown as well.

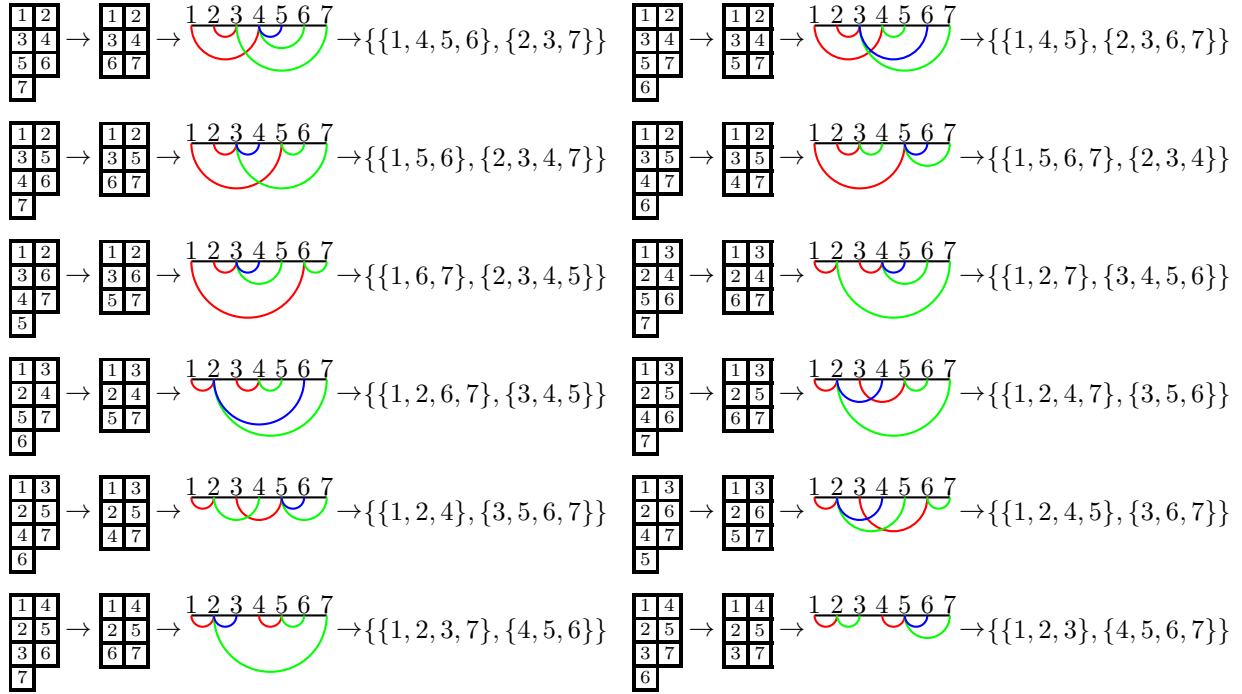
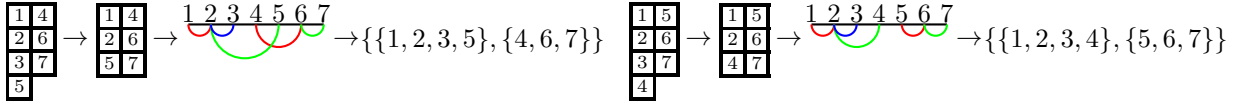
\begin{center}
\ytableausetup{smalltableaux}
\begin{tikzpicture}[scale = .3]
\node at (-7,-1.2){\begin{ytableau}
1&2\\
3&4\\
5&6\\
7
\end{ytableau}};

\node at (-5,-.7) {$\rightarrow$};

\node at (-3,-.7) { \begin{ytableau}
1&2\\
3&4\\
6&7\\
\end{ytableau}};

\node at (-1,-.7) {$\rightarrow$};
\coordinate (p1) at (0,0);
\coordinate (p2) at (1,0);
\coordinate (p3) at (2,0);
\coordinate (p4) at (3,0);
\coordinate (p5) at (4,0);
\coordinate (p6) at (5,0);
\coordinate (p7) at (6,0);
\foreach \i in {1,...,7}{
\node at ($(p\i)-(0,-.5)$) {$\i$};
}
\draw[thick] (p1)--(p7);

\draw[thick,red]  (p4) arc (0:-180: 1.5);
\draw[thick,red]  (p3) arc (0:-180:.5);

\draw[thick,green]  (p6) arc (0:-180:1); 
\draw[thick,green]  (p7) arc (0:-180:2); 

\draw[thick,blue] (p5) arc (0:-180:.5);

\node at (7,-.7) {$\rightarrow$};

\node at (13,-.7) {$\{\{1,4,5,6\}, \{2,3,7\}\}$};

\end{tikzpicture}
\begin{tikzpicture}[scale = .3]
\node at (-7,-1.2){\begin{ytableau}
1&2\\
3&4\\
5&7\\
6
\end{ytableau}};

\node at (-5,-.7) {$\rightarrow$};

\node at (-3,-.7) { \begin{ytableau}
1&2\\
3&4\\
5&7\\
\end{ytableau}};

\node at (-1,-.7) {$\rightarrow$};
\coordinate (p1) at (0,0);
\coordinate (p2) at (1,0);
\coordinate (p3) at (2,0);
\coordinate (p4) at (3,0);
\coordinate (p5) at (4,0);
\coordinate (p6) at (5,0);
\coordinate (p7) at (6,0);
\foreach \i in {1,...,7}{
\node at ($(p\i)-(0,-.5)$) {$\i$};
}
\draw[thick] (p1)--(p7);

\draw[thick,red]  (p4) arc (0:-180: 1.5);
\draw[thick,red]  (p3) arc (0:-180:.5);

\draw[thick,green]  (p5) arc (0:-180:.5); 
\draw[thick,green]  (p7) arc (0:-180:2); 

\draw[thick,blue] (p6) arc (0:-180:1.5);

\node at (7,-.7) {$\rightarrow$};

\node at (13,-.7) {$\{\{1,4,5\}, \{2,3,6,7\}\}$};

\end{tikzpicture}

\begin{tikzpicture}[scale = .3]
\node at (-7,-1.2){\begin{ytableau}
1&2\\
3&5\\
4&6\\
7
\end{ytableau}};

\node at (-5,-.7) {$\rightarrow$};

\node at (-3,-.7) { \begin{ytableau}
1&2\\
3&5\\
6&7\\
\end{ytableau}};

\node at (-1,-.7) {$\rightarrow$};
\coordinate (p1) at (0,0);
\coordinate (p2) at (1,0);
\coordinate (p3) at (2,0);
\coordinate (p4) at (3,0);
\coordinate (p5) at (4,0);
\coordinate (p6) at (5,0);
\coordinate (p7) at (6,0);
\foreach \i in {1,...,7}{
\node at ($(p\i)-(0,-.5)$) {$\i$};
}
\draw[thick] (p1)--(p7);

\draw[thick,red]  (p5) arc (0:-180: 2);
\draw[thick,red]  (p3) arc (0:-180:.5);

\draw[thick,green]  (p6) arc (0:-180:.5); 
\draw[thick,green]  (p7) arc (0:-180:2); 

\draw[thick,blue] (p4) arc (0:-180:.5);

\node at (7,-.7) {$\rightarrow$};

\node at (13,-.7) {$\{\{1,5,6\}, \{2,3,4,7\}\}$};

\end{tikzpicture}
\begin{tikzpicture}[scale = .3]
\node at (-7,-1.2){\begin{ytableau}
1&2\\
3&5\\
4&7\\
6
\end{ytableau}};

\node at (-5,-.7) {$\rightarrow$};

\node at (-3,-.7) { \begin{ytableau}
1&2\\
3&5\\
4&7\\
\end{ytableau}};

\node at (-1,-.7) {$\rightarrow$};
\coordinate (p1) at (0,0);
\coordinate (p2) at (1,0);
\coordinate (p3) at (2,0);
\coordinate (p4) at (3,0);
\coordinate (p5) at (4,0);
\coordinate (p6) at (5,0);
\coordinate (p7) at (6,0);
\foreach \i in {1,...,7}{
\node at ($(p\i)-(0,-.5)$) {$\i$};
}
\draw[thick] (p1)--(p7);

\draw[thick,red]  (p5) arc (0:-180: 2);
\draw[thick,red]  (p3) arc (0:-180:.5);

\draw[thick,green]  (p4) arc (0:-180:.5); 
\draw[thick,green]  (p7) arc (0:-180:1); 

\draw[thick,blue] (p6) arc (0:-180:.5);

\node at (7,-.7) {$\rightarrow$};

\node at (13,-.7) {$\{\{1,5,6,7\}, \{2,3,4\}\}$};

\end{tikzpicture}

\begin{tikzpicture}[scale = .3]
\node at (-7,-1.2){\begin{ytableau}
1&2\\
3&6\\
4&7\\
5
\end{ytableau}};

\node at (-5,-.7) {$\rightarrow$};

\node at (-3,-.7) { \begin{ytableau}
1&2\\
3&6\\
5&7\\
\end{ytableau}};

\node at (-1,-.7) {$\rightarrow$};
\coordinate (p1) at (0,0);
\coordinate (p2) at (1,0);
\coordinate (p3) at (2,0);
\coordinate (p4) at (3,0);
\coordinate (p5) at (4,0);
\coordinate (p6) at (5,0);
\coordinate (p7) at (6,0);
\foreach \i in {1,...,7}{
\node at ($(p\i)-(0,-.5)$) {$\i$};
}
\draw[thick] (p1)--(p7);

\draw[thick,red]  (p6) arc (0:-180: 2.5);
\draw[thick,red]  (p3) arc (0:-180:.5);

\draw[thick,green]  (p5) arc (0:-180:1); 
\draw[thick,green]  (p7) arc (0:-180:.5); 

\draw[thick,blue] (p4) arc (0:-180:.5);

\node at (7,-.7) {$\rightarrow$};

\node at (13,-.7) {$\{\{1,6,7\}, \{2,3,4,5\}\}$};

\end{tikzpicture}
\begin{tikzpicture}[scale = .3]
\node at (-7,-1.2){\begin{ytableau}
1&3\\
2&4\\
5&6\\
7
\end{ytableau}};

\node at (-5,-.7) {$\rightarrow$};

\node at (-3,-.7) { \begin{ytableau}
1&3\\
2&4\\
6&7\\
\end{ytableau}};

\node at (-1,-.7) {$\rightarrow$};
\coordinate (p1) at (0,0);
\coordinate (p2) at (1,0);
\coordinate (p3) at (2,0);
\coordinate (p4) at (3,0);
\coordinate (p5) at (4,0);
\coordinate (p6) at (5,0);
\coordinate (p7) at (6,0);
\foreach \i in {1,...,7}{
\node at ($(p\i)-(0,-.5)$) {$\i$};
}
\draw[thick] (p1)--(p7);

\draw[thick,red]  (p4) arc (0:-180: .5);
\draw[thick,red]  (p2) arc (0:-180:.5);

\draw[thick,green]  (p6) arc (0:-180:1); 
\draw[thick,green]  (p7) arc (0:-180:2.5); 

\draw[thick,blue] (p5) arc (0:-180:.5);

\node at (7,-.7) {$\rightarrow$};

\node at (13,-.7) {$\{\{1,2,7\}, \{3,4,5,6\}\}$};

\end{tikzpicture}

\begin{tikzpicture}[scale = .3]
\node at (-7,-1.2){\begin{ytableau}
1&3\\
2&4\\
5&7\\
6
\end{ytableau}};

\node at (-5,-.7) {$\rightarrow$};

\node at (-3,-.7) { \begin{ytableau}
1&3\\
2&4\\
5&7\\
\end{ytableau}};

\node at (-1,-.7) {$\rightarrow$};
\coordinate (p1) at (0,0);
\coordinate (p2) at (1,0);
\coordinate (p3) at (2,0);
\coordinate (p4) at (3,0);
\coordinate (p5) at (4,0);
\coordinate (p6) at (5,0);
\coordinate (p7) at (6,0);
\foreach \i in {1,...,7}{
\node at ($(p\i)-(0,-.5)$) {$\i$};
}
\draw[thick] (p1)--(p7);

\draw[thick,red]  (p4) arc (0:-180: .5);
\draw[thick,red]  (p2) arc (0:-180:.5);

\draw[thick,green]  (p5) arc (0:-180:.5); 
\draw[thick,green]  (p7) arc (0:-180:2.5); 

\draw[thick,blue] (p6) arc (0:-180:2);

\node at (7,-.7) {$\rightarrow$};

\node at (13,-.7) {$\{\{1,2,6,7\}, \{3,4,5\}\}$};

\end{tikzpicture}
\begin{tikzpicture}[scale = .3]
\node at (-7,-1.2){\begin{ytableau}
1&3\\
2&5\\
4&6\\
7
\end{ytableau}};

\node at (-5,-.7) {$\rightarrow$};

\node at (-3,-.7) { \begin{ytableau}
1&3\\
2&5\\
6&7\\
\end{ytableau}};

\node at (-1,-.7) {$\rightarrow$};
\coordinate (p1) at (0,0);
\coordinate (p2) at (1,0);
\coordinate (p3) at (2,0);
\coordinate (p4) at (3,0);
\coordinate (p5) at (4,0);
\coordinate (p6) at (5,0);
\coordinate (p7) at (6,0);
\foreach \i in {1,...,7}{
\node at ($(p\i)-(0,-.5)$) {$\i$};
}
\draw[thick] (p1)--(p7);

\draw[thick,red]  (p5) arc (0:-180: 1);
\draw[thick,red]  (p2) arc (0:-180:.5);

\draw[thick,green]  (p6) arc (0:-180:.5); 
\draw[thick,green]  (p7) arc (0:-180:2.5); 

\draw[thick,blue] (p4) arc (0:-180:1);

\node at (7,-.7) {$\rightarrow$};

\node at (13,-.7) {$\{\{1,2,4,7\}, \{3,5,6\}\}$};

\end{tikzpicture}

\begin{tikzpicture}[scale = .3]
\node at (-7,-1.2){\begin{ytableau}
1&3\\
2&5\\
4&7\\
6
\end{ytableau}};

\node at (-5,-.7) {$\rightarrow$};

\node at (-3,-.7) { \begin{ytableau}
1&3\\
2&5\\
4&7\\
\end{ytableau}};

\node at (-1,-.7) {$\rightarrow$};
\coordinate (p1) at (0,0);
\coordinate (p2) at (1,0);
\coordinate (p3) at (2,0);
\coordinate (p4) at (3,0);
\coordinate (p5) at (4,0);
\coordinate (p6) at (5,0);
\coordinate (p7) at (6,0);
\foreach \i in {1,...,7}{
\node at ($(p\i)-(0,-.5)$) {$\i$};
}
\draw[thick] (p1)--(p7);

\draw[thick,red]  (p5) arc (0:-180: 1);
\draw[thick,red]  (p2) arc (0:-180:.5);

\draw[thick,green]  (p4) arc (0:-180:1); 
\draw[thick,green]  (p7) arc (0:-180:1); 

\draw[thick,blue] (p6) arc (0:-180:.5);

\node at (7,-.7) {$\rightarrow$};

\node at (13,-.7) {$\{\{1,2,4\}, \{3,5,6,7\}\}$};

\end{tikzpicture}
\begin{tikzpicture}[scale = .3]
\node at (-7,-1.2){\begin{ytableau}
1&3\\
2&6\\
4&7\\
5
\end{ytableau}};

\node at (-5,-.7) {$\rightarrow$};

\node at (-3,-.7) { \begin{ytableau}
1&3\\
2&6\\
5&7\\
\end{ytableau}};

\node at (-1,-.7) {$\rightarrow$};
\coordinate (p1) at (0,0);
\coordinate (p2) at (1,0);
\coordinate (p3) at (2,0);
\coordinate (p4) at (3,0);
\coordinate (p5) at (4,0);
\coordinate (p6) at (5,0);
\coordinate (p7) at (6,0);
\foreach \i in {1,...,7}{
\node at ($(p\i)-(0,-.5)$) {$\i$};
}
\draw[thick] (p1)--(p7);

\draw[thick,red]  (p6) arc (0:-180: 1.5);
\draw[thick,red]  (p2) arc (0:-180:.5);

\draw[thick,green]  (p5) arc (0:-180:1.5); 
\draw[thick,green]  (p7) arc (0:-180:.5); 

\draw[thick,blue] (p4) arc (0:-180:1);

\node at (7,-.7) {$\rightarrow$};

\node at (13,-.7) {$\{\{1,2,4,5\}, \{3,6,7\}\}$};

\end{tikzpicture}

\begin{tikzpicture}[scale = .3]
\node at (-7,-1.2){\begin{ytableau}
1&4\\
2&5\\
3&6\\
7
\end{ytableau}};

\node at (-5,-.7) {$\rightarrow$};

\node at (-3,-.7) { \begin{ytableau}
1&4\\
2&5\\
6&7\\
\end{ytableau}};

\node at (-1,-.7) {$\rightarrow$};
\coordinate (p1) at (0,0);
\coordinate (p2) at (1,0);
\coordinate (p3) at (2,0);
\coordinate (p4) at (3,0);
\coordinate (p5) at (4,0);
\coordinate (p6) at (5,0);
\coordinate (p7) at (6,0);
\foreach \i in {1,...,7}{
\node at ($(p\i)-(0,-.5)$) {$\i$};
}
\draw[thick] (p1)--(p7);

\draw[thick,red]  (p5) arc (0:-180: .5);
\draw[thick,red]  (p2) arc (0:-180:.5);

\draw[thick,green]  (p6) arc (0:-180:.5); 
\draw[thick,green]  (p7) arc (0:-180:2.5); 

\draw[thick,blue] (p3) arc (0:-180:.5);

\node at (7,-.7) {$\rightarrow$};

\node at (13,-.7) {$\{\{1,2,3,7\}, \{4,5,6\}\}$};

\end{tikzpicture}
\begin{tikzpicture}[scale = .3]
\node at (-7,-1.2){\begin{ytableau}
1&4\\
2&5\\
3&7\\
6
\end{ytableau}};

\node at (-5,-.7) {$\rightarrow$};

\node at (-3,-.7) { \begin{ytableau}
1&4\\
2&5\\
3&7\\
\end{ytableau}};

\node at (-1,-.7) {$\rightarrow$};
\coordinate (p1) at (0,0);
\coordinate (p2) at (1,0);
\coordinate (p3) at (2,0);
\coordinate (p4) at (3,0);
\coordinate (p5) at (4,0);
\coordinate (p6) at (5,0);
\coordinate (p7) at (6,0);
\foreach \i in {1,...,7}{
\node at ($(p\i)-(0,-.5)$) {$\i$};
}
\draw[thick] (p1)--(p7);

\draw[thick,red]  (p5) arc (0:-180: .5);
\draw[thick,red]  (p2) arc (0:-180:.5);

\draw[thick,green]  (p3) arc (0:-180:.5); 
\draw[thick,green]  (p7) arc (0:-180:1); 

\draw[thick,blue] (p6) arc (0:-180:.5);

\node at (7,-.7) {$\rightarrow$};

\node at (13,-.7) {$\{\{1,2,3\}, \{4,5,6,7\}\}$};

\end{tikzpicture}

\begin{tikzpicture}[scale = .3]
\node at (-7,-1.2){\begin{ytableau}
1&4\\
2&6\\
3&7\\
5
\end{ytableau}};

\node at (-5,-.7) {$\rightarrow$};

\node at (-3,-.7) { \begin{ytableau}
1&4\\
2&6\\
5&7\\
\end{ytableau}};

\node at (-1,-.7) {$\rightarrow$};
\coordinate (p1) at (0,0);
\coordinate (p2) at (1,0);
\coordinate (p3) at (2,0);
\coordinate (p4) at (3,0);
\coordinate (p5) at (4,0);
\coordinate (p6) at (5,0);
\coordinate (p7) at (6,0);
\foreach \i in {1,...,7}{
\node at ($(p\i)-(0,-.5)$) {$\i$};
}
\draw[thick] (p1)--(p7);

\draw[thick,red]  (p6) arc (0:-180: 1);
\draw[thick,red]  (p2) arc (0:-180:.5);

\draw[thick,green]  (p5) arc (0:-180:1.5); 
\draw[thick,green]  (p7) arc (0:-180:.5); 

\draw[thick,blue] (p3) arc (0:-180:.5);

\node at (7,-.7) {$\rightarrow$};

\node at (13,-.7) {$\{\{1,2,3,5\}, \{4,6,7\}\}$};

\end{tikzpicture}
\begin{tikzpicture}[scale = .3]
\node at (-7,-1.2){\begin{ytableau}
1&5\\
2&6\\
3&7\\
4
\end{ytableau}};

\node at (-5,-.7) {$\rightarrow$};

\node at (-3,-.7) { \begin{ytableau}
1&5\\
2&6\\
4&7\\
\end{ytableau}};

\node at (-1,-.7) {$\rightarrow$};
\coordinate (p1) at (0,0);
\coordinate (p2) at (1,0);
\coordinate (p3) at (2,0);
\coordinate (p4) at (3,0);
\coordinate (p5) at (4,0);
\coordinate (p6) at (5,0);
\coordinate (p7) at (6,0);
\foreach \i in {1,...,7}{
\node at ($(p\i)-(0,-.5)$) {$\i$};
}
\draw[thick] (p1)--(p7);

\draw[thick,red]  (p6) arc (0:-180: .5);
\draw[thick,red]  (p2) arc (0:-180:.5);

\draw[thick,green]  (p4) arc (0:-180:1); 
\draw[thick,green]  (p7) arc (0:-180:.5); 

\draw[thick,blue] (p3) arc (0:-180:.5);

\node at (7,-.7) {$\rightarrow$};

\node at (13,-.7) {$\{\{1,2,3,4\}, \{5,6,7\}\}$};

\end{tikzpicture}

\end{center}

\end{example}

We can now define and prove our basis.

\begin{theorem}
\label{basis}
Let $r\geq 2$. Order each weakly noncrossing set partition $\gamma \in WNC(n,d,r)$ to create an ordered set partition $\pi_\gamma$. Then the set  $\{[\pi_\gamma]_r \mid \gamma \in WNC(n,d,r)\}$ is a basis for the flamingo Specht module $S^{(d^r, 1^{n-rd})}$.
\end{theorem}

\begin{proof}
By Proposition~\ref{tableauxbiject} and Theorem~\ref{inspecht}, it suffices to show that $\{[\pi_\gamma]_r \mid \gamma \in WNC(n,d,r)\}$ is linearly independent. To do so, we introduce a monomial order and show that under this order, each $[\pi_\gamma]_r$ has a unique leading term.

Recall that the $r$-jellyfish invariant is a polynomial in the $\nu \times n$ variables:
\begin{center}$
\begin{bmatrix}
x_{1,1} & x_{1,2} & \cdots & x_{1,n}\\
x_{2,1} & x_{2,2} & \cdots & x_{2,n}\\
\vdots &  &\ddots& \vdots\\
x_{\nu, 1} &x_{\nu,2} & \cdots &x_{\nu, n}
\end{bmatrix}$
\end{center}
We order these variables in a somewhat unusual way. Define an order on these variables by $x_{i_1, j_1} < x_{i_2, j_2}$ if and only if one of the following holds:
\begin{enumerate}
\item $i_1 < i_2$
\item $i_1 = i_2 \neq r \textrm{ and } j_1 < j_2$
\item $i_1 = i_2 = r \textrm{ and } j_1 > j_2$
\end{enumerate}
In other words, we order them in reading order except we read the $r^{th}$ row backwards. We then take the lexicogrpahic monomial order with respect to this ordering of variables. The unusual ordering is chosen to make use of Corollary~\ref{setsunique}. Without reversing the $r^{th}$ row, lexicogrpahic leading terms are not unique.

 Let the $i^{th}$ block of $\gamma$ be 
\[ \gamma_i := \{\gamma_{i,1}, \gamma_{i,2}, \dots, \gamma_{i,|\gamma_i|}\}
\]
Let $T$ be a jellyfish tableau associated to $\gamma$. Then the leading term of $J(T)$ is straightforward to compute from the definition, we have
\[
\textrm{lt}(J(t)) = \left(\prod_{i=1}^d \prod_{j=1}^{r-1} x_{j,\gamma_{i,j}}\right) \prod_{i=1}^d x_{r, \gamma_{i, |\gamma_i|}} \prod_{j=r+1}^{\nu} x_{j, u_j}
\]
where $u_j$ is the entry appearing in row $j$ of $J(T)$. In words, for $i=1, \dots, r-1$,  $x_{i,j}$ will appear if and only if $j$ is the $i^{th}$ smallest element of some block of $\gamma$, and $x_{r,j}$ will appear if and only if $j$ is the largest element of some block of $\gamma$. The leading term of $[\pi_\gamma]_r$ will be the leading term of one of these $J(T)$, and we can see that the leading term contains the information of the sets described in Corllary~\ref{setsunique}. Thus, the leading term of $[\pi_\gamma]_r$ is unique and thus $\{[\pi_\gamma]_r \mid \gamma \in WNC(n,d,r)\}$ is linearly independent as desired.
\end{proof}

\begin{remark}
The property of being $r$-weakly noncrossing is not preserved under rotation, for example $\{1,4,5\}, \{2,3,6\}$ is weakly 3-noncrossing, but $\{\{1,3,4\}, \{2,5,6\}\}$ is not. So the basis given in Theorem~\ref{basis} is not rotation invariant as desired of a web basis. The next section will give a different basis which is rotation invariant in the case $r=3$.
\end{remark}

\section{Rotation invariant webs for the $r=3$ case}

For the rest of the paper, we specialize to the case $r=3$. We will introduce a new basis for this case which is rotation and reflection invariant. To index our basis, we introduce a subset normal plabic graphs which we call {\em augmented webs}. We call them augmentd webs to allude to the fact that they are similar to $SL_3$ webs, but potentially with vertices of higher degree. We will show that augmented webs are in bijection with 3-weakly noncrossing set partitions, and thus have the correct enumeration to index a basis of $S^{(d^3, 1^{n-3d})}$. The benefit of working with augmented webs over 3-weakly noncrossing set paritions is that the set of augmented webs is rotation invariant.

\begin{definition}
An {\em augmented web} is a normal plabic graph which contains no faces of degree less than 6 and no black vertices of degree less than 3. The {\em exceedance} of an augmented web is the number of black vertices minus the number of white vertices. Let $AW(n,d)$ denote the set of all augmented webs with $n$ boundary vertices and exceedance $d$.
\end{definition}

\begin{remark}
The term {\em exceedance} is chosen because the exceedance of an augmented web is also the number of exceedances in the trip permutation.
\end{remark}

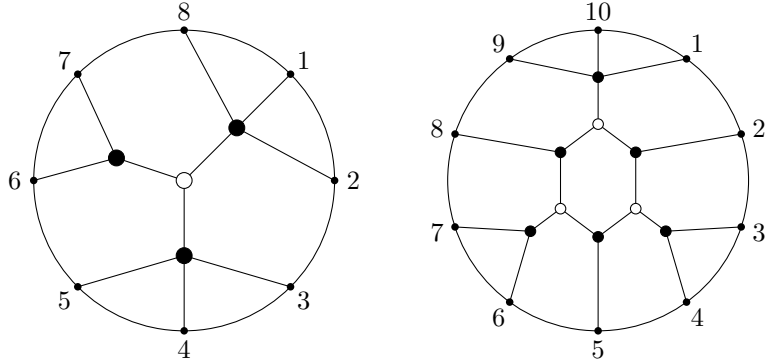
\begin{figure}[h]
\begin{center}
\begin{tikzpicture}
\equicc[2cm]{8}{0}{0};
\draw (N2)--(.7,.7)--(0,0)--(0,-1)--(N3);
\draw (N1)--(.7,.7)--(N8);
\draw (N4)--(0,-1)--(N5);
\draw (N7)--(-.9,.3)--(0,0);
\draw (N6)--(-.9,.3);
\draw[fill=white] (0,0) circle (3pt);
\draw[fill =black] (.7,.7) circle (3pt);
\draw[fill =black] (0,-1) circle (3pt);
\draw[fill =black] (-.9,.3) circle (3pt);
\end{tikzpicture}\qquad
\begin{tikzpicture}
\equicc[2cm]{10}{0}{0};
\coordinate (N11) at (0,1.5/2);
\coordinate (N12) at (1/2,.75/2);
\coordinate (N13) at (1/2,-.75/2);
\coordinate (N14) at (0,-1.5/2);
\coordinate (N15) at (-1/2,-.75/2);
\coordinate (N16) at (-1/2,.75/2);
\coordinate(N17) at (0,2.75/2);
\coordinate (N18) at (1.8/2,-1.35/2);
\coordinate (N19) at (-1.8/2,-1.35/2);
\draw (N6)--(N19)--(N7);
\draw (N5)--(N14)--(N15)--(N16)--(N11)--(N12)--(N13)--(N14);
\draw (N8)--(N16);
\draw (N9)--(N17)--(N1);
\draw (N10)--(N17)--(N11);
\draw (N2)--(N12);
\draw (N19)--(N15);
\draw (N3)--(N18)--(N4);
\draw (N18)--(N13);
\draw[fill=black]  (N17) circle (2pt);
\draw[fill=black]  (N18) circle (2pt);
\draw[fill=black]  (N19) circle (2pt);
\draw[fill=white]  (N11) circle (2pt);
\draw[fill=black]  (N12) circle (2pt);
\draw[fill=white]  (N13) circle (2pt);
\draw[fill=black]  (N14) circle (2pt);
\draw[fill=white]  (N15) circle (2pt);
\draw[fill=black]  (N16) circle (2pt);
\end{tikzpicture}
\end{center}
\caption{Examples of augmented in $A(8,2)$ and $A(10,3)$. }
\end{figure}

\begin{remark}
When an augmented web has no white vertices, it contains exactly the same information as a strongly noncrossing set partition, with the sets of all boundary vertices connected to a particular interior vertex forming the blocks.
\end{remark}

\subsection{Combinatorial properties of augmented webs}

In this subsection, we develop combinatorial results for augmented webs. 
Our first result is that all augmented webs are reduced plabic graphs.

\begin{proposition}
\label{reduced}
Let $W \in AW(n,d)$. Then $W$ is reduced. 
\end{proposition}

\begin{proof}
As no square faces or vertices of degree two are present, normal plabic graph moves are not possible. Thus it suffices to check that $W$ itself has no forbidden configurations and this is clear.
\end{proof}

Next, we show that augmented webs have an inductive structure we can exploit. 
\begin{lemma}
\label{induct}
Let $W \in AW(n,d)$, and suppose $W$ has at least one white vertex. Then for each connected component of $W$, there exists at least two black vertices each connected to exactly one white vertex. 
\end{lemma}

\begin{proof}
Let $C$ be a maximal cycle in $W$, i.e. a cycle with no edge incident to an interior face of $W$. Since $W$ is reduced by Proposition~\ref{reduced}, it contains no round trips or essential self-intersections. Therefore there are at least two black vertices of $W$ exterior to $C$ which connect to a white vertex in $C$. Create a graph $G$ with two types of vertices: a vertex for every vertex of $W$ which is on the exterior of every cycle in $W$, and a vertex for every maximal cycle. Add an edge between a vertex $v$ and a maximal cycle $C$ whenever $v$ is adjacent to a vertex in $C$. Then $G$ is a forest, and since $W$ had at least one white vertex, $G$ has at least one edge. So $G$ has two leaves, and these two leaves must be black vertices connected to exactly one white vertex.
\end{proof}

The use of Lemma~\ref{induct} is that every augmented web with at least one white vertex can be built out of an augmented web with one fewer white vertices in the following way. Let $u$ and $v$ be the black vertex and its white neighbor identified by Lemma~\ref{induct}. If we remove vertices $u$ and $v$, then connect the neighbors of $v$ to the boundary by at least one edge each in a planar way, we get an augmented web $W'$.

We will need a notion of depth of a face, which we now define. 
\begin{definition}
Let $W \in AW(n,d)$. Let $f$ be a face of $W$, and let $f_0$ be the face connected to the section of boundary between 1 and $n$. The {\em depth} of $f$ is the number of exceedances which separate $f$ from $f_0$. 

Let $e$ be an edge of $W$.  We say that $e$ is a {\em depth boundary edge} if the depth of the faces incident to $e$ are not equal. Equivalently, $e$ is a depth boundary edge if exactly one of the trips passing through $e$ is an exceedance. We say $e$ is a {\em left-to-right} depth boundary edge if, when oriented towads its black vertex endpoint, the depth of the face on the right is higher than the depth of the face on the left. Equivalently, $e$ is a left-to-right depth boundary if only the trip passing through $e$ from towards its black vertex is an exceedance. We define {\em right-to-left} deth boundary edges similarly.
\end{definition}

\begin{lemma}
\label{depthboundaryvertex}
Let $W \in AW(n,d)$. Let $v$ be an interior vertex of $W$. Then exactly two edges incident to $v$ are depth boundary edges.
\end{lemma}

\begin{proof}
Consider the set of all trips $t_1, \dots, t_k$ passing through $v$, ordered cyclically. Since $W$ is reduced, the starts of all these trips must appear in the same cyclic order around the boundary of $W$, since otherwise we would introduce a bad double corssing. Similarly, the ends of these trips appear in the same cyclic order around $W$. Therefore, the set of exceedances passing through $v$ is a proper nonempty subset of these trips which is cyclically consecutive around $v$. The first and last of these trips will contribute a depth boundary edge.
\end{proof}

\begin{lemma}
\label{nokiss}
Let $W \in AW(n,d)$. Let $u$ and $v$ be two adjacent vertices of $W$. Let $t_u$ be any trip which passes through $u$ but does not use edge $(u,v)$, and let $t_v$ be any trip which passes through $v$ but does not use edge $(u,v)$. Then trips $t_u$ and $t_v$ do not share any vertices.
\end{lemma}

\begin{proof}
This follows from Euler's formula for planar graphs. Assume the contrary, that trips $t_u$ and $t_v$ meet at some vertex $x$. Let $C$ be the cycle formed from $t_u$, $t_v$ and edge $(u,v)$, and suppose it is of length $k$. Consider the graph $G$ containing all vertices and edges of $W$ that are part of $C$ or in its interior. Let $V_{\textrm{int}}$ denote the number of vertices strictly in the interior of $C$, and let $\alpha$ be the average degree of these interior vertices. Then we have
\[
|E(G)| \geq \frac{5}{4}k + \frac{\alpha}{2}V_{\textrm{int}}-1
\]
and thus by Euler's formula the number of faces of $G$ not including the external face is at least $\frac{1}{4}k + \frac{\alpha-2}{2}V_{\textrm{int}}$. The total degree of these faces is 
\[
\frac{3}{2}k-2 + \frac{\alpha}V_{\textrm{int}} 
\]
and thus their average degree is strictly less than 6, a contradiction.
\end{proof}

\subsection{A bijection from tableaux to augmented webs}

We can now show that augmented webs are in bijection with weakly 3-noncrossing set partitions, $WNC(n,d,3)$. To define our bijection, we first formally define the arc diagram used in the proof of Proposition~\ref{tableauxbiject}. We call these $m$-diagrams, based on the objects of the same name developed by J. Tymoczko in \cite{Tymoczko}. 

\begin{definition}
Let $\pi \in WNC(n,d,3)$. To form the {\em $m$-diagram} for $\pi$, place the vertices 1 through $n$ equally spaced in a line, then for each block $B = \{b_1 < b_2 < \cdots < b_k\}$, draw a semicircular arc in the lower half plane from $b_1$ to $b_2$, and from $b_2$ to all other elements of $B$. We call the arc between $b_1$ and $b_2$ a first arc, and all other arcs second arcs. Note that the definition of weakly $3$-noncrossing guarantees that first arcs do not cross first arcs, and second arcs do not cross second arcs. For visual clarity, we will often color first arcs in red and second arcs in black. The name $m$-diagram is due to the fact that in Tymoczko's definition, blocks were always size three and thus had a unique second arc, so the diagram appeared visually as a number of intersecting $m$ shapes. 

The collection of first arcs and maximal second arcs of each block divide the lower half plane into a number of regions. We define the {\em depth} of each region to be the number of first arcs and maximal second arcs which the region lies above.
\end{definition}

\begin{example}
Let $\pi \in WNC(13,3,3)$ be the weakly 3-noncrossing set partition with three blocks, $\{\{1,4,6,7,8\}, \{2,3,9,10\}, \{5, 11,12,13\}\}$. The $m$-diagram associated to $\pi$ appears below.
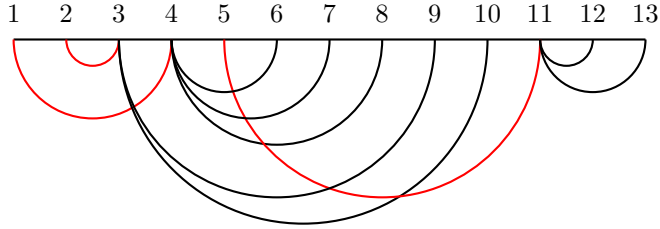
\begin{center}
\begin{tikzpicture}[scale = .7]
\coordinate (p1) at (0,0);
\coordinate (p2) at (1,0);
\coordinate (p3) at (2,0);
\coordinate (p4) at (3,0);
\coordinate (p5) at (4,0);
\coordinate (p6) at (5,0);
\coordinate (p7) at (6,0);
\coordinate (p8) at (7,0);
\coordinate (p9) at (8,0);
\coordinate (p10) at (9,0);
\coordinate (p11) at (10,0);
\coordinate (p12) at (11,0);
\coordinate (p13) at (12,0);

\foreach \i in {1,...,13}{
\node at ($(p\i)+(0,.5)$) {$\i$};
}
\draw[thick] (p1)--(p13);
\draw[thick,red]  (p4) arc (0:-180: 1.5);
\draw[thick]  (p8) arc (0:-180:2);
\draw[thick,red]  (p3) arc (0:-180:.5); 
\draw[thick]  (p10) arc (0:-180:3.5); 
\draw[thick,red]  (p11) arc (0:-180:3); 
\draw[thick]  (p13) arc (0:-180:1); 
\draw[thick]  (p12) arc (0:-180:.5); 
\draw[thick]  (p9) arc (0:-180:3); 
\draw[thick]  (p6) arc (0:-180:1); 
\draw[thick]  (p7) arc (0:-180:1.5); 
\end{tikzpicture}
\end{center}
\end{example}

We can now define our bijection.

\begin{definition}
\label{bijectdef}
The function $\varphi:WNC(n,d,3) \rightarrow AW(n,d)$ is defined as follows:

Let $\pi \in WNC(n,d,3)$, and let $M$ be its $m$ diagram. For each block $B = \{b_1<b_2< \cdots <b_k\}$, introduce a black vertex $v_B$ slightly above $b_2$, connected to $b_2$ by an edge. In a small region around $b_2$, modify the arcs connecting to $b_2$ to instead connect to $v_B$. Then, for every pair of blocks, if the first arc of one crosses some of the second arcs of the other, replace a small region containg all intersections as shown in figure~\ref{replacementfig}.
\end{definition}

\begin{figure}[h]
\begin{center}

\begin{tikzpicture}[scale = .6]

\draw[thick,red] (0,2)--(2,0);
\draw[thick,black] (3, 2)--(1,0.1)--(3,3.42);

\draw[fill=black]  (1,0) circle (3pt);
\draw[fill=white]  (3,2) circle (3pt);
\draw[fill=white]  (3,3.42) circle (3pt);
\draw[fill=white]  (0,2) circle (3pt);
\draw[fill=black]  (2,0) circle (3pt);

\node at (3.5,.75) {$\rightarrow$};

\draw[thick, red] (4,2)--(5.5,1.6);
\draw[thick, black] (5,0)--(5.5,.8);
\draw[thick, blue] (5.5,.8)--(5.5,1.6);
\draw[thick, black] (7,2)--(5.5,1.6)--(7,3.42);
\draw[thick,red] (5.5,.8)--(6,0);

\draw[fill=white]  (5.5,.8) circle (3pt);
\draw[fill=black]  (5.5,1.6) circle (3pt);
\draw[fill=black]  (5,0) circle (3pt);
\draw[fill=white]  (7,2) circle (3pt);
\draw[fill=white]  (7,3.42) circle (3pt);
\draw[fill=white]  (4,2) circle (3pt);
\draw[fill=black]  (6,0) circle (3pt);
\end{tikzpicture}
\end{center}
\caption{The replacement operation used in the definition of $\varphi$. The first arc is depicted in red, and the second arcs are depicted in black.}
\label{replacementfig}
\end{figure}
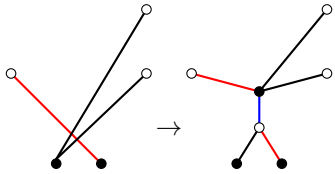

\begin{example}
Continuing our prior example of $\pi = \{1,4,6,7,8\}, \{2,3,9,10\}, \{5, 11,12,13\}$, the resulting web is depicted below.

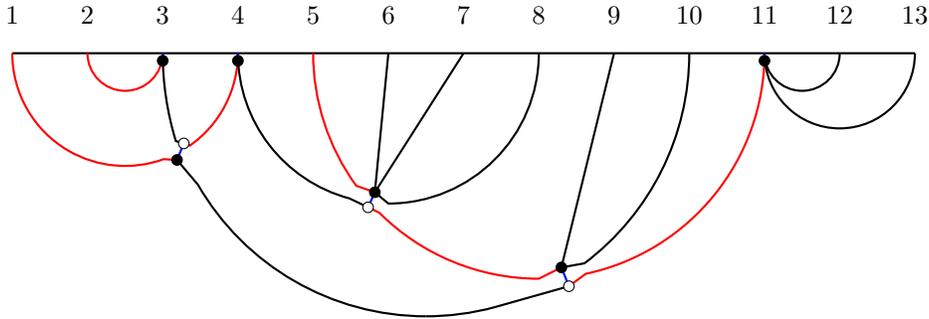
\begin{center}

\begin{tikzpicture}
\coordinate (p1) at (0,0);
\coordinate (p2) at (1,0);
\coordinate (p3) at (2,0);
\coordinate (p4) at (3,0);
\coordinate (p5) at (4,0);
\coordinate (p6) at (5,0);
\coordinate (p7) at (6,0);
\coordinate (p8) at (7,0);
\coordinate (p9) at (8,0);
\coordinate (p10) at (9,0);
\coordinate (p11) at (10,0);
\coordinate (p12) at (11,0);
\coordinate (p13) at (12,0);

\foreach \i in {1,...,13}{
\node at ($(p\i)+(0,.5)$) {$\i$};
}
\draw[thick] (p1)--(p13);

\draw[thick,red]  (p1) arc (-180:-70: 1.5)--(2.19,-1.42);
\draw[thick,blue] (2.19,-1.42)--(2.28,-1.2);
\draw[thick,black] (p3) arc (-180:-163:4)--(2.28,-1.2);
\draw[thick,red] (p4) arc (0:-55:1.5)--(2.28,-1.2);

\draw[thick,red] (p2) arc (-180:0:.5);

\draw[thick,red] (p5) arc (-180:-144:3)--(4.82,-1.85);
\draw[thick,black] (p8) arc (0:-90:2)--(4.82,-1.85);
\draw[thick,black] (p4) arc (-180:-105:2)--(4.73,-2.05);
\draw[thick,blue] (4.73,-2.05)--(4.82,-1.85);

\draw[thick,blue] (7.4,-3.1)--(7.3,-2.85);
\draw[thick, black] (p10) arc (0:-53:3.5)--(7.3,-2.85);
\draw[thick,red] (p11) arc (0:-78:3)--(7.4,-3.1);
\draw[thick,red] (7.3,-2.85)--(7,-3) arc (-90:-135:3)--(4.73,-2.05);
\draw[thick,black] (5.5,-3.5) arc (-90:-150:3.5)--(2.19,-1.42);
\draw[thick,black] (5.5,-3.5) arc (-90:-75:3.5)--(7.4,-3.1);

\draw[thick,black] (p6)--(4.82,-1.85)--(p7);
\draw[thick,black] (p9)--(7.3,-2.85);
\draw[thick,black] (p11) arc (-180:0:.5);
\draw[thick,black] (p11) arc (-180:0:1);
\draw[fill = black] (2.19,-1.42) circle (2pt);
\draw[fill = white] (2.28,-1.2) circle (2pt);

\draw[fill = white] (4.73,-2.05) circle (2pt);
\draw[fill = black] (4.82,-1.85) circle (2pt);

\draw[fill = white] (7.4,-3.1) circle (2pt);
\draw[fill = black] (7.3,-2.85) circle (2pt);

\draw[thick,blue] (2,0)--(2,-.1);
\draw[fill = black] (2,-.1) circle (2pt);

\draw[thick,blue] (3,0)--(3,-.1);
\draw[fill = black] (3,-.1) circle (2pt);

\draw[thick,blue] (10,0)--(10,-.1);
\draw[fill = black] (10,-.1) circle (2pt);

\end{tikzpicture}
\end{center}
\end{example}

\begin{proposition}
The function $\varphi$ is well defined, i.e. if $\pi \in WNC(n,d,3)$, then $\varphi(\pi)$ is indeed in $AW(n,d)$ 
\end{proposition}

\begin{proof}
We need to check that the resulting graph does not have a cycle of length 4, the other properties are clear. A 4 cycle would necessarily have two edges coming from first arcs and two edge coming from second arcs, such that no new edges are created. Orienting all edges in the $m$-diagram away from $v_B$ for each block $B$, a 4 cycle would require that the arcs intersect with opposite orientations at each corner. But the the two first arcs would have to have different orientations, and this is not possible.
\end{proof}

\begin{lemma}
\label{depthcor}
Let $\pi \in WNC(n,d,3)$. The first arcs and maximal second arcs of the $m$ diagram for $\pi$ divide the half plane into a number of regions. There is a depth preserving correspondence between faces of $\varphi(\pi)$ and these regions. 
\end{lemma}

\begin{proof}
Each replacement can be made so that edges coming from first arcs and maximal second arcs stay the same except in a $\epsilon$ radius region around each intersection. Each face of $\varphi(\pi)$ is thus contained (except for an $\epsilon$-small portion) in a unique region. The trip starting at the first arc of each $m$ will be an exceedance of $\varphi(\pi)$, which travels left to right along first arcs, maximal second arcs, and edges introduced by intersection replacement steps. These trips either cross at each intersection, using the new edge introduced at that intersection twice, or turn at from eachother at each intersection, using the new edge introduced at that intersection 0 times. Thus, the depth boundary paths consist exactly of those edges which come from first arcs and maximal second arcs, and thus the depth of each face matches the depth of the region it is contained in.
\end{proof}

\begin{figure}[h]
\begin{center}
\begin{tikzpicture}
\coordinate (p1) at (0,0);
\coordinate (p2) at (1,0);
\coordinate (p3) at (2,0);
\coordinate (p4) at (3,0);
\coordinate (p5) at (4,0);
\coordinate (p6) at (5,0);
\coordinate (p7) at (6,0);
\coordinate (p8) at (7,0);
\coordinate (p9) at (8,0);
\coordinate (p10) at (9,0);
\coordinate (p11) at (10,0);
\coordinate (p12) at (11,0);
\coordinate (p13) at (12,0);

\foreach \i in {1,...,13}{
\node at ($(p\i)+(0,.5)$) {$\i$};
}

\fill[gray!30] (p4) arc (0:-180: 1.5);
\fill[gray!30] (p3) arc (0:-180:.5);
\fill[gray!30] (p11) arc (0:-180:3);
\fill[gray!30] (p4) arc (0:-180: 1.5);
\fill[gray!30] (p13) arc (0:-180:1);
\fill[gray!30] (p3) arc (-180:0:3.5);
\fill[gray!60] (p2) arc (-180:0:.5);
\fill[gray!60] (p3) arc (-180:-158:3.5)--(p4);
\fill[gray!60] (p4) arc (0:-60:1.5)--(p3);
\fill[gray!60] (p4) arc (-180:0:2);
\fill[gray!60] (p5) arc (-180:-87:3)--(p10);
\fill[gray!60] (p10) arc (0:-60:3.5)--(p5);
\fill[gray!90] (p5) arc (-180: -138: 3)--(p8);
\fill[gray!90] (p8) arc (0:-97:2)--(p5);

\draw[thick] (p1)--(p13);
\draw[red,thick]  (p4) arc (0:-180: 1.5);
\draw[thick]  (p8) arc (0:-180:2);
\draw[red,thick]  (p3) arc (0:-180:.5); 
\draw[thick]  (p10) arc (0:-180:3.5); 
\draw[red,thick]  (p11) arc (0:-180:3); 
\draw[thick]  (p13) arc (0:-180:1);

\draw[thick, fill = gray!30] (10,-4)--(10.3,-4)--(10.3,-3.7)--(10,-3.7)--(10,-4);
\draw[thick, fill = gray!60] (10,-3)--(10.3,-3)--(10.3,-2.7)--(10,-2.7)--(10,-3);
\draw[thick, fill = gray!90] (10,-2)--(10.3,-2)--(10.3,-1.7)--(10,-1.7)--(10,-2);
\node at (11.3,-3.8) {Depth 1};
\node at (11.3,-2.8) {Depth 2};
\node at (11.3,-1.8) {Depth 3};
\end{tikzpicture}

\begin{tikzpicture}
\coordinate (p1) at (0,0);
\coordinate (p2) at (1,0);
\coordinate (p3) at (2,0);
\coordinate (p4) at (3,0);
\coordinate (p5) at (4,0);
\coordinate (p6) at (5,0);
\coordinate (p7) at (6,0);
\coordinate (p8) at (7,0);
\coordinate (p9) at (8,0);
\coordinate (p10) at (9,0);
\coordinate (p11) at (10,0);
\coordinate (p12) at (11,0);
\coordinate (p13) at (12,0);

\foreach \i in {1,...,13}{
\node at ($(p\i)+(0,.5)$) {$\i$};
}
\draw[thick] (p1)--(p13);

\fill[gray!30] (p1) arc (-180:-70: 1.5)--(2.19,-1.42) -- (p1)--(p11) arc (0:-78:3)--(7.4,-3.1)--(5.5,-3.5) arc (-90:-75:3.5)--(7.4,-3.1)--(5.5,-3.5) arc (-90:-150:3.5)--(2.19,-1.42);
\fill[gray!30] (p11) arc (-180:0:1);
\fill[gray!60] (p2) arc (-180:0:.5);
\fill[gray!60] (p3) arc (-180:-163:4)--(2.28,-1.2)--(p3)--(p4) arc (0:-55:1.5)--(2.28,-1.2);
\fill[gray!60] (p4) arc (-180:-105:2)--(4.73,-2.05) -- (p4)--(p10) arc (0:-53:3.5)--(7.3,-2.85)--(7,-3) arc (-90:-135:3)--(4.73,-2.05);
\fill[gray!90] (p5) arc (-180:-144:3)--(4.82,-1.85)--(p5)--(p8) arc (0:-90:2)--(4.82,-1.85);

\draw[thick,red]  (p1) arc (-180:-70: 1.5)--(2.19,-1.42);
\draw[thick,blue] (2.19,-1.42)--(2.28,-1.2);
\draw[thick,black] (p3) arc (-180:-163:4)--(2.28,-1.2);
\draw[thick,red] (p4) arc (0:-55:1.5)--(2.28,-1.2);

\draw[thick,red] (p2) arc (-180:0:.5);

\draw[thick,red] (p5) arc (-180:-144:3)--(4.82,-1.85);
\draw[thick,black] (p8) arc (0:-90:2)--(4.82,-1.85);
\draw[thick,black] (p4) arc (-180:-105:2)--(4.73,-2.05);
\draw[thick,blue] (4.73,-2.05)--(4.82,-1.85);

\draw[thick,blue] (7.4,-3.1)--(7.3,-2.85);
\draw[thick, black] (p10) arc (0:-53:3.5)--(7.3,-2.85);
\draw[thick,red] (p11) arc (0:-78:3)--(7.4,-3.1);
\draw[thick,red] (7.3,-2.85)--(7,-3) arc (-90:-135:3)--(4.73,-2.05);
\draw[thick,black] (5.5,-3.5) arc (-90:-150:3.5)--(2.19,-1.42);
\draw[thick,black] (5.5,-3.5) arc (-90:-75:3.5)--(7.4,-3.1);

\draw[thick,black] (p6)--(4.82,-1.85)--(p7);
\draw[thick,black] (p9)--(7.3,-2.85);
\draw[thick,black] (p11) arc (-180:0:.5);
\draw[thick,black] (p11) arc (-180:0:1);
\draw[fill = black] (2.19,-1.42) circle (2pt);
\draw[fill = white] (2.28,-1.2) circle (2pt);

\draw[fill = white] (4.73,-2.05) circle (2pt);
\draw[fill = black] (4.82,-1.85) circle (2pt);

\draw[fill = white] (7.4,-3.1) circle (2pt);
\draw[fill = black] (7.3,-2.85) circle (2pt);

\draw[thick,blue] (2,0)--(2,-.1);
\draw[fill = black] (2,-.1) circle (2pt);

\draw[thick,blue] (3,0)--(3,-.1);
\draw[fill = black] (3,-.1) circle (2pt);

\draw[thick,blue] (10,0)--(10,-.1);
\draw[fill = black] (10,-.1) circle (2pt);

\draw[thick, fill = gray!30] (10,-4)--(10.3,-4)--(10.3,-3.7)--(10,-3.7)--(10,-4);
\draw[thick, fill = gray!60] (10,-3)--(10.3,-3)--(10.3,-2.7)--(10,-2.7)--(10,-3);
\draw[thick, fill = gray!90] (10,-2)--(10.3,-2)--(10.3,-1.7)--(10,-1.7)--(10,-2);
\node at (11.3,-3.8) {Depth 1};
\node at (11.3,-2.8) {Depth 2};
\node at (11.3,-1.8) {Depth 3};
\end{tikzpicture}
\end{center}
\caption{An example of the correspondence between $m$-diagram depth and augmented web depth. Above, an $m$-diagram with non-maximal second arcs removed and regions shaded by depth. Below, the corresponding augmented web with first and maximal second arcs mostly preserved, and faces shaded by depth.}
\label{depthcorfig}
\end{figure}
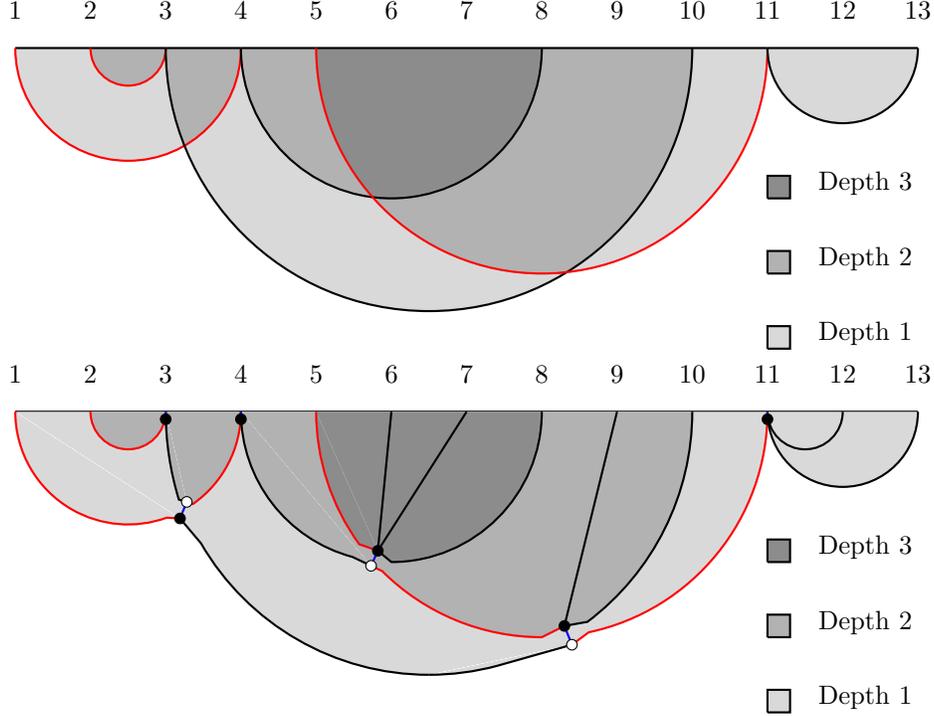

\begin{theorem}
\label{biject}
The function $\varphi: WNC(n,d,3) \rightarrow AW(n,d)$ is a bijection. 
\end{theorem}

\begin{proof}
To show that $\varphi$ is invertible, we introduce the following definition. The idea is that we will record extra information in the process of applying $\varphi$ by way of coloring the edges. This extra information will allow us to invert the $\varphi$ map. We will then show that the extra information was redundant, so $\varphi$ is invertible.
\begin{definition}
Given an augmented web $W \in AW(n,d)$, a {\em valid coloring} of $W$ is a coloring of the edges of $W$ with three colors, red, blue, and black such that the following conditions are satisfied:
\begin{enumerate}
\item Every interior vertex is incident to exactly one red edge, exactly one blue edge, and at least one black edge. Additionally, at each interior vertex, the incident red edge shares a face with the incident blue edge.
\item Right-to-left depth boundaries are colored red.
\item Left-to-right depth boundary edges incident to a boundary vertex are colored black.
\item No face has three consecutive edges colored red-blue-red.
\end{enumerate}
Note that we do not require this coloring to be proper, a vertex may have multiple black edges incident. 
\end{definition}

Given a weakly 3-noncrossing set partition $\pi\in WNC(n,d,3)$, we can create a valid coloring of $\varphi(\pi)$ by initially coloring first arcs red, second arcs black, and for each block $b$, the connection between $v_b$ and $b_2$ blue. Then, at each replacement step, color the newly introduced edge blue and preserve all other colors. To see that the coloring obtained is a valid coloring, first note that initially and at each replacement step, the created vertices satisfy property 1 of a valid coloring. Property 2 of valid colorings holds by Lemma~\ref{depthcor}. To see that the third property of valid colorings holds, note that initially no such face exists and at no replacement step could such a face be created.

The interior blue edges contain the information of exactly which replacement steps have been performed, so given an augmented web $\varphi(\pi)$ and the valid coloring obtained through $\varphi$, we can recover $\pi$. Therefore, it suffices to show that every augmented web $w$ admits exactly one valid coloring.

To do so, we will give an algorithm for finding a valid coloring and show that each step is forced. Let $W$ be an augmented web. 

\begin{enumerate}

\item Consider the set of all right-to-left depth boundaries. In order to satisfy condition 2 of a valid coloring, we must color all such faces red. By Lemma~\ref{depthboundaryvertex}, every vertex now has exactly one red edge. Similarly, color all left-to-right depth boundary edges which are incident to a boundary vertex black.

\item Every interior vertex is now incident to a red edge, so all remaining edges must be blue or black. Edges which do not share a face with the red edge at each of their vertices must be black by condition 1 of valid colorings, so color all such edges black.

\item Consider the set of yet uncolored edges adjacent to two red edges on the same face. By condition 3 of valid colorings, these edges must be black, so color all such edges black.

\item Consider the set of yet uncolored edges. Every interior vertex is incident to at most two of these edges, so their union consists of a set of disjoint paths and cycles. We claim that their union is in fact a disjoint union of paths with exactly one interior endpoint and odd length paths with two interior endpoints.
Given this claim, there is exactly one way to satisfy condition 1 of valid colorings by coloring some of these edges blue and the rest black, that is, by coloring edges of a path alternating blue and black, starting at an interior endpoint of the path. So, if the claim holds, we are done.

\end{enumerate}

To see that the claim holds, suppose towards a contradiction there is a cycle $C$ among the uncolored edges after step 3. Let $D_k$ be the maximal $k$ depth boundary path sharing a vertex with $C$. Then $D_k$ shares exactly one edge with $C$, as the uncolored edges incident to $v$ cannot be on the same side of the depth boundary path passing through $v$. But this is a contradiction as that edge would have been colored black at step 3. Now suppose there is an even length path whose endpoints are both interior vertices. The path must be as shown below, though possibly longer or with the colors of vertices and trips reversed.
 
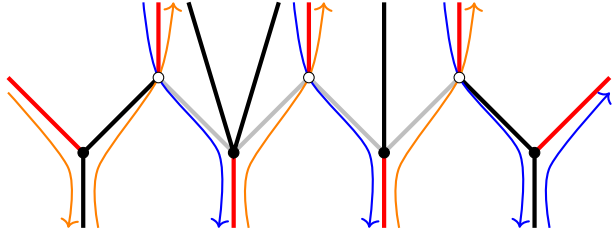
\begin{center}
 \begin{tikzpicture}
 \draw[ultra thick,lightgray] (0,0)--(1,1)--(2,0)--(3,1)--(4,0)--(5,1)--(6,0);
 \draw[ultra thick,black] (0,0)--(1,1);
 \draw[ultra thick,black] (5,1)--(6,0);
 \draw[ultra thick,black] (0,-1)--(0,0);
 \draw[ultra thick,black] (6,-1)--(6,0);
 \draw[ultra thick,red] (0,0)--(-1,1);
 \draw[ultra thick,red] (1,1)--(1,2);
 \draw[ultra thick,red] (3,1)--(3,2);
 \draw[ultra thick,red] (6,0)--(7,1);
 \draw[ultra thick,red] (5,1)--(5,2);
 \draw[ultra thick,red] (4,0)--(4,-1);
 \draw[ultra thick,red] (2,0)--(2,-1);
 \draw[ultra thick,black] (1.4,2)--(2,0)--(2.6,2);
 \draw[ultra thick,black] (4,2)--(4,0);
 \draw[thick, orange, ->] plot [smooth] coordinates {(4.2,-1) (4.2,-.2) (5,1) (5.2,2)};
 \draw[thick, orange, ->] plot [smooth] coordinates {(2.2,-1) (2.2,-.2) (3,1) (3.2,2)};
 \draw[thick, orange, ->] plot [smooth] coordinates {(0.2,-1) (0.2,-.2) (1,1) (1.2,2)};
 \draw[thick, orange, ->] plot [smooth] coordinates {(-1,.8) (-.2,-.2) (-.2,-1)};
 \draw[thick, blue, ->] plot [smooth] coordinates {(4.8,2) (5,1) (5.8,0) (5.8,-1)};
 \draw[thick, blue, ->] plot [smooth] coordinates {(2.8,2) (3,1) (3.8,0) (3.8,-1)};
 \draw[thick, blue, ->] plot [smooth] coordinates {(0.8,2) (1,1) (1.8,0) (1.8,-1)};
 \draw[thick, blue, ->] plot [smooth] coordinates {(6.2,-1) (6.2,-.2) (7,.8)};
 \draw[fill = black] (0,0) circle (2pt);
 \draw[fill = black] (2,0) circle (2pt);
 \draw[fill = black] (4,0) circle (2pt);
 \draw[fill = black] (6,0) circle (2pt);
 \draw[fill = white] (1,1) circle (2pt);
 \draw[fill = white] (3,1) circle (2pt);
 \draw[fill = white] (5,1) circle (2pt);
 \end{tikzpicture}
 \end{center}
 For the red edges to have been colored red in step 1, the trips drawn in orange must be exceedances, and the trip drawn in blue must not be. The two rightmost blue trips cannot cross more than once as $W$ is reduced, so the face between $n$ and $1$ must lie to the right of the second rightmost blue trip. By Lemma~\ref{nokiss}, the leftmost orange trip and the three leftmost blue trips cannot cross eachother, and since $W$ is reduced, the leftmost orange trip cannot cross itself. Thus, the leftmost orange trip cannot be an exceedance, but this is a contradiction. Lastly, suppose there is a boundary to boundary path among the uncolored edges. Consider the two black vertices incident to the endpoints of this path. The red edges adjacent to them must lie on the same side of the path, so one of the edges along this path incident to the boundary must be a left-to-right depth boundary edge, and this is a contradiction.

 Thus, every augmented web has exactly one valid coloring, and $\varphi$ is a bijection.

\end{proof}
\section{$SL_3$ invariants for normal plabic graphs}

In this section we introduce an invariant associated to each perfectly orientable normal plabic graph in such a way that the invariants associated to augmented webs form a basis of $S^{(d^3, 1^{n-3d})}$. To do so, we first need to give an orientation to each plabic graph, which will determine the sign of our invariant.

\subsection{Perfect orientations}
A key idea in defining our augmented web invariants is that of a {\em perfect orientation}. Perfect orientations were introduced by A. Postnikov in \cite{Postnikov}. Our definition will be slightly different in that our sink vertices will be interior vertices rather than boundary vertices, and we also include the information of a total order on the sinks.
\begin{definition}
Let $W$ be a normal plabic graph. A perfect orientation $\mathcal{O}$ of $W$ is a choice of two things. First, an orientation of each edge of $W$ such that each boundary edge is oriented away from the boundary, each interior white vertex has exactly one ingoing edge, each interior black vertex has {\em at most} one outgoing edge. There will then be a set of $d$ black vertices with no outward edges, we refer to these as the sinks of $\mathcal{O}$ and denote them by $S_{\mathcal{O}}$. We also require that for every vertex $v$ in $W$ there is a directed path from $v$ to a sink vertex. A perfect orientation also includes the information of a total order on the sinks, i.e. a bijection $f_\mathcal{O}: S_{\mathcal{O}} \rightarrow  \{1, \dots, d\}$. For each perfect orientation, we call the set of edges which are oriented away from black vertices the independent set of $\mathcal{O}$ and denote it $I(\mathcal{O})$.

If there exists at least one perfect orientation of $W$, we say that $W$ is perfectly orientable.
\end{definition}

\begin{remark}
The set of perfectly orientable plabic under our definition differs slightly from the set of perfectly orientable plabic graphs under Postnikov's definition. For example, a white vertex connected by three edges to a single black vertex which is also connected to the boundary is perfectly orientable under Postnikov's definition but not ours. However, every plabic graph with an acyclic perfect orientation per Postnikov's definition will be perfectly orientable per our definition, and Postnikov, Speyer, and Williams show that all reduced plabic graphs have an acyclic perfect orientation \cite[Lemma 3.2]{PSW}.  
\end{remark}

We use this modified definition in order to allow for perfect orientations to be obtained from eachother via a sequence of small changes. 

\begin{definition}
Let $W$ be an augmented web with perfect orientation $\mathcal{O}$. Let $v$ be a white vertex of $W$. A {\em swivel move} is a change in orientation of exactly one ingoing edge and one outgoing edge at $v$ which connect to distinct black vertices. We have necessarily removed one sink vertex and added one sink vertex, let the new sink be in the same position in the total order as the old sink. We call this a swivel move due to the fact that the set $I(\mathcal{O})$ is being rotated around this white vertex.
\end{definition}

\begin{proposition}
\label{swiveldecomp}
Any two perfect orientations can be transformed into each other via a sequence of swivel moves and a reordering of the sink vertices.
\end{proposition}

\begin{proof}
Let $\mathcal{O}_1$ and $\mathcal{O}_2$ be two perfect orientations. Consider the symmetric difference of independent sets of the two perfect orientations, $I(\mathcal{O}_1) \Delta I(\mathcal{O}_2)$ it is necessarily a union of disjoint cycles and paths between sinks of $\mathcal{O}_1$ and sinks of $\mathcal{O}_2$. We induct on the number of cycles present in $I(\mathcal{O}_1) \Delta I(\mathcal{O}_2)$. If there are no cycles, performing a swivel move at each white vertex along the paths from sinks of $\mathcal{O}_1$ and $\mathcal{O}_2$ using the edges of this path will transform $\mathcal{O}_1$ into a perfect orientation the same as $\mathcal{O}_2$ up to a reordering of its sinks. If there is a cycle in $I(\mathcal{O}_1) \Delta I(\mathcal{O}_2)$, find a walk in $W$ which starts at a sink vertex of $\mathcal{O}_1$, travels to a white vertex of the cycle, travels around the cycle, then returns to the starting sink via the same path such that every other edge of this walk is in $I(\mathcal{O}_1)$. Performing a swivel move at each white vertex of this walk using the edges along this walk will result in a perfect orientation $\mathcal{O}_3$ for which $I(\mathcal{O}_3) \Delta I(\mathcal{O}_2)$ has one fewer cycles.
\end{proof}

We first check that this is a sensible definition of sign.
\begin{proposition}
\label{signagree}
Let $\mathcal{O}_1, \mathcal{O}_2, \mathcal{O}_3$ be three perfect orientations of a web $W \in AW(n,d)$. Then we have
\[
\textrm{sign}(\mathcal{O}_1, \mathcal{O}_3) = \textrm{sign}(\mathcal{O}_1, \mathcal{O}_2) \textrm{sign}(\mathcal{O}_2, \mathcal{O}_3)
\]
\end{proposition}

\begin{proof}

By Proposition~\ref{swiveldecomp}, it suffices to check this holds when $\mathcal{O}_2$ and $\mathcal{O}_3$ differ by a single swivel move. 
Note that 
\[
I(\mathcal{O}_1) \Delta I(\mathcal{O}_3) = (I(\mathcal{O}_1) \Delta I(\mathcal{O}_2) )\Delta (I(\mathcal{O}_2) \Delta I(\mathcal{O}_3)),
\]
and $(I(\mathcal{O}_2) \Delta I(\mathcal{O}_3))$ is size two. We split into three cases.

\begin{itemize} 
\item Case 1: $(I(\mathcal{O}_1) \Delta I(\mathcal{O}_2) )$ and $(I(\mathcal{O}_2) \Delta I(\mathcal{O}_3))$ do not intersect.
Then $(I(\mathcal{O}_1) \Delta I(\mathcal{O}_2) )$ and $(I(\mathcal{O}_1) \Delta I(\mathcal{O}_2) )$ are the same except one path is two edges longer, and $\sigma(\mathcal{O}_1, \mathcal{O}_2) = \sigma(\mathcal{O}_1, \mathcal{O}_3)$. Thus, $\textrm{sign}(\mathcal{O}_1, \mathcal{O}_3) = -\textrm{sign}(\mathcal{O}_1, \mathcal{O}_2)$

\item Case 2: $(I(\mathcal{O}_1) \Delta I(\mathcal{O}_2) )$ and $(I(\mathcal{O}_2) \Delta I(\mathcal{O}_3))$ intersect in a single edge.
Then $(I(\mathcal{O}_1) \Delta I(\mathcal{O}_2) )$ and $(I(\mathcal{O}_1) \Delta I(\mathcal{O}_2) )$ are the same except for two paths. In this case, $\sigma(\mathcal{O}_1, \mathcal{O}_2)$ and $\sigma(\mathcal{O}_1, \mathcal{O}_3)$ differ by a single transposition and $\textrm{sign}(\mathcal{O}_1, \mathcal{O}_3) = -\textrm{sign}(\mathcal{O}_1, \mathcal{O}_2)$.

\item Case 3: $(I(\mathcal{O}_1) \Delta I(\mathcal{O}_2) )$ and $(I(\mathcal{O}_2) \Delta I(\mathcal{O}_3))$ interesect in two edges.
Then $(I(\mathcal{O}_1) \Delta I(\mathcal{O}_2) )$ and $(I(\mathcal{O}_1) \Delta I(\mathcal{O}_2) )$ are the same except one path is two edges shorter, and $\sigma(\mathcal{O}_1, \mathcal{O}_2) = \sigma(\mathcal{O}_1, \mathcal{O}_3)$. Thus, $\textrm{sign}(\mathcal{O}_1, \mathcal{O}_3) = -\textrm{sign}(\mathcal{O}_1, \mathcal{O}_2)$
\end{itemize}

\end{proof}

Proposition~\ref{signagree} allows for an alternative definition of relative sign, which we find more intuitive. The relative sign between two perfect orientations is given by the parity of the number of swivel moves and the sign of the permutation of the sinks in any sequence of swivel moves and a permutation of the sinks which transforms one orientation into the other. Proposition~\ref{signagree} guarantees this is well defined. In this sense, swivel moves can be thought of as playing a similar role to adjacent transpositions in determing the sign of a permutation.

\subsection{Consistent Labellings}

Our invariants will be defined in terms of consistent labellings, which we now define.

\begin{definition}
Let $W$ be a perfectly orientable normal plabic graph. A consistent labelling $\ell$ of $W$ is a choice of a possibly empty subset $\ell(e)$ of $\{1, \dots, \nu\}$ for each edge $e$ of $W$, such that the following hold:
\begin{enumerate}
\item At each interior white vertex, incident edge labels are disjoint and their union is $\{1,2,3\}$. 

\item At each black vertex, incident edge labels are disjoint and their union contains $\{1,2,3\}$. 

\item The label at each boundary edge has size 1, and for each $i\in \{4, \dots, \nu\}$, $\{i\}$ appears exactly once among boundary labels.
\end{enumerate}
The edges whose labels contain 1, 2, or 3 can be thought of as determining three dimer covers of $W$.

The boundary word of $\ell$ is the word given by reading off the labels at boundary edges in order, denoted
\[
\textrm{bd}(\ell) = \textrm{bd}(\ell)_1 \cdots \textrm{bd}(\ell)_n
\]The boundary monomial of $\ell$ is the monomial 
\[
\mathbf{x}_{\textrm{bd}(\ell)} = x_{\textrm{bd}(\ell)_1,1} \cdots x_{\textrm{bd}(\ell)_n,n}
\]
The weight of a consistent labelling is
\[
 \textrm{wt}(\ell) = \left(-\frac{1}{2}\right)^{\#\textrm{edge labels of size two}}
\]

To each consistent labelling we also associate a sign, made up of a number of factors. Firstly, for each $1 \leq i \leq 3$, let $E_i$ denote the set of edges of $W$ whose label contains $i$. Consider the symmetric difference $E_i \Delta \mathcal{O}$. This will be a union of disjoint cycles and disjoint paths from sinks of $\mathcal{O}$ to boundary vertices whose incident edge is labelled $i$. For each boundary vertex $b$ with label $i$, let the {\em origin} of $b$ be the sink vertex it connects to, denoted $\textrm{origin}(b)$.  An origin inversion $of$ $w$ is a pair of boundary vertices $b_1 < b_2$ with $\textrm{origin}(b_1) > \textrm{origin}(b_2)$. For each $i$, we get a contribution to the sign of the labelling of 
\[
(-1)^{\#\textrm{origin inversions between vertices labelled } i +\# \textrm{cycles of length 2 modulo 4 in }E_i \Delta \mathcal{O}}
\]
We also have contributions to the sign of a consistent labelling coming from the number of edges of $W$ with labels of size 2, the number of edges of $I(\mathcal{O})$ with an even size label, and inversions in the boundary word of $\ell$. The $\textrm{sign}$ of $\ell$ with respect to orientation $\mathcal{O}$ is given by
\[
\left(\prod_{i=1}^3  (-1)^{\# \textrm{cycles of length 2 modulo 4 in }E_i \Delta \mathcal{O}} \right)
(-1)^{\#\textrm{origin inversions} + \textrm{inv}(\textrm{bd}(\ell)) + \#\{e \in I(\mathcal{O})\mid |\ell(e)| \textrm{ is even}\}}.
\]

We can also think of the sign contribution coming from origin inversions in a different way. Let the {\em decorated boundary word} of $\ell$, $\tilde{\textrm{bd}}(\ell)$, be the boundary word of $\ell$ with a subscript for the origin attached to each letter $1 \leq i \leq 3$. We can consider the decorated boundary word to be a permutation under the order
\[
1_1 < 1_2 < \dots <1_d <2_1 <\dots <3_d < 4 < \dots < \nu
\]
Then the sign of $\ell$ is 
\[
\left(\prod_{i=1}^3 (-1)^{\# \textrm{cycles of length 2 modulo 4 in }E_i \Delta \mathcal{O}} \right)
\textrm{sign}(\tilde{\textrm{bd}}(\ell))(-1)^{\#\{e \in I(\mathcal{O})\mid |\ell(e)| \textrm{ is even}\}}.
\]
times the sign of the decorated boundary word.
\end{definition}

\begin{remark}
The definition of weight of a labelling is a bit mysterious to us. It is chosen to make the Skein relations upcoming in Section~\ref{skeinsect} hold, and we lack any further explanation beyond that. The definition of sign of a labelling is chosen so that a change in orientation introduces a consistent change in sign among all possible consistent labellings. 
\end{remark}

\begin{example}Consider the augmented web and consistent labelling shown below. Edges in $I(\mathcal{O})$ are highlighted in yellow, and the three sinks are labelled with their position in the total order on sinks.
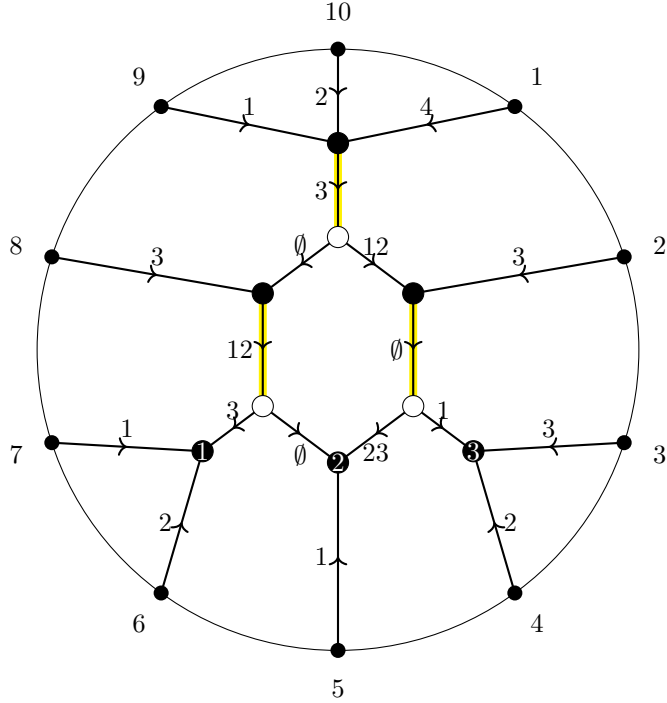
\begin{center}
\begin{tikzpicture}[scale = 2]
\equicc[2cm]{10}{0}{0};
\coordinate (N11) at (0,1.5/2);
\coordinate (N12) at (1/2,.75/2);
\coordinate (N13) at (1/2,-.75/2);
\coordinate (N14) at (0,-1.5/2);
\coordinate (N15) at (-1/2,-.75/2);
\coordinate (N16) at (-1/2,.75/2);
\coordinate(N17) at (0,2.75/2);
\coordinate (N18) at (1.8/2,-1.35/2);
\coordinate (N19) at (-1.8/2,-1.35/2);
\begin{scope}[thick,decoration={
    markings,
    mark=at position 0.5 with {\arrow{>}}}
    ] 
\draw[line width = 3pt, yellow] (N17)--(N11);
\draw[line width = 3pt, yellow] (N16)--(N15);
\draw[line width = 3pt, yellow] (N13)--(N12);

\draw[postaction={decorate}] (N6)--(N19) node[midway,left] {2};
%\draw (N19)--+(.2,-.2);
\draw[postaction={decorate}] (N7)--(N19) node[midway,above] {1};
\draw[postaction={decorate}] (N5)--(N14) node[midway,left] {1};
%\draw (N14)--+(.2,-.2);
\draw[postaction={decorate}] (N15)--(N14) node[midway,below] {$\emptyset$};
\draw[postaction={decorate}] (N16)--(N15) node[midway,left] {12};
\draw[postaction={decorate}] (N11)--(N16) node[midway,above] {$\emptyset$};
\draw[postaction={decorate}] (N11)--(N12) node[midway,above] {12};
\draw[postaction={decorate}] (N12)--(N13) node[midway,left] {$\emptyset$};
\draw[postaction={decorate}] (N13)--(N14) node[midway,below] {23};
\draw[postaction={decorate}] (N8)--(N16) node[midway,above] {3};
%\draw (N16)--+(-.1,.25);
\draw[postaction={decorate}] (N9)--(N17) node[midway,above] {1};
\draw[postaction={decorate}] (N1)--(N17) node[midway,above] {4};
\draw[postaction={decorate}] (N10)--(N17) node[midway,left] {2};
%\draw (N17)--+(.2,-.2);
\draw[postaction={decorate}] (N17)--(N11) node[midway,left] {3};
\draw[postaction={decorate}] (N2)--(N12) node[midway,above] {3};
%\draw (N12)--+(.1,.25);
\draw[postaction={decorate}] (N15)--(N19) node[midway,above] {3};
\draw[postaction={decorate}] (N3)--(N18) node[midway,above] {3};
%\draw (N18)--+(.2,.2);
\draw[postaction={decorate}] (N4)--(N18) node[midway,right] {2};
\draw[postaction={decorate}] (N13)--(N18) node[midway,above] {1};
%\draw (N11)--+(0,-.3);
%\draw (N13)--+(0,-.3);
%\draw (N15)--+(0,-.3);
\end{scope}
\draw[fill=black]  (N17) circle (2pt);
\draw[fill=black]  (N18) circle (2pt);
\draw[fill=black]  (N19) circle (2pt);
\draw[fill=white]  (N11) circle (2pt);
\draw[fill=black]  (N12) circle (2pt);
\draw[fill=white]  (N13) circle (2pt);
\draw[fill=black]  (N14) circle (2pt);
\draw[fill=white]  (N15) circle (2pt);
\draw[fill=black]  (N16) circle (2pt);
\node[white] at (N14) {2};
\node[white] at (N18) {3};
\node[white] at (N19) {1};

\end{tikzpicture}
\end{center}
The boundary word of this labelling is $4332121312$, with 28 inversions. The origins of the boundary vertices for each label are:

\begin{center}\begin{tabular}{ccccccccccc}
boundary vertices labelled 1: & 5&7&9\\
origins: &2&1&3\\
\end{tabular}

\begin{tabular}{ccccccccccc}
boundary vertices labelled 2: & 4&6&10\\
origins: &3&1&2\\
\end{tabular}

\begin{tabular}{ccccccccccc}
boundary vertices labelled 3: & 2&3&8\\
origins: &2&3&1\\
\end{tabular}
\end{center}

The total number of origin inversions of $\ell$ is 5, one from label 1 and two each from labels 2 and 3. The decorated boundary word of $\ell$ is $43_23_32_31_22_11_13_11_32_2$, with 33 inversions. Our orientation is acyclic, so there are no cycles to consider, and two of the highlighted edges have an even size label. The sign of this labelling is therefore
\[
(-1)^{33}(-1)^2 = -1
\]
There are three edges with labels of size two, so the weight of this labelling is $(-\frac{1}{2})^3 = -\frac{1}{8}$.
\end{example}

\begin{proposition}
\label{compatsame}
Let $W$ be a perfectly orientable normal plabic graph with consistent labelling $\ell$. Let $\mathcal{O}_1$ and $\mathcal{O}_2$ be two distinct perfect orientations for $\ell$. Then the sign of $\ell$ with respect to these orientations is related by
\[
\textrm{sign}(\ell, \mathcal{O}_1) = \textrm{sign}(\mathcal{O}_1, \mathcal{O}_2) \textrm{sign}(\ell, \mathcal{O}_2)
\]
\end{proposition}

\begin{proof}
It suffices to show that this holds when $\mathcal{O}_1$ and $\mathcal{O}_2$ differ by a swivel move at a white vertex $v$. Let $u_1$ and $u_2$ be the sinks which vary between $\mathcal{O}_1$ and $\mathcal{O}_2$, and let $u_3$ denote the third neighbor of $v$. The origin of each boundary vertex is the same in $\mathcal{O}_1$ and $\mathcal{O}_2$ except for the the boundary vertices whose origin path travels along edge $(v, u_3)$, which have swapped origins in $\mathcal{O}_1$ and $\mathcal{O}_2$. Either we have $|\ell(v, u_3)|$ is odd, in which case $|\ell(v, u_2)|$ and $|\ell(v, u_1)|$ are the same parity, or $|\ell(v, u_3)|$ is even, in which case $|\ell(v, u_2)|$ and $|\ell(v, u_1)|$ have opposite parity. In either case, we have
\[
\textrm{sign}(\ell, \mathcal{O}_1) = -\textrm{sign}(\ell, \mathcal{O}_2)
\]
as desired.
\end{proof}

We can now define our invariants for normal plabic graphs.

\begin{definition}
Let $W$ be a perfectly orientable normal plabic graph with perfect orientation $\mathcal{O}$. Let $CL(W)$ denote the set of all consistent labellings of $W$. Define an $SL_3$ invariant attached to $W$, denoted $[W, \mathcal{O}]$ by:
\[
[W, \mathcal{O}] = \sum_{\ell \in CL(W)}\textrm{sign}(\ell, \mathcal{O}) \textrm{wt}(\ell) \mathbf{x}_{\textrm{bd}(\ell)}
\]
\end{definition}

\begin{example}
Consider the augmented web $W \in AW(7,2)$ and perfect orientation $\mathcal{O}$ shown below. 
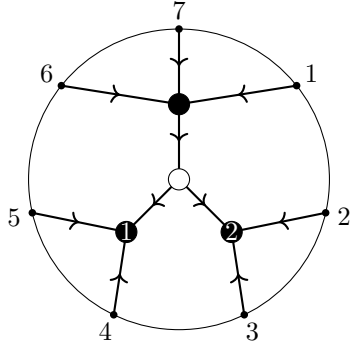
\begin{center}
\begin{tikzpicture}
\equicc[2cm]{7}{0}{0};
\coordinate (N8) at (0,0);
\coordinate (N9) at (0,1);
\coordinate (N10) at (-.7,-.7);
\coordinate (N11) at (.7,-.7);
\begin{scope}[thick,decoration={
    markings,
    mark=at position 0.5 with {\arrow{>}}}
    ] 
    
\draw[postaction={decorate}] (N7)--(N9) ;
\draw[postaction={decorate}] (N6)--(N9) ;
\draw[postaction={decorate}] (N1)--(N9) ;
\draw[postaction={decorate}] (N9)--(N8) ;
\draw[postaction={decorate}] (N8)--(N10) ;
\draw[postaction={decorate}] (N8)--(N11) ;
\draw[postaction={decorate}] (N2)--(N11) ;
\draw[postaction={decorate}] (N3)--(N11) ;
\draw[postaction={decorate}] (N4)--(N10) ;
\draw[postaction={decorate}] (N5)--(N10) ;
 \end{scope}
 \draw[fill = black] (N9) circle (4pt);
 \draw[fill = black] (N10) circle (4pt);
 \draw[fill = black] (N11) circle (4pt);
 \draw[fill = white] (N8) circle (4pt);
 
 \node[white] at (N10) {1};
 \node[white] at (N11) {2};
\end{tikzpicture}
\end{center}
There are 288 consistent labellings in total, but only 2 up to graph automorphism (not necessarily boundary preserving) and permutation of $\{1,2,3\}$, shown below:

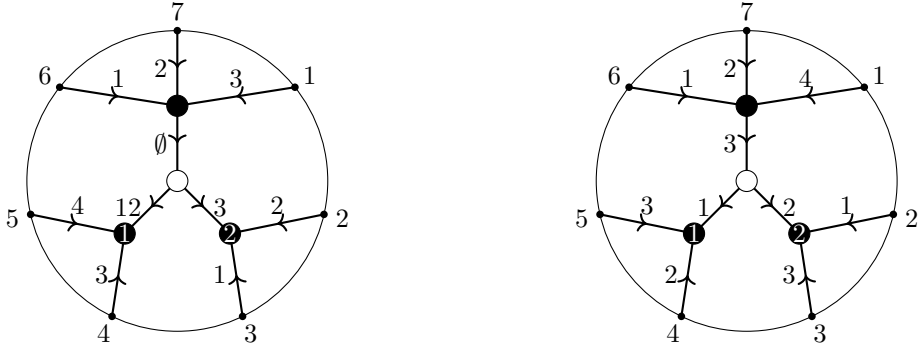
\begin{center}
\begin{tikzpicture}
\equicc[2cm]{7}{0}{0};
\coordinate (N8) at (0,0);
\coordinate (N9) at (0,1);
\coordinate (N10) at (-.7,-.7);
\coordinate (N11) at (.7,-.7);
\begin{scope}[thick,decoration={
    markings,
    mark=at position 0.5 with {\arrow{>}}}
    ] 
    
\draw[postaction={decorate}] (N7)--(N9) node[midway, left] {2};
\draw[postaction={decorate}] (N6)--(N9) node[midway, above] {1};
\draw[postaction={decorate}] (N1)--(N9) node[midway, above] {3};
\draw[postaction={decorate}] (N9)--(N8) node[midway, left] {$\emptyset$};
\draw[postaction={decorate}] (N8)--(N10) node[midway, left] {12};
\draw[postaction={decorate}] (N8)--(N11) node[midway, right] {3};
\draw[postaction={decorate}] (N2)--(N11) node[midway, above] {2};
\draw[postaction={decorate}] (N3)--(N11) node[midway, left] {1};
\draw[postaction={decorate}] (N4)--(N10) node[midway, left] {3};
\draw[postaction={decorate}] (N5)--(N10) node[midway, above] {4};
 \end{scope}
 \draw[fill = black] (N9) circle (4pt);
 \draw[fill = black] (N10) circle (4pt);
 \draw[fill = black] (N11) circle (4pt);
 \draw[fill = white] (N8) circle (4pt);
 
 \node[white] at (N10) {1};
 \node[white] at (N11) {2};
\end{tikzpicture}
\hspace{1in}
\begin{tikzpicture}
\equicc[2cm]{7}{0}{0};
\coordinate (N8) at (0,0);
\coordinate (N9) at (0,1);
\coordinate (N10) at (-.7,-.7);
\coordinate (N11) at (.7,-.7);
\begin{scope}[thick,decoration={
    markings,
    mark=at position 0.5 with {\arrow{>}}}
    ] 
    
\draw[postaction={decorate}] (N7)--(N9) node[midway, left] {2};
\draw[postaction={decorate}] (N6)--(N9) node[midway, above] {1};
\draw[postaction={decorate}] (N1)--(N9) node[midway, above] {4};
\draw[postaction={decorate}] (N9)--(N8) node[midway, left] {3};
\draw[postaction={decorate}] (N8)--(N10) node[midway, left] {1};
\draw[postaction={decorate}] (N8)--(N11) node[midway, right] {2};
\draw[postaction={decorate}] (N2)--(N11) node[midway, above] {1};
\draw[postaction={decorate}] (N3)--(N11) node[midway, left] {3};
\draw[postaction={decorate}] (N4)--(N10) node[midway, left] {2};
\draw[postaction={decorate}] (N5)--(N10) node[midway, above] {3};
 \end{scope}
 \draw[fill = black] (N9) circle (4pt);
 \draw[fill = black] (N10) circle (4pt);
 \draw[fill = black] (N11) circle (4pt);
 \draw[fill = white] (N8) circle (4pt);
 
 \node[white] at (N10) {1};
 \node[white] at (N11) {2};
\end{tikzpicture}
\end{center}
The left labelling has combined sign and weight of $-\frac{1}{2}$, and the right labelling has combined sign and weight  $1$. Let $\textrm{Aut}(W) \subset \mathfrak{S}_n$ denote the group of automorphisms of $W$ identified with the corrseponding permutation of boundary vertices, which has size $6\cdot2\cdot2\cdot2=48$. Then we have
\begin{align*}
[W, \mathcal{O}] = \sum_{\sigma \in \textrm{Aut}(W)} \sum_{\omega \in A_3} \textrm{sign}(\sigma)( &-\frac{1}{2} x_{\omega(3), \sigma(1)} x_{\omega(2), \sigma(2)} x_{\omega(1), \sigma(3)} x_{\omega(3), \sigma(4)} x_{4, \sigma(5)} x_{\omega(1), \sigma(6)}x_{\omega(2), \sigma(7)}  \\&+ x_{4, \sigma(1)} x_{\omega(1), \sigma(2)} x_{\omega(3), \sigma(3)} x_{\omega(2), \sigma(4)} x_{\omega(3), \sigma(5)} x_{\omega(1), \sigma(6)}x_{\omega(2), \sigma(7)})
\end{align*}
where $A_3$ is the alternating group on $\{1,2,3\}$.
\end{example}

To verify that this is a sensible definition, we first check that changing the orientation only introduces a change of sign.
\begin{proposition}
Let $W$ be a perfectly orientable normal plabic graph with perfect orientations $\mathcal{O}_1$ and $\mathcal{O}_2$. Then
\[
[W, \mathcal{O}_1] = \textrm{sign}(\mathcal{O}_1, \mathcal{O}_2) [W, \mathcal{O}_2]
\] 
\end{proposition}
\begin{proof}
By Proposition~\ref{signagree}, we have
\begin{align*}
[W, \mathcal{O}_1] &= \sum_{\ell \in CL(W)} \textrm{sign}(\ell, \mathcal{O}_1) \textrm{wt}(\ell) \mathbf{x}_{\textrm{bd}(\ell)} \\&= \textrm{sign}(\mathcal{O}_1, \mathcal{O}_2)\sum_{\ell \in CL(W)}  \textrm{sign}(\ell, \mathcal{O}_2) \textrm{wt}(\ell) \mathbf{x}_{\textrm{bd}(\ell)}\\&=\textrm{sign}(\mathcal{O}_1, \mathcal{O}_2)[W, \mathcal{O}_2]
\end{align*}
\end{proof}

We call these invariants because the resulting polynomials will be invariant under a certain action of $SL_3$. Define an action of $SL_3$ which acts on the matrix of $3n$ variables 
\[
\begin{bmatrix}
x_{1,1} & x_{1,2} & \cdots & x_{1,n}\\
x_{2,1} & x_{2,2} & \cdots & x_{2,n}\\
x_{3,1} & x_{3,2} & \cdots & x_{3,n}
\end{bmatrix}
\]
via left multiplication and leaves all other variables fixed. Then we have the following:
\begin{proposition}
\label{invariant}
Let $X \in SL_3$ and $W \in AW(n,d)$ with perfect orientation $\mathcal{O}$. Then
\[
X \cdot [W,\mathcal{O}] = [W,\mathcal{O}]
\]
\end{proposition}
This is clear for normal plabic graphs without interior white vertices and with only vertices of degree at least 3, i.e. jellyfish invariants. We will defer the proof for all augmented webs until Section 6, which will show that augmented web invariants live in $\mathfrak{S}_n$ closure of jellyfish invariants. The action of $\mathfrak{S}_n$ commutes with the action of $SL_3$, so the result will follow.

We next show that this definition generalizes jellyfish invariants.

\begin{proposition}
\label{matchjellyfish}
Let $W \in AW(n,d)$ have no white vertices, let $\mathcal{O}$ be a perfect orientation of $W$, and let $\pi$ be the corresponding ordered set partition. Then $[W, \mathcal{O}] = [\pi]_3$.
\end{proposition}

\begin{proof}Let $v_1, \dots, v_d$ denote the interior vertices of $W$ in order. We claim that a consistent labelling of $W$ corresponds to a choice of jellyfish tableau for $\pi$ as well as a choice of permutation for the elements of each block of $\pi$. Indeed, we can create a jellyfish tableau $T_\ell$ for $\pi$ in the following manner. For each boundary vertex $1 \leq b \leq n$, if $b$ has label $i$ and is connected to interior vertex $v_j$, fill box $i,j$ with the entry $b$. Then, for each interior vertex $v_j$, let $R_j(\ell)$ be the set of boundary vertices connected to $v_j$. Let $\sigma(v_j) \in \mathfrak{S}_n$ denote the permutation which reorders the elements of $R_j$ to have increasing labels. We claim that
\begin{equation}
\label{termmatch}
\textrm{sign}(\ell) \mathbf{x}_{\textrm{bd}(\ell)} = \textrm{sign}(J(T_\ell)) \prod_{j = 1}^d \textrm{sign}(\sigma(v_j)) \left(\prod_{i\in R_j} x_{\textrm{bd}(\ell)_i,i }\right)
\end{equation}
The right side here represents one term in the monomial expansion of the product of determinants defining $J(T_\ell)$. From the definition of $\mathbf{X}_{\textrm{bd}(\ell)}$ we see that the variables appearing on both sides of (\ref{termmatch}) agree, so the content of this claim is that the signs match. We show this in two parts. First, we claim that
\begin{equation}
\# \textrm{origin inversions of } \ell = \#\textrm{inversions within rows of }J(T_\ell)
\end{equation}
Suppose $(b_1,b_2)$ is an origin inversion of $\ell$, with label $i$. Then $b_1$ appears in box $(i,\textrm{origin}(b_1))$ and $b_2$ appears in box $(i,\textrm{origin}(b_2))$. So $(b_1,b_2)$ is also an origin inversion of $T_\ell$. Next, we claim that
\begin{equation}
(-1)^{\textrm{inv}(\textrm{bd}(\ell))} = (-1)^{\#\textrm{inversions between rows of }J(T_\ell)}\prod_{j=1}^d \textrm{sign}(\sigma(v_j))
\end{equation}
or equivalently,
\begin{equation}
\label{invmatch}
(-1)^{\textrm{inv}(\sigma(v_1)\sigma(v_2)\cdots\sigma(v_d)\cdot\textrm{bd}(\ell))} = (-1)^{\#\textrm{inversions between rows of }J(T_\ell)}.
\end{equation}
If the $k^{th}$ letter of $(\sigma(v_1)\sigma(v_2)\cdots\sigma(v_d)\cdot\textrm{bd}(\ell))$ is $i$, then a $k$ appears in the $i^{th}$ row of $J(T)$, so (\ref{invmatch}) holds. We therefore have
\[
[W, \mathcal{O}] = [\pi]_3
\]
as desired.
\end{proof}

%When an augmented web has only vertices of degree 3, and is thus a standard $SL_3$ web, the normal plabic graph invariant is {\em not} the same as the usual $SL_3$ web invariant, but will be unitriangularly related. 

%\begin{proposition}
%The set $\{ [W,\mathcal{O}_W] \mid W \in AW(n,\frac{n}{3}), W \textrm{ has no vertices of degree more than 3} \}$ is a basis of the $SL_3$ invariant space of $(\mathbb{C}^3)^{\otimes \frac{n}{3}}$ unitriangularly related to the $SL_3$ web basis.
%\end{proposition}

%\begin{proof}
%Let $W$ be a normal plabic graph with no vertices of degree more than 3. There is a bijection from consistent labellings of $W$ to the disjoint union of proper labellings of all $SL_3$ webs which are graph theoretic minors of $W$, given by deleting edges with labels of size $0$ then contracting edges incident to degree 2 white vertices. This bijection preserves boundary labels, so $[W, \mathcal{O}_W]$ is equal to the standard $SL_3$-web invariant for $W$ plus a sum of $SL_3$-web invariants of webs with fewer white vertices.
%\end{proof}

Augmented web invariants satisfy the rotation and reflection invariance properties laid out in \cite{FPPS}, as well as something slightly stronger. 
\begin{proposition}
Let $W \in AW(n,d)$ with perfect orientation $\mathcal{O}$ and let $\sigma \in \mathfrak{S}_n$. Let $\sigma \cdot W$ be the graph obtained by permuting the boundary vertices of $W$ according to $\sigma$, i.e. if $b$ is a boundary vertex and $(v,b)$ is an edge of $W$, then $(v, \sigma(b))$ will be an edge of $\sigma \cdot W$. Suppose that $\sigma\cdot W$ is planar, so it is also in $AW(n,d)$. Then, abusively letting $\mathcal{O}$ also be a perfect orientation of $\sigma \cdot W$, we have
\[
\sigma \cdot [W, \mathcal{O}] = \textrm{sign}(\sigma)[\sigma \cdot W, \mathcal{O}]
\]
\end{proposition}
\begin{proof}
For each consistent labelling $\ell$ of $W$, we get a corresponding consistent labelling $\sigma \circ \ell$ of $\sigma \circ W$. The decorated boundary word of $\sigma \cdot \ell$ is obtained by applying $\sigma$ to the decorated boundary word of $\ell$, so the result follows.
\end{proof}
\begin{corollary}
\label{rotationreflect}
Let $c$ be the long cycle in $\mathfrak{S}_n$ and let $w_0$ be the long element. Given an augmented web $W \in AW(n,d)$ with orientation $\mathcal{O}_W$, let $\textrm{rot(W)}$ and $\textrm{rot}(\mathcal{O}_W)$ be the web obtained by  rotating $W$ and $\mathcal{O}_W$ clockwise by $\frac{2\pi}{n}$. Let $\textrm{refl}(W)$ and $\textrm{refl}(\mathcal{O}_W)$ be the web and orientation obtained by reflecting $W$ and $\mathcal{O}_w$ across the perpendicular bisector between boundary vertices 1 and n. We have
\[
c \cdot [W, \mathcal{O}] = (-1)^{n-1}[\textrm{rot}(W), \textrm{rot}(\mathcal{O}_W)]
\]
and
\[
w_0 \cdot [W, \mathcal{O}] = (-1)^{n-1}[\textrm{refl}(W), \textrm{refl}(\mathcal{O}_W)]
\]
\end{corollary}

\begin{remark}
In the definition of $\sigma \cdot W$, we are implicitly using the fact that if a web is planar and the positions of its boundary vertices are fixed, it has a unique planar embedding in the disk up to boundary-preserving homeomorphism. This is due to a classical theorem of Whitney \cite{Whitney}.
\end{remark}

Our main theorem regarding augmented web invariants is that they form a basis for the flamingo Specht module. We state it now but defer its proof until the end of the next section. 

\begin{theorem}
\label{basisthm}
Choose a perfect orientation $\mathcal{O}_W$ for each augmented web $W \in AW(n,d)$. Then the set $\{[W, \mathcal{O}_W] \mid W \in AW(n,d)\}$ is a basis for the flamingo Specht module $S^{(d^3, 1^{n-3d})}$.
\end{theorem}

\section{Skein relations for augmented webs}
\label{skeinsect}
This section will introduce skein relations for normal plabic graph invariants, showing that they satsify property (5) of web bases. Furthermore, these skein relations will demonstrate that the span of augmented web invariants is an $\mathfrak{S}_n$ invariant module containing $S^{(d^3, n-3d)}$. Along with our combinatorial bijection from standard Young tableaux, Proposition~\ref{biject}, we will thus obtain a proof of Theorem~\ref{basisthm}.

We first give a diagrammatic representation of these relations. In each image below, shaded gray areas represent an unknown number of edges connecting to other vertices of the graph, and in the perfect orientation, edges are assumed to be oriented towards black vertices unless shown otherwise.

\begin{center}
Crossing Reduction Rule

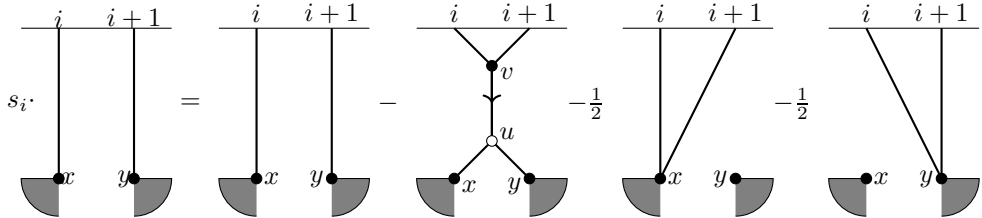
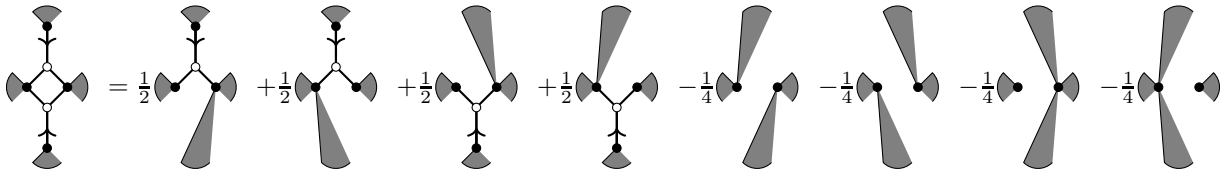
\begin{tikzpicture}[scale = .5]
\draw (-1,0)--(3,0);
\draw[thick] (0,0)--(0,-4);
\draw[fill = gray] (0,-4)--(-1,-4) arc (-180:-90:1);
\draw[thick] (2,0)--(2,-4);
\draw[fill = gray] (2,-4)--(3,-4) arc (0:-90:1);
\node at (0,.25) {$i$};
\node at (2,.25) {$i+1$};
\node at (.25, -4) {$x$};
\node at (1.75, -4) {$y$};
\node at (3.5, -2) {$=$};
\node at (-1,-2) {$s_i \cdot$};
\draw[fill=black] (0,-4) circle (4pt);
\draw[fill=black] (2,-4) circle (4pt);
\end{tikzpicture}
\begin{tikzpicture}[scale = .5]
\draw (-1,0)--(3,0);
\draw[thick] (0,0)--(0,-4);
\draw[fill = gray] (0,-4)--(-1,-4) arc (-180:-90:1);
\draw[thick] (2,0)--(2,-4);
\draw[fill = gray] (2,-4)--(3,-4) arc (0:-90:1);
\node at (0,.4) {$i$};
\node at (2,.4) {$i+1$};
\node at (.4, -4) {$x$};
\node at (1.6, -4) {$y$};
\node at (3.5, -2) {$-$};
\draw[fill=black] (0,-4) circle (4pt);
\draw[fill=black] (2,-4) circle (4pt);
\end{tikzpicture}
\begin{tikzpicture}[scale = .5]
\draw (-1,0)--(3,0);
\draw[thick] (0,0)--(1,-1)--(1,-3)--(0,-4);
\draw[thick, ->] (1,-1)--(1,-2);
\draw[fill = gray] (0,-4)--(-1,-4) arc (-180:-90:1);
\draw[thick] (2,0)--(1,-1)--(1,-3)--(2,-4);
\draw[fill = gray] (2,-4)--(3,-4) arc (0:-90:1);
\node at (0,.4) {$i$};
\node at (2,.4) {$i+1$};
\node at (.4, -4.2) {$x$};
\node at (1.6, -4.2) {$y$};
\node at (1.4, -1.2) {$v$};
\node at (1.4, -2.8) {$u$};
\node at (3.5, -2) {$-\frac{1}{2}$};
\draw[fill=black] (0,-4) circle (4pt);
\draw[fill=black] (2,-4) circle (4pt);
\draw[fill=white] (1,-3) circle (4pt);
\draw[fill=black] (1,-1) circle (4pt);
\end{tikzpicture}
\begin{tikzpicture}[scale = .5]
\draw (-1,0)--(3,0);
\draw[thick] (0,0)--(0,-4);
\draw[fill = gray] (0,-4)--(-1,-4) arc (-180:-90:1);
\draw[thick] (2,0)--(0,-4);
\draw[fill = gray] (2,-4)--(3,-4) arc (0:-90:1);
\node at (0,.4) {$i$};
\node at (2,.4) {$i+1$};
\node at (.4, -4) {$x$};
\node at (1.6, -4) {$y$};
\node at (3.5, -2) {$-\frac{1}{2}$};
\draw[fill=black] (0,-4) circle (4pt);
\draw[fill=black] (2,-4) circle (4pt);
\end{tikzpicture}
\begin{tikzpicture}[scale = .5]
\draw (-1,0)--(3,0);
\draw[thick] (0,0)--(2,-4);
\draw[fill = gray] (0,-4)--(-1,-4) arc (-180:-90:1);
\draw[thick] (2,0)--(2,-4);
\draw[fill = gray] (2,-4)--(3,-4) arc (0:-90:1);
\node at (0,.4) {$i$};
\node at (2,.4) {$i+1$};
\node at (.4, -4) {$x$};
\node at (1.6, -4) {$y$};
\draw[fill=black] (0,-4) circle (4pt);
\draw[fill=black] (2,-4) circle (4pt);
\end{tikzpicture}

\vspace{.2in}
Square Reduction Rule
\vspace{.2in}

\begin{tikzpicture}[scale = .27]
\draw[thick,->] (0,0)--(0,-1);
\draw[thick,->] (0,-6)--(0,-5);
\draw[thick] (0,0)--(0,-2)--(-1,-3)--(0,-4)--(0,-6);
\draw[thick] (0,0)--(0,-2)--(1,-3)--(0,-4)--(0,-6);
\draw[fill = gray] (-1,-3)--(-1.71,-2.29) arc (135:225:1);
\draw[fill = gray] (1,-3)--(1.71,-2.29) arc (45:-45:1);
\draw[fill = gray] (0,0)--(-.71,.71) arc (135:45:1);
\draw[fill = gray] (0,-6)--(-.71,-6.71) arc (225:315:1);
\draw[fill = black] (0,0) circle (6pt);
\draw[fill = white] (0,-2) circle (6pt);
\draw[fill = black] (-1,-3) circle (6pt);
\draw[fill = black] (1,-3) circle (6pt);
\draw[fill = white] (0,-4) circle (6pt);
\draw[fill = black] (0,-6) circle (6pt);
\end{tikzpicture}
\begin{tikzpicture}[scale = .27]
\node at (-3.2,-3) {$=\frac{1}{2}$};
\draw[thick,->] (0,0)--(0,-1);
\draw[thick] (0,0)--(0,-2)--(-1,-3);
\draw[thick] (0,0)--(0,-2)--(1,-3);
\draw[fill = gray] (-1,-3)--(-1.71,-2.29) arc (135:225:1);
\draw[fill = gray] (1,-3)--(1.71,-2.29) arc (45:-45:1);
\draw[fill = gray] (0,0)--(-.71,.71) arc (135:45:1);
\draw[fill = gray] (1,-3)--(-.71,-6.71) arc (225:315:1);
\draw[fill = black] (0,0) circle (6pt);
\draw[fill = white] (0,-2) circle (6pt);
\draw[fill = black] (-1,-3) circle (6pt);
\draw[fill = black] (1,-3) circle (6pt);
\end{tikzpicture}
\begin{tikzpicture}[scale = .27]
\node at (-3,-3) {$+\frac{1}{2}$};
\draw[thick,->] (0,0)--(0,-1);
\draw[thick] (0,0)--(0,-2)--(-1,-3);
\draw[thick] (0,0)--(0,-2)--(1,-3);
\draw[fill = gray] (-1,-3)--(-1.71,-2.29) arc (135:225:1);
\draw[fill = gray] (1,-3)--(1.71,-2.29) arc (45:-45:1);
\draw[fill = gray] (0,0)--(-.71,.71) arc (135:45:1);
\draw[fill = gray] (-1,-3)--(-.71,-6.71) arc (225:315:1);
\draw[fill = black] (0,0) circle (6pt);
\draw[fill = white] (0,-2) circle (6pt);
\draw[fill = black] (-1,-3) circle (6pt);
\draw[fill = black] (1,-3) circle (6pt);
\end{tikzpicture}
\begin{tikzpicture}[scale = .27]
\node at (-3,-3) {$+\frac{1}{2}$};
\draw[thick,->] (0,-6)--(0,-5);
\draw[thick] (-1,-3)--(0,-4)--(0,-6);
\draw[thick] (1,-3)--(0,-4)--(0,-6);
\draw[fill = gray] (-1,-3)--(-1.71,-2.29) arc (135:225:1);
\draw[fill = gray] (1,-3)--(1.71,-2.29) arc (45:-45:1);
\draw[fill = gray] (1,-3)--(-.71,.71) arc (135:45:1);
\draw[fill = gray] (0,-6)--(-.71,-6.71) arc (225:315:1);
\draw[fill = black] (-1,-3) circle (6pt);
\draw[fill = black] (1,-3) circle (6pt);
\draw[fill = white] (0,-4) circle (6pt);
\draw[fill = black] (0,-6) circle (6pt);
\end{tikzpicture}
\begin{tikzpicture}[scale = .27]
\node at (-3,-3) {$+\frac{1}{2}$};
\draw[thick,->] (0,-6)--(0,-5);
\draw[thick] (-1,-3)--(0,-4)--(0,-6);
\draw[thick] (1,-3)--(0,-4)--(0,-6);
\draw[fill = gray] (-1,-3)--(-1.71,-2.29) arc (135:225:1);
\draw[fill = gray] (1,-3)--(1.71,-2.29) arc (45:-45:1);
\draw[fill = gray] (-1,-3)--(-.71,.71) arc (135:45:1);
\draw[fill = gray] (0,-6)--(-.71,-6.71) arc (225:315:1);
\draw[fill = black] (-1,-3) circle (6pt);
\draw[fill = black] (1,-3) circle (6pt);
\draw[fill = white] (0,-4) circle (6pt);
\draw[fill = black] (0,-6) circle (6pt);
\end{tikzpicture}
\begin{tikzpicture}[scale = .27]
\node at (-3,-3) {$-\frac{1}{4}$};
\draw[fill = gray] (-1,-3)--(-1.71,-2.29) arc (135:225:1);
\draw[fill = gray] (1,-3)--(1.71,-2.29) arc (45:-45:1);
\draw[fill = gray] (-1,-3)--(-.71,.71) arc (135:45:1);
\draw[fill = gray] (1,-3)--(-.71,-6.71) arc (225:315:1);
\draw[fill = black] (-1,-3) circle (6pt);
\draw[fill = black] (1,-3) circle (6pt);
\end{tikzpicture}
\begin{tikzpicture}[scale = .27]
\node at (-3,-3) {$-\frac{1}{4}$};
\draw[fill = gray] (-1,-3)--(-1.71,-2.29) arc (135:225:1);
\draw[fill = gray] (1,-3)--(1.71,-2.29) arc (45:-45:1);
\draw[fill = gray] (1,-3)--(-.71,.71) arc (135:45:1);
\draw[fill = gray] (-1,-3)--(-.71,-6.71) arc (225:315:1);
\draw[fill = black] (-1,-3) circle (6pt);
\draw[fill = black] (1,-3) circle (6pt);
\end{tikzpicture}
\begin{tikzpicture}[scale = .27]
\node at (-3,-3) {$-\frac{1}{4}$};
\draw[fill = gray] (-1,-3)--(-1.71,-2.29) arc (135:225:1);
\draw[fill = gray] (1,-3)--(1.71,-2.29) arc (45:-45:1);
\draw[fill = gray] (1,-3)--(-.71,.71) arc (135:45:1);
\draw[fill = gray] (1,-3)--(-.71,-6.71) arc (225:315:1);
\draw[fill = black] (-1,-3) circle (6pt);
\draw[fill = black] (1,-3) circle (6pt);
\end{tikzpicture}
\begin{tikzpicture}[scale = .27]
\node at (-3,-3) {$-\frac{1}{4}$};
\draw[fill = gray] (-1,-3)--(-1.71,-2.29) arc (135:225:1);
\draw[fill = gray] (1,-3)--(1.71,-2.29) arc (45:-45:1);
\draw[fill = gray] (-1,-3)--(-.71,.71) arc (135:45:1);
\draw[fill = gray] (-1,-3)--(-.71,-6.71) arc (225:315:1);
\draw[fill = black] (-1,-3) circle (6pt);
\draw[fill = black] (1,-3) circle (6pt);
\end{tikzpicture}

\vspace{.2in}
Double Edge Reduction Rule
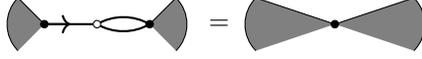
\vspace{.2in}

\begin{tikzpicture}[scale = .7]
\draw[thick,->] (0,0)--(.5,0);
\draw[thick] (0,0) -- (1,0) arc (120:60:1);
\draw[thick] (1,0) arc (-120:-60:1);
\draw[fill=gray] (0,0)--(-.5,.5) arc (135:225:.71);
\draw[fill=gray] (2,0)--(2.5,.5) arc (45:-45:.71);
\draw[fill=black] (0,0) circle (2pt);
\draw[fill=white] (1,0) circle (2pt);
\draw[fill=black] (2,0) circle (2pt);
\end{tikzpicture}
\begin{tikzpicture}[scale = .7]
\node at (-1.2,0) {$=$}; 
\draw[fill=gray] (1,0)--(-.5,.5) arc (135:225:.71);
\draw[fill=gray] (1,0)--(2.5,.5) arc (45:-45:.71);
\draw[fill=black] (1,0) circle (2pt);
\end{tikzpicture}

\vspace{.2in}
Bivalent Vertex Reduction Rule
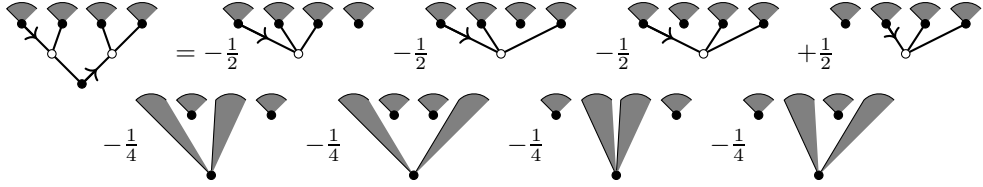
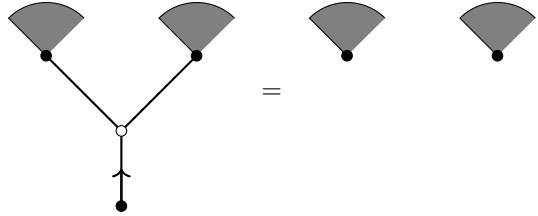
\vspace{.2in}

\begin{tikzpicture}[scale = .4]
\draw[thick] (-.65,2)--(-1,1)--(0,0)--(1,1)--(.65,2);
\draw[thick] (-2,2)--(-1,1);
\draw[thick] (2,2)--(1,1);
\draw[thick,->] (0,0)--(.5,.5);
\draw[thick, ->] (-2,2)--(-1.5,1.5);
\draw[fill=gray] (-2,2)--(-2.5,2.5) arc (135:45:.71);
\draw[fill=gray] (-.65,2)--(-1.15,2.5) arc (135:45:.71);
\draw[fill=gray] (.65,2)--(.15,2.5) arc (135:45:.71);
\draw[fill=gray] (2,2)--(1.5,2.5) arc (135:45:.71);
\draw[fill=black] (-2,2) circle (4pt);
\draw[fill=black] (-.65,2) circle (4pt);
\draw[fill=black] (.65,2) circle (4pt);
\draw[fill=black] (2,2) circle (4pt);
\draw[fill=black] (0,0) circle (4pt);
\draw[fill=white] (1,1) circle (4pt);
\draw[fill=white] (-1,1) circle (4pt);
\end{tikzpicture}
\begin{tikzpicture}[scale = .4]
\node at (-3,1) {$=-\frac{1}{2}$};
\draw[thick] (-.65,2)--(0,1)--(.65,2);
\draw[thick] (-2,2)--(0,1);
\draw[thick, ->] (-2,2)--(-1,1.5);
\draw[fill=gray] (-2,2)--(-2.5,2.5) arc (135:45:.71);
\draw[fill=gray] (-.65,2)--(-1.15,2.5) arc (135:45:.71);
\draw[fill=gray] (.65,2)--(.15,2.5) arc (135:45:.71);
\draw[fill=gray] (2,2)--(1.5,2.5) arc (135:45:.71);
\draw[fill=black] (-2,2) circle (4pt);
\draw[white] (0,0) circle (4pt);
\draw[fill=black] (-.65,2) circle (4pt);
\draw[fill=black] (.65,2) circle (4pt);
\draw[fill=black] (2,2) circle (4pt);
\draw[fill=white] (0,1) circle (4pt);
\end{tikzpicture}
\begin{tikzpicture}[scale = .4]
\node at (-3,1) {$-\frac{1}{2}$};
\draw[thick] (-.65,2)--(0,1);
\draw[thick] (-2,2)--(0,1)--(2,2);
\draw[thick, ->] (-2,2)--(-1,1.5);
\draw[fill=gray] (-2,2)--(-2.5,2.5) arc (135:45:.71);
\draw[fill=gray] (-.65,2)--(-1.15,2.5) arc (135:45:.71);
\draw[fill=gray] (.65,2)--(.15,2.5) arc (135:45:.71);
\draw[fill=gray] (2,2)--(1.5,2.5) arc (135:45:.71);
\draw[fill=black] (-2,2) circle (4pt);
\draw[white] (0,0) circle (4pt);
\draw[fill=black] (-.65,2) circle (4pt);
\draw[fill=black] (.65,2) circle (4pt);
\draw[fill=black] (2,2) circle (4pt);
\draw[fill=white] (0,1) circle (4pt);
\end{tikzpicture}
\begin{tikzpicture}[scale = .4]
\node at (-3,1) {$-\frac{1}{2}$};
\draw[thick] (0,1)--(.65,2);
\draw[thick] (-2,2)--(0,1)--(2,2);
\draw[thick, ->] (-2,2)--(-1,1.5);
\draw[fill=gray] (-2,2)--(-2.5,2.5) arc (135:45:.71);
\draw[fill=gray] (-.65,2)--(-1.15,2.5) arc (135:45:.71);
\draw[fill=gray] (.65,2)--(.15,2.5) arc (135:45:.71);
\draw[fill=gray] (2,2)--(1.5,2.5) arc (135:45:.71);
\draw[fill=black] (-2,2) circle (4pt);
\draw[white] (0,0) circle (4pt);
\draw[fill=black] (-.65,2) circle (4pt);
\draw[fill=black] (.65,2) circle (4pt);
\draw[fill=black] (2,2) circle (4pt);
\draw[fill=white] (0,1) circle (4pt);
\end{tikzpicture}
\begin{tikzpicture}[scale = .4]
\node at (-3,1) {$+\frac{1}{2}$};
\draw[thick] (-.65,2)--(0,1)--(.65,2);
\draw[thick] (2,2)--(0,1);
\draw[thick, ->] (-.65,2)--(-.325,1.5);
\draw[fill=gray] (-2,2)--(-2.5,2.5) arc (135:45:.71);
\draw[fill=gray] (-.65,2)--(-1.15,2.5) arc (135:45:.71);
\draw[fill=gray] (.65,2)--(.15,2.5) arc (135:45:.71);
\draw[fill=gray] (2,2)--(1.5,2.5) arc (135:45:.71);
\draw[fill=black] (-2,2) circle (4pt);
\draw[white] (0,0) circle (4pt);
\draw[fill=black] (-.65,2) circle (4pt);
\draw[fill=black] (.65,2) circle (4pt);
\draw[fill=black] (2,2) circle (4pt);
\draw[fill=white] (0,1) circle (4pt);
\end{tikzpicture}

\begin{tikzpicture}[scale = .4]
\node at (-3,1) {$-\frac{1}{4}$};
\draw[fill=gray] (0,0)--(-2.5,2.5) arc (135:45:.71);
\draw[fill=gray] (-.65,2)--(-1.15,2.5) arc (135:45:.71);
\draw[fill=gray] (0,0)--(.15,2.5) arc (135:45:.71);
\draw[fill=gray] (2,2)--(1.5,2.5) arc (135:45:.71);
\draw[white] (0,0) circle (4pt);
\draw[fill=black] (-.65,2) circle (4pt);
\draw[fill=black] (2,2) circle (4pt);
\draw[fill=black] (0,0) circle (4pt);
\end{tikzpicture}
\begin{tikzpicture}[scale = .4]
\node at (-3,1) {$-\frac{1}{4}$};
\draw[fill=gray] (0,0)--(-2.5,2.5) arc (135:45:.71);
\draw[fill=gray] (-.65,2)--(-1.15,2.5) arc (135:45:.71);
\draw[fill=gray] (.65,2)--(.15,2.5) arc (135:45:.71);
\draw[fill=gray] (0,0)--(1.5,2.5) arc (135:45:.71);
\draw[white] (0,0) circle (4pt);
\draw[fill=black] (-.65,2) circle (4pt);
\draw[fill=black] (.65,2) circle (4pt);
\draw[fill=black] (0,0) circle (4pt);
\end{tikzpicture}
\begin{tikzpicture}[scale = .4]
\node at (-3,1) {$-\frac{1}{4}$};
\draw[fill=gray] (-2,2)--(-2.5,2.5) arc (135:45:.71);
\draw[fill=gray] (0,0)--(-1.15,2.5) arc (135:45:.71);
\draw[fill=gray] (0,0)--(.15,2.5) arc (135:45:.71);
\draw[fill=gray] (2,2)--(1.5,2.5) arc (135:45:.71);
\draw[fill=black] (-2,2) circle (4pt);
\draw[white] (0,0) circle (4pt);
\draw[fill=black] (2,2) circle (4pt);
\draw[fill=black] (0,0) circle (4pt);
\end{tikzpicture}
\begin{tikzpicture}[scale = .4]
\node at (-3,1) {$-\frac{1}{4}$};
\draw[fill=gray] (-2,2)--(-2.5,2.5) arc (135:45:.71);
\draw[fill=gray] (0,0)--(-1.15,2.5) arc (135:45:.71);
\draw[fill=gray] (.65,2)--(.15,2.5) arc (135:45:.71);
\draw[fill=gray] (0,0)--(1.5,2.5) arc (135:45:.71);
\draw[fill=black] (-2,2) circle (4pt);
\draw[white] (0,0) circle (4pt);
\draw[fill=black] (.65,2) circle (4pt);
\draw[fill=black] (0,0) circle (4pt);
\end{tikzpicture}

\vspace{.2in}
Leaf Reduction Rule
\vspace{.2in}

\begin{tikzpicture}
\draw[thick] (0,1)--(0,2)--(-1,3);
\draw[thick,->] (0,1)--(0,1.5);
\draw[thick] (0,2)--(1,3);
\draw[fill = gray] (-1,3)--(-1.5,3.5) arc (135:45:.71);
\draw[fill = gray] (1,3)--(.5,3.5) arc (135:45:.71);
\draw[fill = black] (0,1) circle (2pt);
\draw[fill = white] (0,2) circle (2pt);
\draw[fill = black] (-1,3) circle (2pt);
\draw[fill = black] (1,3) circle (2pt);
\node at (2,2.5) {$=$};
\draw[fill = gray] (3,3)--(2.5,3.5) arc (135:45:.71);
\draw[fill = gray] (5,3)--(4.5,3.5) arc (135:45:.71);
\draw[fill = black] (3,3) circle (2pt);
\draw[fill = black] (5,3) circle (2pt);
\end{tikzpicture}

\vspace{.2in}
Boundary-Adjacent Bivalent and Leaf Reduction Rules
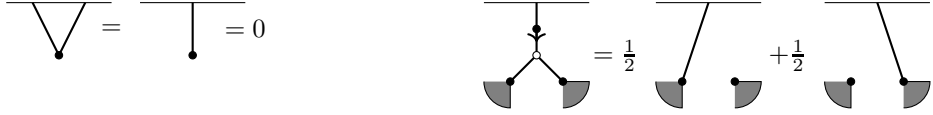
\vspace{.2in}

\begin{tikzpicture}[scale=.7]
\draw (0,0)--(2,0);
\draw[thick] (.5,0)--(1,-1)--(1.5,0);
\draw[fill = black] (1,-1) circle (2pt);
\node at (2,-.5) {$=$};
\draw[white] (0,-2)--(1,-2);
\end{tikzpicture}
\begin{tikzpicture}[scale=.7]
\draw (0,0)--(2,0);
\draw[thick] (1,-1)--(1,0);
\draw[fill = black] (1,-1) circle (2pt);
\node at (2,-.5) {$=0$};
\draw[white] (0,-2)--(1,-2);
\end{tikzpicture}
\hspace{1in}
\begin{tikzpicture}[scale=.7]
\node at (2.5,-1) {$=\frac{1}{2}$};
\draw (0,0)--(2,0);
\draw[thick,->] (1,-.5)--(1,-.75);
\draw[thick] (1.5,-1.5)--(1,-1);
\draw[thick] (.5,-1.5)--(1,-1)--(1,-.5)--(1,0);
\draw[fill = gray] (1.5,-1.5)--(2,-1.5) arc (0:-90:.5);
\draw[fill = gray] (.5,-1.5)--(.5,-2) arc (-90:-180:.5);
\draw[fill = black] (1,-.5) circle (2pt);
\draw[fill = white] (1,-1) circle (2pt);
\draw[fill = black] (.5,-1.5) circle (2pt);
\draw[fill = black] (1.5,-1.5) circle (2pt);
\end{tikzpicture}
\begin{tikzpicture}[scale=.7]
\node at (2.5,-1) {$+\frac{1}{2}$};
\draw (0,0)--(2,0);
\draw[thick] (1,0)--(.5,-1.5);
\draw[fill = gray] (1.5,-1.5)--(2,-1.5) arc (0:-90:.5);
\draw[fill = gray] (.5,-1.5)--(.5,-2) arc (-90:-180:.5);
\draw[fill = black] (.5,-1.5) circle (2pt);
\draw[fill = black] (1.5,-1.5) circle (2pt);
\end{tikzpicture}
\begin{tikzpicture}[scale=.7]
\draw (0,0)--(2,0);
\draw[thick] (1,0)--(1.5,-1.5);
\draw[fill = gray] (1.5,-1.5)--(2,-1.5) arc (0:-90:.5);
\draw[fill = gray] (.5,-1.5)--(.5,-2) arc (-90:-180:.5);
\draw[fill = black] (.5,-1.5) circle (2pt);
\draw[fill = black] (1.5,-1.5) circle (2pt);
\end{tikzpicture}

\end{center}

\begin{proposition}[Crossing reduction rule]
\label{crossing}
Let $W$ be a perfectly orientable normal plabic graph. Suppose $W$ has two adjacent boundary vertices $i, i+1$ which connect to distinct interior vertices $x$ and $y$ respectively. Let $W_{I}$ denote the web obtained from $W$ by removing edges $(x,i)$ and $(y,i+1)$, and attaching an ``I'' shape, i.e. adding an interior black vertex $u$ and an interior white vertex $v$, then adding in edges $(x,v), (y,v) (v,u), (u,i)$ and $(u,i+1)$.  We have the following relation:
\begin{equation}
\label{crosseq}
s_i \cdot [W] = [W] - [W_{I}] - -\frac{1}{2}[W_{x}]  - \frac{1}{2}[W_{y}] 
\end{equation}
\end{proposition}

\begin{proof}
We divide consistent labellings for our webs into classes based on a fixed choice $C$ of labels among edges in these webs other than those between $x,y,u,v,i,$ and $i+1$, then show that within each class, equation~\ref{crosseq} holds. Up to symmetry, there are five possible cases, we will explain the first in detail and give a table for the rest.

Case 1: Among the fixed labels of edges incident to $x$, $2$ and $3$ are present. Among the fixed labels of edges incident to $y$, $1, 2,$ and $3$ are present. The missing labels around the boundary are $1$ and $4$. Then there is exactly one way to label the remaining edges of $W$, edge $(x,i)$ must have label 1 and edge $(y,i+1)$ must have label 4. Call this labelling $\ell$. There are two ways to label $W_{I}$, both with weight $\frac{1}{2}$: edge $(x,u)$ must have label 1, edge $(y,u)$ must be unlabelled, edge $(u,v)$ must have label $\{2,3\}$, and edges $(v,i)$ and $(v,i+1)$ must have labels 1 and 4 in either order. Call these labellings $\ell_{I,1}$ and $\ell_{I,2}$. There are two ways to label $W_{x}$. There are no ways to label $W_y$. Note that origin inversions do not change between these labellings, and the chosen orientation of each web is compatible with the specified labelling, so the relative sign is given only by the relative change in the boundary word. Thus, there exist a fixed monomial $m$ such that

\[
s_i \cdot \sum_{\substack{\ell \in CL(W)\\\ell \textrm{ extends }C}} \textrm{sign}(\mathcal{O}, \mathcal{O}_\ell) \textrm{sign}(\ell, \mathcal{O}_\ell) \textrm{wt}(\ell) \mathbf{x}_{\textrm{bd}(\ell)} = (x_{4,i}x_{1,i+1})m
\]
\[
\sum_{\substack{\ell \in CL(W)\\\ell \textrm{ extends }C}} \textrm{sign}(\mathcal{O}, \mathcal{O}_\ell) \textrm{sign}(\ell, \mathcal{O}_\ell) \textrm{wt}(\ell) \mathbf{x}_{\textrm{bd}(\ell)} = (x_{1,i}x_{4,i+1})m
\]
\[
\sum_{\substack{\ell \in CL(W_I)\\\ell \textrm{ extends }C}} \textrm{sign}(\mathcal{O}, \mathcal{O}_\ell) \textrm{sign}(\ell, \mathcal{O}_\ell) \textrm{wt}(\ell) \mathbf{x}_{\textrm{bd}(\ell)} = \frac{1}{2} (x_{1,i}x_{4,i+1}-x_{4,i}x_{1,i+1})m
\]
\[
\sum_{\substack{\ell \in CL(W_x)\\\ell \textrm{ extends }C}} \textrm{sign}(\mathcal{O}, \mathcal{O}_\ell) \textrm{sign}(\ell, \mathcal{O}_\ell) \textrm{wt}(\ell) \mathbf{x}_{\textrm{bd}(\ell)} = (x_{1,i}x_{4,i+1}-x_{4,i}x_{1,i+1})m
\]
\[
\sum_{\substack{\ell \in CL(W_y)\\\ell \textrm{ extends }C}} \textrm{sign}(\mathcal{O}, \mathcal{O}_\ell) \textrm{sign}(\ell, \mathcal{O}_\ell) \textrm{wt}(\ell) \mathbf{x}_{\textrm{bd}(\ell)} = 0
\]
Therefore, since
\[
(x_{4,i}x_{1,i+1}) = (x_{1,i}x_{4,i+1}) - \frac{1}{2} (x_{1,i}x_{4,i+1}-x_{4,i}x_{1,i+1})-\frac{1}{2} (x_{1,i}x_{4,i+1}-x_{4,i}x_{1,i+1})-0,
\]among classes of labellings which fit into this case, equation~\ref{crosseq} holds.

The remaining cases are as follows. To read the following table, first note that to condense information, we have replaced the monomial $x_{a,i}x_{b,i+1}$ with the word $ab$. Then let $C$ denote a fixed way to label the edges of $W$, $W_I$, $W_x$ and $W_y$ other than those between $x,y,u,v,i$, and $i+1$. The fixed labels of $C$ at $x$, $y$, and the missing labels of $C$ around the boundary fit into one of the cases listed in the rows of this table, up to a permutation of $\{1,2,3\}$ and $\{x,y\}$. Then there is a fixed monomial $m$ such that for each column headed by a web, if the entry in that row and column headed by $s_i \cdot W$ is $a$, then there exists a monomial $m$ such that
\[
\sum_{\substack{\ell \in CL(W)\\\ell \textrm{ extends }C}} \textrm{sign}(\mathcal{O}, \mathcal{O}_\ell) \textrm{sign}(\ell, \mathcal{O}_\ell) \textrm{wt}(\ell) \mathbf{x}_{\textrm{bd}(\ell)} = am
\]
The first row is Case 1.

\begin{center}
\begin{tabular}{|c|c|c||c|c|c|c|c|}
\hline
Labels at $x$ & Labels at $y$ & Boundary& $s_i \cdot W$ & $W$ &$W_{I}$&$W_{x}$ & $W_y$\\
\hline
$\{2,3\}$&$\{1,2,3\}$&$\{1,4\}$& 41&14&$ \frac{1}{2}(14-41)$ & $14-41 $&0\\\hline 
$\{2,3\}$ & $\{2,3\}$ & $\{1,1\}$ & 11 &11 &0&0&0\\\hline
$\{2,3\}$ & $\{1,3\}$ & $\{1,2\}$ & 21 &12 &$12-21$&0&0\\\hline
$\{1,2,3\}$ & $\{1,2,3\}$ & $\{4,5\}$ & $54-45$ &$45-54$& $45-54$&$45-54$&$45-54$ \\\hline
$\{1\}$&$\{1,2,3\}$&$\{2,3\}$ & 0&0&$\frac{1}{2}(21-12)$ & $12-21$&0\\\hline
\end{tabular}

\end{center}

\end{proof}

\begin{proposition}[Square reduction rule]
\label{squarerule}
Let $W$ be a perfectly orientable normal plabic graph. Suppose $W$ has a face of degree 4. Let $v_1$, $v_2$ denote the white vertices of this face, and let $u_1$, $u_2$, $u_3$, $u_4$ denote the neighbors of $v_1, v_2$, connected by edges $(v_1,u_1), (v_1, u_2), (v_1, u_4)$ and $(v_2, u_2), (v_2, u_3), (v_3, u_4)$. Let $\pi$ be a noncrossing set partition of $\{1,2,3,4\}$ in which no block contains both $u_2$ and $u_4$. Let $W_\pi$ be the web obtained from $W$ by first deleting $u_1$ or $u_2$ if they connect to two vertices $v_i, v_j$ whose indices lie in the same block of $\pi$, then identifying all vertices among $v_1, \dots, v_4$ whose indices are in the same block in $\pi$. We will write these set partitions without brackets and with vertical bars between blocks, e.g. $W_{1|23|4}$. Then we have the following relation:
\begin{equation}
\label{square}
\sum_{\pi} \left(-\frac{1}{2}\right)^{4-\# \textrm{ blocks of } \pi} [W_\pi] = 0
\end{equation}
\end{proposition}

\begin{proof}

The proof is similar to that of the crossing rule, but there are many more cases. Again, we explain one case in detail and give a table for the rest.

Let $C$ be a fixed way to label edges of the $W_\pi$ other than those incident to $v_1$ and $v_2$.  Then the boundary monomial of any consistent labelling extending $C$ is fixed, so we need to check that the coefficients satisfy equation~\ref{square}. We proceed casewise, based on the fixed labels of $C$ present at each of the four black vertices $u_1$, $u_2$, $u_3$, and $u_4$. 

Case 1: The fixed labels at $v_1$ contain $1,2,$ and 3. The fixed labels at $v_2$ contain $1$ and 2 but not $3$. The fixed labels at $v_3$ contain 3 but not 1 or 2. The fixed labels at $v_4$ do not contain 1, 2, or 3. 

There are two consistent labellings of $W_{1|2|3|4}$ which extend $C$, as shown below.
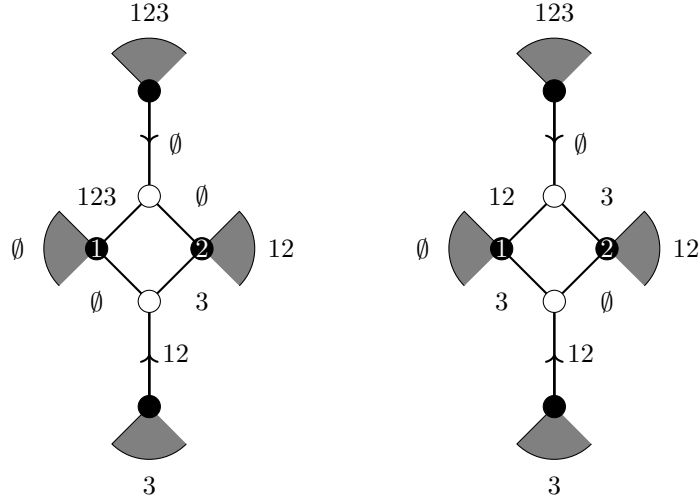
\begin{center}
\begin{tikzpicture}[scale = .7]
\draw[thick,->] (0,0)--(0,-1);
\node at (.5,-1) {$\emptyset$};
\node at (1,-2) {$\emptyset$};
\node at (1,-4) {$3$};
\node at (-1,-2) {$123$};
\node at (.5,-5) {$12$};
\node at (-1,-4) {$\emptyset$};
\node at (0,1.5) {$123$};
\node at (2.5,-3) {$12$};
\node at (-2.5,-3) {$\emptyset$};
\node at (0,-7.5) {$3$};
\draw[thick,->] (0,-6)--(0,-5);
\draw[thick] (0,0)--(0,-2)--(-1,-3)--(0,-4)--(0,-6);
\draw[thick] (0,0)--(0,-2)--(1,-3)--(0,-4)--(0,-6);
\draw[fill = gray] (-1,-3)--(-1.71,-2.29) arc (135:225:1);
\draw[fill = gray] (1,-3)--(1.71,-2.29) arc (45:-45:1);
\draw[fill = gray] (0,0)--(-.71,.71) arc (135:45:1);
\draw[fill = gray] (0,-6)--(-.71,-6.71) arc (225:315:1);
\draw[fill = black] (0,0) circle (6pt);
\draw[fill = white] (0,-2) circle (6pt);
\draw[fill = black] (-1,-3) circle (6pt);
\draw[fill = black] (1,-3) circle (6pt);
\draw[fill = white] (0,-4) circle (6pt);
\draw[fill = black] (0,-6) circle (6pt);
\node[white] at (-1,-3) {1};
\node[white] at (1,-3) {2};
\end{tikzpicture}\hspace{.5in}
\begin{tikzpicture}[scale = .7]
\node at (.5,-1) {$\emptyset$};
\node at (1,-2) {$3$};
\node at (1,-4) {$\emptyset$};
\node at (-1,-2) {$12$};
\node at (.5,-5) {$12$};
\node at (-1,-4) {$3$};
\node at (0,1.5) {$123$};
\node at (2.5,-3) {$12$};
\node at (-2.5,-3) {$\emptyset$};
\node at (0,-7.5) {$3$};
\draw[thick,->] (0,0)--(0,-1);
\draw[thick,->] (0,-6)--(0,-5);
\draw[thick] (0,0)--(0,-2)--(-1,-3)--(0,-4)--(0,-6);
\draw[thick] (0,0)--(0,-2)--(1,-3)--(0,-4)--(0,-6);
\draw[fill = gray] (-1,-3)--(-1.71,-2.29) arc (135:225:1);
\draw[fill = gray] (1,-3)--(1.71,-2.29) arc (45:-45:1);
\draw[fill = gray] (0,0)--(-.71,.71) arc (135:45:1);
\draw[fill = gray] (0,-6)--(-.71,-6.71) arc (225:315:1);
\draw[fill = black] (0,0) circle (6pt);
\draw[fill = white] (0,-2) circle (6pt);
\draw[fill = black] (-1,-3) circle (6pt);
\draw[fill = black] (1,-3) circle (6pt);
\draw[fill = white] (0,-4) circle (6pt);
\draw[fill = black] (0,-6) circle (6pt);
\node[white] at (-1,-3) {1};
\node[white] at (1,-3) {2};
\end{tikzpicture}
\end{center}

Relative to eachother, the left labelling has weight and sign $-\frac{1}{2}$, and the right labelling has weight and sign $-\frac{1}{4}$.

There is one consistent labelling of each of $W_{1|23|4}$, $W_{1|2|34}$, $W_{14|2|3}$, and $W_{14|23}$ which extends $C$, and none for the remaining $W_\pi$.

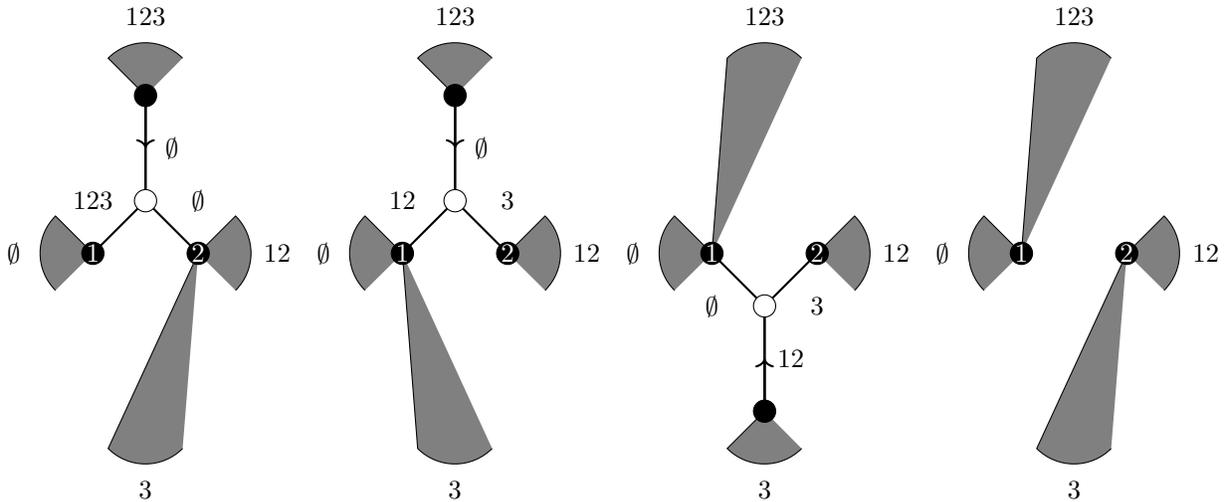
\begin{center}
\begin{tikzpicture}[scale = .7]
\draw[thick,->] (0,0)--(0,-1);
\node at (.5,-1) {$\emptyset$};
\node at (1,-2) {$\emptyset$};
\node at (-1,-2) {$123$};
\node at (0,1.5) {$123$};
\node at (2.5,-3) {$12$};
\node at (-2.5,-3) {$\emptyset$};
\node at (0,-7.5) {$3$};
\draw[thick] (0,0)--(0,-2)--(-1,-3);
\draw[thick] (0,0)--(0,-2)--(1,-3);
\draw[fill = gray] (-1,-3)--(-1.71,-2.29) arc (135:225:1);
\draw[fill = gray] (1,-3)--(1.71,-2.29) arc (45:-45:1);
\draw[fill = gray] (0,0)--(-.71,.71) arc (135:45:1);
\draw[fill = gray] (1,-3)--(-.71,-6.71) arc (225:315:1);
\draw[fill = black] (0,0) circle (6pt);
\draw[fill = white] (0,-2) circle (6pt);
\draw[fill = black] (-1,-3) circle (6pt);
\draw[fill = black] (1,-3) circle (6pt);
\node[white] at (-1,-3) {1};
\node[white] at (1,-3) {2};
\end{tikzpicture}
\begin{tikzpicture}[scale = .7]
\node at (.5,-1) {$\emptyset$};
\node at (1,-2) {$3$};
\node at (-1,-2) {$12$};
\node at (0,1.5) {$123$};
\node at (2.5,-3) {$12$};
\node at (-2.5,-3) {$\emptyset$};
\node at (0,-7.5) {$3$};
\draw[thick,->] (0,0)--(0,-1);
\draw[thick] (0,0)--(0,-2)--(-1,-3);
\draw[thick] (0,0)--(0,-2)--(1,-3);
\draw[fill = gray] (-1,-3)--(-1.71,-2.29) arc (135:225:1);
\draw[fill = gray] (1,-3)--(1.71,-2.29) arc (45:-45:1);
\draw[fill = gray] (0,0)--(-.71,.71) arc (135:45:1);
\draw[fill = gray] (-1,-3)--(-.71,-6.71) arc (225:315:1);
\draw[fill = black] (0,0) circle (6pt);
\draw[fill = white] (0,-2) circle (6pt);
\draw[fill = black] (-1,-3) circle (6pt);
\draw[fill = black] (1,-3) circle (6pt);
\node[white] at (-1,-3) {1};
\node[white] at (1,-3) {2};
\end{tikzpicture}
\begin{tikzpicture}[scale = .7]
\node at (1,-4) {$3$};
\node at (.5,-5) {$12$};
\node at (-1,-4) {$\emptyset$};
\node at (0,1.5) {$123$};
\node at (2.5,-3) {$12$};
\node at (-2.5,-3) {$\emptyset$};
\node at (0,-7.5) {$3$};
\draw[thick,->] (0,-6)--(0,-5);
\draw[thick] (-1,-3)--(0,-4)--(0,-6);
\draw[thick] (1,-3)--(0,-4)--(0,-6);
\draw[fill = gray] (-1,-3)--(-1.71,-2.29) arc (135:225:1);
\draw[fill = gray] (1,-3)--(1.71,-2.29) arc (45:-45:1);
\draw[fill = gray] (-1,-3)--(-.71,.71) arc (135:45:1);
\draw[fill = gray] (0,-6)--(-.71,-6.71) arc (225:315:1);
\draw[fill = black] (-1,-3) circle (6pt);
\draw[fill = black] (1,-3) circle (6pt);
\draw[fill = white] (0,-4) circle (6pt);
\draw[fill = black] (0,-6) circle (6pt);
\node[white] at (-1,-3) {1};
\node[white] at (1,-3) {2};
\end{tikzpicture}
\begin{tikzpicture}[scale = .7]
\node at (0,1.5) {$123$};
\node at (2.5,-3) {$12$};
\node at (-2.5,-3) {$\emptyset$};
\node at (0,-7.5) {$3$};
\draw[fill = gray] (-1,-3)--(-1.71,-2.29) arc (135:225:1);
\draw[fill = gray] (1,-3)--(1.71,-2.29) arc (45:-45:1);
\draw[fill = gray] (-1,-3)--(-.71,.71) arc (135:45:1);
\draw[fill = gray] (1,-3)--(-.71,-6.71) arc (225:315:1);
\draw[fill = black] (-1,-3) circle (6pt);
\draw[fill = black] (1,-3) circle (6pt);
\node[white] at (-1,-3) {1};
\node[white] at (1,-3) {2};
\end{tikzpicture}
\end{center}

Relative to our earlier labellings, these carry weights and sign $-1$, $-\frac{1}{2}$, $\frac{1}{2}$ and $1$, respectively. We check that these satsify equation~\ref{square}:
\[
0 = (-\frac{1}{2}-\frac{1}{4}) + (-\frac{1}{2})(-1+\frac{1}{2} -\frac{1}{2}) + (\frac{1}{4})(1)
\]

The remaining cases our included in the following table, case 1 appears in row 4. Cases which can be obtained by a permutation of $\{1,2,3\}$ or a symmetry of the square are not included.
\begin{center}
\tiny
\begin{tabular}{|c|c|c|c||c|c|c|c|c|c|c|c|c|}
\hline
Labels at $u_1$ &  $u_2$ & $u_3$ & $u_4$ & $W_{1|2|3|4}$ & $W_{1|23|4}$ &$W_{1|2|34}$ &$W_{12|3|4}$ &$W_{14|2|3}$ &$W_{14|23}$ &$W_{12|34}$ &$W_{123|4}$ &$W_{134|2}$  \\\hline
$\{1,2,3\}$ & $\{1,2,3\}$ & $\emptyset$ & $\emptyset$&$ -1$ & $-1$ & $-1$& 0& 1 & 1&0&0&1\\\hline
$\{1,2,3\}$ & $\{1,2\}$ & $\{3\}$ & $\emptyset$& $-\frac{3}{4}$ & $-1$ & $-\frac{1}{2}$& 0& $\frac{1}{2}$ & 1&0&0&0\\\hline
$\{1,2,3\}$ & $\{1,2\}$  & $\emptyset$&$\{3\}$& $-\frac{1}{2}$ & $-\frac{1}{2}$ & $-\frac{1}{2}$& 0& 0 & 0&0&0&0\\\hline
$\{1,2,3\}$ & $\{1\}$  & $\{2,3\}$& $\emptyset$&$-\frac{3}{4}$ & $-1$ &$ -\frac{1}{2}$& 0& $\frac{1}{2}$ & 1&0&0&0\\\hline
$\{1,2,3\}$ & $\{1\}$  & $\{2\}$& $\{3\}$&0 & $-\frac{1}{2}$ & $\frac{1}{2}$& 0& 0 & 0&0&0&0\\\hline
$\{1,2\}$ & $\{1,2,3\}$  & $\{3\}$& $\emptyset$&$-\frac{1}{4}$ & $0$ & $-\frac{1}{2}$& 0& $\frac{1}{2}$ & 0&0&0&1\\\hline
$\{1,2\}$ & $\{1,2,3\}$  & $\emptyset$& $\{3\}$&$-\frac{1}{2}$& $-\frac{1}{2}$ & $-\frac{1}{2}$& 0& $1$ & 1&0&0&1\\\hline
$\{1,2\}$ & $\{1,2\}$  & $\{3\}$& $\{3\}$&$-\frac{1}{4}$& $-\frac{1}{2}$ & 0& 0& $\frac{1}{2}$ & 1&0&0&0\\\hline
$\{1,2\}$ & $\{1,3\}$  & $\{2,3\}$& $\emptyset$&$0$& 0&$-1$& 0& 1 & 0&0&0&0\\\hline
$\{1,2\}$ & $\{1,3\}$  & $\{2\}$& $\{3\}$&$-\frac{3}{4}$& $-\frac{1}{2}$&$-1$& 0& $\frac{1}{2}$ & 1&0&0&0\\\hline
$\{1,2\}$ & $\{1,3\}$  & $\{3\}$& $\{2\}$&$\frac{1}{2}$& 0&$1$& 0& 0 & 0&0&0&0\\\hline
$\{1,2\}$ & $\{1\}$  & $\{2,3\}$& $\{3\}$&$-\frac{3}{4}$& $-\frac{1}{2}$&$0$& 0& $-\frac{1}{2}$ & 1&0&0&0\\\hline
$\{1,2\}$ & $\{3\}$  & $\{1,2\}$& $\{3\}$&$0$& $\frac{1}{2}$&$-\frac{1}{2}$& $-\frac{1}{2}$& $\frac{1}{2}$ & 1&-1&0&0\\\hline
$\{1,2\}$ & $\{3\}$  & $\{1,2\}$& $\{3\}$&$0$& $\frac{1}{2}$&$-\frac{1}{2}$& $-\frac{1}{2}$& $\frac{1}{2}$ & 1&-1&0&0\\\hline
$\{1\}$ & $\{1,2,3\}$  & $\{1\}$& $\{2,3\}$&$0$& $\frac{1}{2}$&$-\frac{1}{2}$& $-\frac{1}{2}$& $\frac{1}{2}$ & 1&-1&0&0\\\hline
$\{1\}$ & $\{1,2,3\}$  & $\{2\}$& $\{3\}$&$\frac{1}{4}$& $0$&$\frac{1}{2}$& 0& $\frac{1}{2}$ & 0&0&0&1\\\hline
$\{1\}$ & $\{1,2,3\}$  & $\emptyset$& $\{2,3\}$&$\frac{1}{2}$& $\frac{1}{2}$&$\frac{1}{2}$& 0& $1$ & 1&0&0&1\\\hline
$\{1\}$ & $\{1,2\}$  & $\{3\}$& $\{2,3\}$&$\frac{1}{4}$& $\frac{1}{2}$&0& 0& $\frac{1}{2}$ & 1&0&0&0\\\hline
$\{1\}$ & $\{2,3\}$  & $\{1\}$& $\{2,3\}$&$0$& $\frac{1}{2}$&$-\frac{1}{2}$& $-\frac{1}{2}$& $\frac{1}{2}$ & 1&-1&0&0\\\hline
$\emptyset$ & $\{1,2,3\}$  & $\emptyset$& $\{1,2,3\}$&1&1&1&1&1&1&1&1&1\\\hline
\end{tabular}

\end{center}
\normalsize
\end{proof}

\begin{proposition}[Double edge reduction rule]
Let $W$ be a perfectly orientable normal plabic graph with a white vertex $v$ with exactly two neighbors $u_1$, connected by 1 edge, and $u_2$, connected by 2 edges. Let $\mathcal{O}$ be a perfect orientation such that $I(\mathcal{O})$ contains $(v,u_1)$. Let $W'$ and $\mathcal{O}'$ be the plabic graph and orientation obtained by contracting $v, u_1,$ and $u_2$. Then
\[
[W, \mathcal{O}] = [W', \mathcal{O}']
\]
\end{proposition}

\begin{proof}
Fix a labelling of edges other than those incident to $v$. We have 4 cases up to a permutation of $\{1,2,3\}$

\begin{itemize}
\item Case 1: Among the fixed labels at $u_1$, $\{1,2,3\}$ appear. Then the union of the labels of the two edges between $v$ and $u_1$ is $\{1,2,3\}$. There are two ways to label these edges with one of size 3, and these come with each with relative sign and weight $-1$. There are 6 ways to label the edges with one label of size 2 and the other of size 1, each with relative sign and weight $frac{1}{2}$, for a total weight and sign of $1$.
\item Case 2: Among the fixed labels at $u_2$, $\{1,2\}$ appear. Then $(v,u_1)$ has label $3$ and the edges between $v$ and $u_2$ have labels $1$ and $2$ split between them. There are two ways to have a label of size 2, these appear with relative sign and weight $-\frac{1}{2}$. There are two ways to have two labels of size 1, these appear with relative sign and weight 1, for a total weight and sign of $1$.
\item Case 3: Among the fixed labels at $u_1$, only $\{1\}$ appears. Then $(v,u_1)$ has label $\{2,3\}$, and one of the edges to $u_2$ has label $\{1\}$. There are two ways to do this, each with relative sign and weight 1.
\item Case 4: The fixed labels at $u_1$ are all empty. Then there is only one way to label the edges incident to $v$, with weight one.
\end{itemize}
In all cases $W'$ has no choices to be made, and thus has relative sign and weight 1, so the result holds.
\end{proof}

\begin{proposition}[Leaf vertex removal]
Let $W$ be a perfectly orientable plabic graph with a black vertex $u$ of degree 1 connected to white vertex $v$. Let $\mathcal{O}$ be a perfect orientation wth the edge $(u,v)$ oriented towards $u$. Let $W'$ be the plabic graph obtained by removing vertices $u$ and $v$ and $\mathcal{O}'$ be the resulting perfect orientation. Then $[W,\mathcal{O}] = [W', \mathcal{O}']$.  
\end{proposition}

\begin{proof}
If we take a consistent labelling of $W$, and remove vertices $u$ and $v$, we get a consistent labelling of $W'$ with the same sign and weight, so the result follows.
\end{proof}

\begin{proposition}[Boundary adjacent leaf removal]
Let $W$ be a perfectly orientable normal plabic graph with a black vertex of degree one or two only connected to the boundary. Then $[W] = 0$.
\end{proposition}

\begin{proof}
There are no consistent labellings of $W$, so the result follows.
\end{proof}

\begin{proposition}[Boundary adjacent bivalent vertex removal]
Let $W$ be a perfectly orientable normal plabic graph with a degree 2 black vertex $u$ connected to one boundary vertex and one white vertex $v$ Let the other neighbors of $v$ be $x$ and $y$. Let $\mathcal{O}$ be a perfect orientation of $W$ in which $(u,v)$ is oriented towards $v$. Let $W_x$ be the graph obtained by removing $u$ and $v$ and connecting $x$ to the boundary, and let $\mathcal{O}_x$ be the orientation obtained from $\mathcal{O}_x$ in the same fashion. Let $W_y$ and $\mathcal{O}_y$ be analogous. Then we have
\[
[W,\mathcal{O}] = \frac{1}{2}[W_x, \mathcal{O}_x]  + \frac{1}{2}[W_y, \mathcal{O}_y]
\]
\end{proposition}

\begin{proof}
A consistent labelling of $W$ must have an edge label of size one on the boundary edge incident to $u$, and edge label of size 2 on the edge $(u,v)$, and an edge label of size one on exactly one of the edges $(u,x)$ and $(u,y)$. Thus, consistent labellings of $W$ are in bijection with the disjoint union of consistent labellings of $W_x$ and $W_y$, and each carries relative sign and weight $\frac{1}{2}$. Thus
\[
[W,\mathcal{O}] = \frac{1}{2}[W_x, \mathcal{O}_x]  + \frac{1}{2}[W_y, \mathcal{O}_y]
\]
as desired.
\end{proof}

\begin{proposition}[Bivalent vertex reduction rule]
Let $W$ be a normal plabic graph with a black vertex $u$ of degree 2 connected to two white vertices $v_1$ and $v_2$. Let $u_1, u_2, u_3, u_4$ be the other neighbors of $v_1$ and $v_2$. Let $\mathcal{O}$ be a perfect orientation such that $(u_1,v_1)$ and $(u,v_2)$ are in $I(\mathcal{O})$. For $1\leq i \leq 8$, let $W_i$ denote the plabic graphs with orientation $\mathcal{O}_i$ as shown in the bivalent vertex reduction rule above. We have
\[
[W, \mathcal{O}] = -\frac{1}{2}[W_1, \mathcal{O}_1]-\frac{1}{2}[W_2, \mathcal{O}_2]-\frac{1}{2}[W_3, \mathcal{O}_3]+\frac{1}{2}[W_4, \mathcal{O}_4] -\frac{1}{4}[W_5, \mathcal{O}_5]-\frac{1}{4}[W_6, \mathcal{O}_6]-\frac{1}{4}[W_7, \mathcal{O}_7]-\frac{1}{4}[W_8, \mathcal{O}_8]
\]
\end{proposition}

\begin{proof}
As per the proof of the square rule, we have four cases to consider as shown in the following table

\begin{center}

\begin{tabular}{|c|c|c|c||c|c|c|c|c|c|c|c|c|}
\hline
Labels at $u_1$ &  $u_2$ & $u_3$ & $u_4$ & $W$ & $W_1$ & $W_2$ & $W_3$ &$W_4$ &$ W_5$ & $W_6$ & $W_7$ & $W_8$  \\\hline
$\{1,2,3\}$ & $\{1,2,3\}$ & $\{1,2,3\}$ & $\emptyset$& 1&0&$-1$&$-1$&1&0&1&1&0\\\hline
$\{1,2,3\}$ & $\{1,2,3\}$ & $\{1,2\}$ & $\{3\}$& $\frac{1}{2}$& 0&0&$-\frac{1}{2}$&$\frac{1}{2}$&0&0&0&0\\\hline
$\{1,2,3\}$ & $\{1,2\}$  & $\{1,3\}$&$\{2,3\}$& $-\frac{1}{2}$ & 0& 0& 0& $-1$ & 0&0&0&0\\\hline
$\{1,2,3\}$ & $\{1,2\}$  & $\{1,2,3\}$& $\{3\}$&$-\frac{1}{4}$ & $0$ &$ \frac{1}{2}$& 0& $-\frac{1}{2}$ & 0&0&-1&0\\\hline
\end{tabular}

\end{center}
\end{proof}

We can now give a proof of Theorem~\ref{basisthm}, which states the set $\{[W, \mathcal{O}_W] \mid W \in AW(n,d)\}$ is a basis of $S^{(d^3 1^{n-3d})}$.

\begin{proof}
Let $W \in AW(n,d)$. Apply the crossing reduction rule to rewrite $s_i \cdot [W, \mathcal{O}_w]$ as a sum of invariants for normal plabic graphs $G_i$ with coefficients $c_i$, i.e.
\[
s_i \cdot [W, \mathcal{O}_w] = \sum_i c_i [G_i, \mathcal{O}_{G_i}]
\]
If any of the $G_i$ are not augmented webs, i.e. they have a face of degree four or a black vertex of degree less than 3, we can apply one of the other skein relations to rewrite $[G_i, \mathcal{O}_{G_i}]$. Each time we do so, we replace a plabic graph with $k$ white vertices by a linear combination of plabic graphs with strictly fewer than $k$ white vertices, so by iterating this process we eventually rewrite $s_i \cdot [W, \mathcal{O}_w]$ as a linear combination of augmented web invariants. Consequently,
\[
\textrm{span}(\{[W, \mathcal{O}_W] \mid W \in AW(n,d)\})
\]
is closed under the action of $\mathfrak{S}_n$. By Theorem~\ref{inspecht} and Proposition~\ref{matchjellyfish}, $\textrm{span}(\{[W, \mathcal{O}_W] \mid W \in AW(n,d)\})$ has a nonzero intersection with $S^{(d^3, 1^{n-3d})}$. By Theorem~\ref{biject}, the dimension of $\textrm{span}(\{[W, \mathcal{O}_W] \mid W \in AW(n,d)\})$ is at most the dimension of $S^{(d^3, 1^{n-3d})}$. Since $S^{(d^3, 1^{n-3d})}$ is irreducible, we have 
\[
\textrm{span}(\{[W, \mathcal{O}_W] \mid W \in AW(n,d)\}) = S^{(d^3, 1^{n-3d})}
\]
and Theorem~\ref{biject} shows that $\{[W, \mathcal{O}_W] \mid W \in AW(n,d)\}$ is indeed a basis.
\end{proof}

\section{Augmented web invariants via weblike subgraphs}

In this section we explain how to interpret our augmented web invariants in terms of the weblike subgraphs introduced by T. Lam in \cite{Lam}. To do so, we first need to reinterpet out augmented web invariants as tensors rather than polynomials in $n\times \nu$ variables.

Let $V_0\cong \mathbb{C}^\nu$ be a $\nu$ dimensional vector space with basis $\{e_1,\dots, e_\nu\}$, let $V = \textrm{span}(\{e_1,e_2,e_3\})$. Consider the space $W$ consisting of the direct sum of all tensor products of $3d$ copies of $V$ and $n-3d$ copies of $\mathbb{C}$, e.g. when $n=4$ and $d=1$, $W$ is 
\[
\left(V \otimes V \otimes V \otimes \mathbb{C} \right) \oplus 
\left(V \otimes V \otimes  \mathbb{C}  \otimes V \right)\oplus\left( V  \otimes \mathbb{C}\otimes V \otimes V \right) \oplus \left(\mathbb{C}\otimes V \otimes V \otimes V\right)
\]
$W$ injects into $V_0^{\otimes n}$ by replacing the $n-3d$ copies of $\mathbb{C}$ with $\bigwedge^{n-3d} \textrm{span}(\{e_4, \dots, e_\nu\})$, e.g. if $v_1, v_2, v_3 \in V$, 
\[
1 \otimes v_1 \otimes v_2 \otimes 1 \otimes v_3 \mapsto e_4 \otimes v_1 \otimes v_2 \otimes e_5 \otimes v_3 - e_5 \otimes v_1 \otimes v_2 \otimes e_4 \otimes v_3
\]
There is also a natural injection of $V_0^{\otimes n}$ into the polynomial ring generated by the $n \times \nu$ variables
\[
\begin{bmatrix}
x_{1,1} & x_{1,2} & \cdots & x_{1,n}\\
x_{2,1} & x_{2,2} & \cdots & x_{2,n}\\
\vdots &  &\ddots& \vdots\\
x_{\nu, 1} &x_{\nu,2} & \cdots &x_{\nu, n}
\end{bmatrix}
\]
given by \[
e_{i_1} \otimes e_{i_2} \otimes \cdots \otimes e_{i_n} \mapsto x_{i_1,1}x_{i_2,2}\cdots x_{i_n,n} 
\]

Recall that our augmented web invariants live in this polynomials ring, and 
furthermore, they live in the span of monomials which contain exactly one variable from each column, each with degree one. Additionally, in each augmented web invariant, the tensor factors corresponding to basis vectors $e_4, \dots, e_{\nu}$ are alternating.  Thus, augmented web invariants live in the image of the injection $\iota: W \hookrightarrow \mathbb{C}[x_{1,1}, \dots, x_{\nu, n}]$. Denote the preimage under this injection of $[W, \mathcal{O}]$ by $\widetilde{[W,\mathcal{O}]}$.

We can make $\iota$ into an $\mathfrak{S}_n$ homomorphism by pulling back the action of $\mathfrak{S}_n$ on  $\mathbb{C}[x_{1,1}, \dots, x_{\nu, n}]$ to $W$. Note that this pullback is {\em not} just simply permuting tensor factors. An adjacent transposition $s_i$ acts on simple basis tensors by 
\[
s_i\cdot (v_1 \otimes \cdots \otimes v_{i} \otimes v_{i+1} \otimes \cdots \otimes v_n) = \begin{cases}-s_i\cdot (v_1 \otimes \cdots \otimes v_{i} \otimes v_{i+1} \otimes \cdots \otimes v_n) & v_1, v_{i+1} \in \mathbb{C}\\ s_i\cdot (v_1 \otimes \cdots \otimes v_{i} \otimes v_{i+1} \otimes \cdots \otimes v_n) & \textrm{otherwise} \end{cases},
\] 
i.e. it picks up a sign if both tensor factors come from $\mathbb{C}$.

The benefit of this viewpoint is that $W$ is more well-studied in terms of webs. Kuperberg's work \cite{Kuperberg} gives a basis for $W$ in terms of $SL_3$ webs with 0 clasps, i.e. $SL_3$ webs with $n-3d$ boundary vertices without edges. We will explain how to expand $\widetilde{[W,\mathcal{O}]}$ into this clasped web basis.

As introduced by Lam, given a normal plabic graph $W$ a 3-weblike subgraph is an assignment of a nonnegative integer to each edge such that the sum of edges around each interior vertex is 3. A weblike subgraph can be turned into an $SL_3$ web with $0$-clasps (i.e. unused boundary vertices) by deleting each edge assigned 0 or 3, and contracting each path of edges alternately assigned 1's and 2's to a single edge. A consistent labelling $\ell$ gives rise to a weblike subgraph $W'(\ell)$ via the sizes of its edge labels. We have the following

\begin{proposition}
Let $W$ be a normal plabic graph with perfect orientation $\mathcal{O}$. Let $W'$ be a 3-weblike subgraph of $W$ with $d(W')$ edges of multiplicity $2$. Then
\[
\iota^{-1}( \sum_{\substack{\ell \in CL(G)\\W'(\ell) = W'}} \textrm{sign}(\ell, \mathcal{O}) \textrm{wt}(\ell) \mathbf{x}_{\textrm{bd}(\ell)}) = \pm(-\frac{1}{2})^{d(W')} [W']_{SL_3}
 \]
 where $[W']_{SL_3}$ denotes the usual $SL_3$ web invariant.  
\end{proposition}

\begin{proof}
Up to a reordering of labels larger than 3, consistent labellings $\ell$ with $W'(\ell) = W'$ are in bijection with proper edge labellings of $W'$. So it suffices to check that the difference in the definition of sign for consistent labellings and proper edge labellings is the same among all such labellings. Any two proper edge labellings of $W'$ can be transformed into eachother via swapping the labels of any path alternately labelled $i,j,\dots,i,j$ for $1 \leq i \leq j$. Swapping such a path will introduce a sign change both in the definition of sign for proper edge labellings and for consistent labellings.
\end{proof}

\begin{corollary}
We thus have
\[
\widetilde{[W, \mathcal{O}]} = \sum_{\textrm{weblike subgraphs } W' \textrm{ of } W} \pm (\frac{1}{2})^{d(W')}[W']_{SL_3}
\]
\end{corollary}

\begin{remark}
Note that when a normal plabic graph is an $SL_3$ web, i.e. all vertices are degree 3, our invariants do not match the usual $SL_3$ invariants. Instead, we get a sum over all $SL_3$ webs which appear as a weblike subgraph in $W$.
\end{remark}

\section{Cyclic sieving}

The basis $\{[W, \mathcal{O}_W] \mid W \in AW(n,d)\}$ given in Theorem~\ref{basis} has all the necessary properties to obtain a cyclic sieving result via Springer's theorem of regular elements. The only detail left is that we need to be careful about the orientations. To address orientation, we need the following Lemma:

\begin{lemma}
\label{rotatorient}
Let $W \in AW(n,d)$ with perfect orientation $\mathcal{O}$. Suppose $W$ is fixed by rotation by $i \geq 2$, i.e. $\textrm{rot}^i(W) = W$. Then $i$ divides $n$ and exactly one of the following holds:
\begin{itemize}
\item $\frac{n}{i} \mid d$
\item $\frac{n}{i} \mid d-1$
\item $\frac{n}{i}=3 \textrm{ and } \frac{n}{i} \mid d+1$
\end{itemize}
Let $k = \frac{di}{n}, \frac{(d-1)i}{n}, \frac{(d+1)i}{n}$ depending on which of the cases above holds. Then we have
\[
\textrm{sign}(\mathcal{O}, \textrm{rot}^i(\mathcal{O})) = (-1)^{(\frac{n}{i}-1)k}
\]
The relevance here is that this sign depends only on $n, d,$ and $i$, not on $W$ itself.
\end{lemma}

\begin{proof}
Note that by Lemma~\ref{signagree} it suffices to prove this for some orientation $\mathcal{O}$. We proceed by induction on the number of interior white vertices. If there are no interior white vertices, then a perfect orientation is simply a total order on the black vertices. Rotation by $i$ induces a permutation on the black vertices with at most one cycle of size 1 and cycles of size $\frac{n}{i}$, and thus $\frac{n}{i} \mid d$ or $\frac{n}{i} \mid d-1$. Thus, 
\[
\textrm{sign}(\mathcal{O}, \textrm{rot}^i(\mathcal{O})) = (-1)^{(\frac{n}{i} -1) k }
\]

If there is a single white vertex $v$, it is necessarily fixed by rotation by $i$, and thus $\frac{n}{i}=3$. Rotation by $i$ induces a permutation of the $d+1$ black vertices into cycles of size $3$, and thus $3 \mid d+1$. The orientation $\mathcal{O}_W$ differs from $\textrm{rot}^i (W)$ via a swivel move at $v$, a transposition of the two sink vertices adjacent to $v$, and a 3-cycle applied to each other orbit of sinks. Thus,
\[
\textrm{sign}(\mathcal{O}, \textrm{rot}^i(\mathcal{O})) = 1 = (-1)^{(\frac{n}{i} -1) k }
\]

If there is more than one white vertex, then by Lemma~\ref{induct} and our rotation invariance assumption, we can find three black vertices each connected to exactly one interior white vertex such that these three white vertices are distinct. Remove these 6 vertices and connect their neighbors to the boundary in a planar and rotationally invariant way to get a web $W'$. From any perfect orientation $\mathcal{O}'$ of $W'$ we can build a perfect orientation $\mathcal{O}$ of $W$ by orienting the removed edges from white vertex to black vertex. We then have by inductive hypothesis
\[
\textrm{sign}(\mathcal{O}, \textrm{rot}^i(\mathcal{O})) = \textrm{sign}(\mathcal{O'}, \textrm{rot}^i(\mathcal{O'})) = (-1)^{(\frac{n}{i} -1) k }.
\]
So the result follows by induction.
\end{proof}

We can now state our cyclic sieving result.

\begin{theorem}
Let $C= \mathbb{Z}/n\mathbb{Z}$ be the cyclic group with generator $c$ acting on $AW(n,d)$ by rotation. Let $X_{n,d}(q)$ be the fake degree polynomial for $S^(d^3, 1^{n-3d})$, i.e.
\[
X_{n,d}(q) = q^{3(d-1)+\binom{n-3(d-1)}{2}}\frac{[n]!_q}{\prod_{(i,j) \in \lambda} [h_{ij}]_q}
\]

If $n$ is odd, then the triple
\[
(AW(n,d), C, X_{n,d}(q))
\]
exhibits the cyclic sieving phenomenon.

If $n$ is even, then we have cyclic sieving up to sign, i.e.
\[
|AW(n,d)^{c^i}| = |X_{n,d}(\zeta^i)|
\]
where $AW(n,d)^{c^i}$ dentoes the fixed point set of $AW(n,d)$ under the action of $c^i$, and $\zeta$ is a primitive $n^{th}$ root of unity.
\end{theorem}

\begin{proof}
If $n$ is odd, then we can choose orientations $\mathcal{O}_W$ for each web $W \in AW(n,d)$ such that 
\[
c \cdot [W, \mathcal{O}_W] = [\textrm{rot}(W), \mathcal{O}_{\textrm{rot}(W)}]
\]
To do so, select a web $W$ from each $C$-orbit and pick any orientation $\mathcal{O}_W$ for it. For $1 \leq i \leq n$, let $\mathcal{O}_{\textrm{rot}^i(W)} = \textrm{rot}^i(\mathcal{O}_W)$. Lemma~\ref{rotatorient} guarantees that this is possible even if $W$ has rotational symmetry, as $(-1)^{(\frac{n}{i}-1)k} = 1$.

Thus, $S^{(d^3, 1^{n-3d})}$, $\{[W, \mathcal{O}_W] \mid W \in AW(n,d)$ and the rotation action of $C$ satisfy the hypotheses of Theorem~\ref{sieving} and the result follows.

If $n$ is even, choose any orientation for each web $W \in AW(n,d)$. Then $\{[W, \mathcal{O}_W] \mid W \in AW(n,d)$ is not necessarily fixed by the action of $c$, but $c$ will act via a {\em signed} permutation matrix. Lemma~\ref{rotatorient} shows that the diagonal of $c^{i}$ will either contain only $0$'s and $1$'s or only $0$'s and $-1$'s. In either case, $|tr(c^i)| = |AW(n,d)^{c^i}|$, and the proof of Theorem~\ref{sieving} \cite{Sagansurvey} shows that 
\[
|AW(n,d)^{c^i}| = |X_{n,d}(\zeta^i)|
\]
holds as desired.
\end{proof}

\begin{example}
When $n=10$, $d=3$, then $X_{10,3}(q)$ has a rather nice form. The hook lengths are 
\begin{center}
\begin{ytableau}
6&4&3\\
5&3&2\\
4&2&1\\
1
\end{ytableau}
\end{center}
and thus $X_{10,3}(q) = q^{12} \begin{bmatrix} 10\\4 \end{bmatrix}_q$. Since $X_{10,3}(q)$ is a single $q$-binomial ($n=10, d=3$ is the only case for which this is true), we can verify the cyclic sieving result in this case by checking that the orbits of $AW(10,3)$ under rotation are in size-preserving bijection with the orbits of size 4 subsets of $\{1,\dots, 10\}$ under cyclic permutation. There are two orbits of size 5 for each set, the orbits containing sequences $\{1,2,6,7\}$ and $\{1,3,6,8\}$ and the orbits containing webs shown below
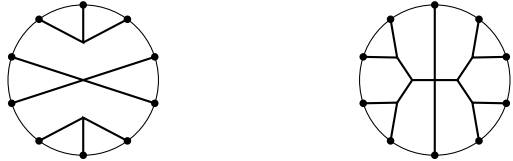
\begin{center}
\begin{tikzpicture}
\equiccnolabels[1cm]{10}{0}{0}
\coordinate (p1) at (0,0);
\coordinate (p2) at (0,.5);
\coordinate (p3) at (0,-.5);

\draw[thick] (N8)--(p1)--(N7);
\draw[thick] (N2)--(p1)--(N3);
\draw[thick] (N9)--(p2)--(N10);
\draw[thick] (N1)--(p2);
\draw[thick] (N4)--(p3)--(N5);
\draw[thick] (p3)--(N6);
\end{tikzpicture}\hspace{1in}
\begin{tikzpicture}
\equiccnolabels[1cm]{10}{0}{0}
\coordinate (p1) at (0,0);
\coordinate (p2) at (-.3,0);
\coordinate (p3) at (.3,0);
\coordinate (p4) at (-.5,.3);
\coordinate (p5) at (-.5,-.3);
\coordinate (p6) at (.5,.3);
\coordinate (p7) at (.5,-.3);

\draw[thick] (N10)--(p1)--(N5);
\draw[thick] (N9)--(p4)--(N8);
\draw[thick] (N7)--(p5)--(N6);
\draw[thick] (N4)--(p7)--(N3);
\draw[thick] (N1)--(p6)--(N2);
\draw[thick] (p6)--(p3)--(p7);
\draw[thick] (p4)--(p2)--(p5);
\draw[thick] (p3)--(p1)--(p2);
\end{tikzpicture}
\end{center}

The remaining 20 orbits are all size 10, one web from each is shown below

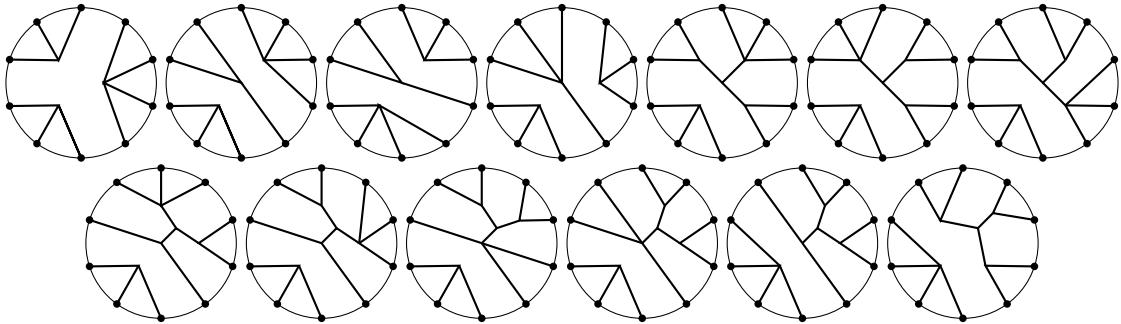
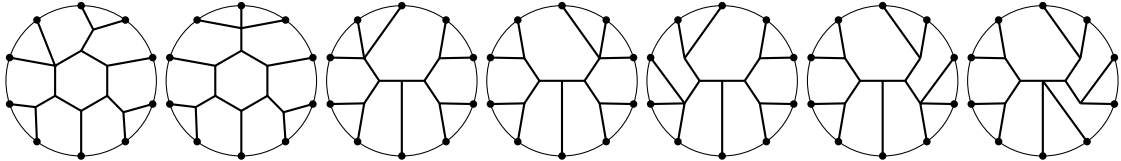
\begin{center}
\begin{tikzpicture}
\equiccnolabels[1cm]{10}{0}{0}
\coordinate (p1) at (.3,0);
\coordinate (p2) at (-.3,-.3);
\coordinate (p3) at (-.3,.3);

\draw[thick] (N10)--(p3)--(N9);
\draw[thick] (N8)--(p3);
\draw[thick] (N7)--(p2)--(N6);
\draw[thick] (N5)--(p2)--(N5);
\draw[thick] (N4)--(p1)--(N3);
\draw[thick] (N1)--(p1)--(N2);
\end{tikzpicture}
\begin{tikzpicture}
\equiccnolabels[1cm]{10}{0}{0}
\coordinate (p1) at (0,0);
\coordinate (p2) at (-.3,-.3);
\coordinate (p3) at (.3,.3);

\draw[thick] (N8)--(p1)--(N9);
\draw[thick] (N4)--(p1);
\draw[thick] (N7)--(p2)--(N6);
\draw[thick] (N5)--(p2)--(N5);
\draw[thick] (N10)--(p3)--(N3);
\draw[thick] (N1)--(p3)--(N2);
\end{tikzpicture}
\begin{tikzpicture}
\equiccnolabels[1cm]{10}{0}{0}
\coordinate (p1) at (0,0);
\coordinate (p2) at (-.3,-.3);
\coordinate (p3) at (.3,.3);

\draw[thick] (N8)--(p1)--(N9);
\draw[thick] (N3)--(p1);
\draw[thick] (N7)--(p2)--(N6);
\draw[thick] (N5)--(p2)--(N4);
\draw[thick] (N10)--(p3);
\draw[thick] (N1)--(p3)--(N2);
\end{tikzpicture}
\begin{tikzpicture}
\equiccnolabels[1cm]{10}{0}{0}
\coordinate (p1) at (0,0);
\coordinate (p2) at (-.3,-.3);
\coordinate (p3) at (.5,0);

\draw[thick] (N8)--(p1)--(N9);
\draw[thick] (N10)--(p1)--(N4);
\draw[thick] (N3)--(p3);
\draw[thick] (N7)--(p2)--(N6);
\draw[thick] (N5)--(p2);
\draw[thick] (N1)--(p3)--(N2);
\end{tikzpicture}
\begin{tikzpicture}
\equiccnolabels[1cm]{10}{0}{0}
\coordinate (p1) at (-.3,-.3);

\coordinate (p2) at (0,0);
\coordinate (p3) at (-.3,.3);
\coordinate (p4) at (.3,.3);
\coordinate (p5) at (.3,-.3);

\draw[thick] (N8)--(p3)--(N9);
\draw[thick] (N10)--(p4)--(N1);
\draw[thick] (N2)--(p4);
\draw[thick] (N3)--(p5)--(N4);
\draw[thick] (p3)--(p2)--(p4);
\draw[thick] (p5)--(p2);
\draw[thick] (N7)--(p1)--(N6);
\draw[thick] (N5)--(p1);

\end{tikzpicture}
\begin{tikzpicture}
\equiccnolabels[1cm]{10}{0}{0}
\coordinate (p1) at (-.3,-.3);

\coordinate (p2) at (0,0);
\coordinate (p3) at (-.3,.3);
\coordinate (p4) at (.3,.3);
\coordinate (p5) at (.3,-.3);

\draw[thick] (N8)--(p3)--(N9);
\draw[thick] (N10)--(p3);
\draw[thick] (p4)--(N1);
\draw[thick] (N2)--(p4);
\draw[thick] (N3)--(p5)--(N4);
\draw[thick] (p3)--(p2)--(p4);
\draw[thick] (p5)--(p2);

\draw[thick] (N7)--(p1)--(N6);
\draw[thick] (N5)--(p1);

\end{tikzpicture}
\begin{tikzpicture}
\equiccnolabels[1cm]{10}{0}{0}
\coordinate (p1) at (-.3,-.3);

\coordinate (p2) at (0,0);
\coordinate (p3) at (-.3,.3);
\coordinate (p4) at (.3,.3);
\coordinate (p5) at (.3,-.3);

\draw[thick] (N8)--(p3)--(N9);
\draw[thick] (N2)--(p5);
\draw[thick] (p4)--(N1);
\draw[thick] (N10)--(p4);
\draw[thick] (N3)--(p5)--(N4);
\draw[thick] (p3)--(p2)--(p4);
\draw[thick] (p5)--(p2);

\draw[thick] (N7)--(p1)--(N6);
\draw[thick] (N5)--(p1);

\end{tikzpicture}
\begin{tikzpicture}
\equiccnolabels[1cm]{10}{0}{0}
\coordinate (p1) at (-.3,-.3);

\coordinate (p2) at (.2,.2);
\coordinate (p3) at (0,0);
\coordinate (p4) at (.5,0);
\coordinate (p5) at (0,.5);

\draw[thick] (N8)--(p3)--(N4);
\draw[thick] (N9)--(p5);
\draw[thick] (N10)--(p5)--(N1);
\draw[thick] (N3)--(p4)--(N2);
\draw[thick] (p3)--(p2)--(p4);
\draw[thick] (p5)--(p2);

\draw[thick] (N7)--(p1)--(N6);
\draw[thick] (N5)--(p1);

\end{tikzpicture}
\begin{tikzpicture}
\equiccnolabels[1cm]{10}{0}{0}
\coordinate (p1) at (-.3,-.3);

\coordinate (p2) at (.2,.2);
\coordinate (p3) at (0,0);
\coordinate (p4) at (.5,0);
\coordinate (p5) at (0,.5);

\draw[thick] (N8)--(p3)--(N4);
\draw[thick] (N10)--(p5)--(N9);
\draw[thick] (N3)--(p4)--(N2);
\draw[thick] (N1)--(p4);
\draw[thick] (p3)--(p2)--(p4);
\draw[thick] (p5)--(p2);

\draw[thick] (N7)--(p1)--(N6);
\draw[thick] (N5)--(p1);

\end{tikzpicture}
\begin{tikzpicture}
\equiccnolabels[1cm]{10}{0}{0}
\coordinate (p1) at (-.3,-.3);

\coordinate (p2) at (.2,.2);
\coordinate (p3) at (0,0);
\coordinate (p4) at (.5,.3);
\coordinate (p5) at (0,.5);

\draw[thick] (N8)--(p3)--(N4);
\draw[thick] (N10)--(p5)--(N9);
\draw[thick] (N1)--(p4)--(N2);
\draw[thick] (N3)--(p3);
\draw[thick] (p3)--(p2)--(p4);
\draw[thick] (p5)--(p2);

\draw[thick] (N7)--(p1)--(N6);
\draw[thick] (N5)--(p1);

\end{tikzpicture}
\begin{tikzpicture}
\equiccnolabels[1cm]{10}{0}{0}
\coordinate (p1) at (-.3,-.3);

\coordinate (p2) at (.2,.2);
\coordinate (p3) at (0,0);
\coordinate (p4) at (.5,0);
\coordinate (p5) at (.3,.5);

\draw[thick] (N8)--(p3)--(N4);
\draw[thick] (N10)--(p5)--(N1);
\draw[thick] (N3)--(p4)--(N2);
\draw[thick] (N9)--(p3);
\draw[thick] (p3)--(p2)--(p4);
\draw[thick] (p5)--(p2);

\draw[thick] (N7)--(p1)--(N6);
\draw[thick] (N5)--(p1);

\end{tikzpicture}
\begin{tikzpicture}
\equiccnolabels[1cm]{10}{0}{0}
\coordinate (p1) at (-.3,-.3);

\coordinate (p2) at (.2,.2);
\coordinate (p3) at (0,0);
\coordinate (p4) at (.5,0);
\coordinate (p5) at (.3,.5);

\draw[thick] (p3)--(N4);
\draw[thick] (N10)--(p5)--(N1);
\draw[thick] (N3)--(p4)--(N2);
\draw[thick] (N9)--(p3);
\draw[thick] (p3)--(p2)--(p4);
\draw[thick] (p5)--(p2);

\draw[thick] (N7)--(p1)--(N6);
\draw[thick] (N5)--(p1)--(N8);

\end{tikzpicture}
\begin{tikzpicture}
\equiccnolabels[1cm]{10}{0}{0}
\coordinate (p1) at (-.3,-.3);

\coordinate (p2) at (.2,.2);
\coordinate (p3) at (-.3,.3);
\coordinate (p4) at (.3,-.3);
\coordinate (p5) at (.4,.4);

\draw[thick] (N10)--(p3)--(N9);
\draw[thick] (N1)--(p5)--(N2);
\draw[thick] (N3)--(p4)--(N4);
\draw[thick] (p3)--(p2)--(p4);
\draw[thick] (p5)--(p2);

\draw[thick] (N7)--(p1)--(N6);
\draw[thick] (N5)--(p1)--(N8);

\end{tikzpicture}

\begin{tikzpicture}
\equiccnolabels[1cm]{10}{0}{0}
 \foreach \i in {1,...,6} {
    \coordinate (p\i) at ($(90-\i*360/6:.4)$);}

\foreach \i in {1,...,3} {
    \coordinate (q\i) at ($(190-\i*340/3:.7)$);}

\draw[thick] (N10)--(q1)--(N1);
\draw[thick] (N2)--(p1);
\draw[thick] (N3)--(q2)--(N4);
\draw[thick] (N5)--(p3);
\draw[thick] (N6)--(q3)--(N7);
\draw[thick] (N8)--(p5)--(N9);

\draw[thick] (q1)--(p6);
\draw[thick] (q2)--(p2);
\draw[thick] (q3)--(p4);
\draw[thick] (p1)--(p2)--(p3)--(p4)--(p5)--(p6)--(p1);

\end{tikzpicture}
\begin{tikzpicture}
\equiccnolabels[1cm]{10}{0}{0}
 \foreach \i in {1,...,6} {
    \coordinate (p\i) at ($(90-\i*360/6:.4)$);}

\foreach \i in {1,...,3} {
    \coordinate (q\i) at ($(190-\i*340/3:.7)$);}

\coordinate (p7) at (0,.7);

\draw[thick] (N10)--(p7)--(N1);
\draw[thick] (N9)--(p7);
\draw[thick] (N2)--(p1);
\draw[thick] (N3)--(q2)--(N4);
\draw[thick] (N5)--(p3);
\draw[thick] (N6)--(q3)--(N7);
\draw[thick] (N8)--(p5);

\draw[thick] (p7)--(p6);
\draw[thick] (q2)--(p2);
\draw[thick] (q3)--(p4);
\draw[thick] (p1)--(p2)--(p3)--(p4)--(p5)--(p6)--(p1);

\end{tikzpicture}
\begin{tikzpicture}
\equiccnolabels[1cm]{10}{0}{0}
\coordinate (p1) at (0,0);
\coordinate (p2) at (-.3,0);
\coordinate (p3) at (.3,0);
\coordinate (p4) at (-.5,.3);
\coordinate (p5) at (-.5,-.3);
\coordinate (p6) at (.5,.3);
\coordinate (p7) at (.5,-.3);

\draw[thick] (N10)--(p4);
\draw[thick] (p1)--(N5);
\draw[thick] (N9)--(p4)--(N8);
\draw[thick] (N7)--(p5)--(N6);
\draw[thick] (N4)--(p7)--(N3);
\draw[thick] (N1)--(p6)--(N2);
\draw[thick] (p6)--(p3)--(p7);
\draw[thick] (p4)--(p2)--(p5);
\draw[thick] (p3)--(p1)--(p2);
\end{tikzpicture}
\begin{tikzpicture}
\equiccnolabels[1cm]{10}{0}{0}
\coordinate (p1) at (0,0);
\coordinate (p2) at (-.3,0);
\coordinate (p3) at (.3,0);
\coordinate (p4) at (-.5,.3);
\coordinate (p5) at (-.5,-.3);
\coordinate (p6) at (.5,.3);
\coordinate (p7) at (.5,-.3);

\draw[thick] (N10)--(p6);
\draw[thick] (p1)--(N5);
\draw[thick] (N9)--(p4)--(N8);
\draw[thick] (N7)--(p5)--(N6);
\draw[thick] (N4)--(p7)--(N3);
\draw[thick] (N1)--(p6)--(N2);
\draw[thick] (p6)--(p3)--(p7);
\draw[thick] (p4)--(p2)--(p5);
\draw[thick] (p3)--(p1)--(p2);
\end{tikzpicture}
\begin{tikzpicture}
\equiccnolabels[1cm]{10}{0}{0}
\coordinate (p1) at (0,0);
\coordinate (p2) at (-.3,0);
\coordinate (p3) at (.3,0);
\coordinate (p4) at (-.5,.3);
\coordinate (p5) at (-.5,-.3);
\coordinate (p6) at (.5,.3);
\coordinate (p7) at (.5,-.3);

\draw[thick] (N10)--(p4);
\draw[thick] (p1)--(N5);
\draw[thick] (N9)--(p4);
\draw[thick] (N8)--(p5);
\draw[thick] (N7)--(p5)--(N6);
\draw[thick] (N4)--(p7)--(N3);
\draw[thick] (N1)--(p6)--(N2);
\draw[thick] (p6)--(p3)--(p7);
\draw[thick] (p4)--(p2)--(p5);
\draw[thick] (p3)--(p1)--(p2);
\end{tikzpicture}
\begin{tikzpicture}
\equiccnolabels[1cm]{10}{0}{0}
\coordinate (p1) at (0,0);
\coordinate (p2) at (-.3,0);
\coordinate (p3) at (.3,0);
\coordinate (p4) at (-.5,.3);
\coordinate (p5) at (-.5,-.3);
\coordinate (p6) at (.5,.3);
\coordinate (p7) at (.5,-.3);

\draw[thick] (N10)--(p6);
\draw[thick] (p1)--(N5);
\draw[thick] (N9)--(p4)--(N8);
\draw[thick] (N2)--(p7);
\draw[thick] (N7)--(p5)--(N6);
\draw[thick] (N4)--(p7)--(N3);
\draw[thick] (N1)--(p6);
\draw[thick] (p6)--(p3)--(p7);
\draw[thick] (p4)--(p2)--(p5);
\draw[thick] (p3)--(p1)--(p2);
\end{tikzpicture}
\begin{tikzpicture}
\equiccnolabels[1cm]{10}{0}{0}
\coordinate (p1) at (0,0);
\coordinate (p2) at (-.3,0);
\coordinate (p3) at (.3,0);
\coordinate (p4) at (-.5,.3);
\coordinate (p5) at (-.5,-.3);
\coordinate (p6) at (.5,.3);
\coordinate (p7) at (.5,-.3);

\draw[thick] (N10)--(p6);
\draw[thick] (p1)--(N5);
\draw[thick] (N9)--(p4)--(N8);
\draw[thick] (N2)--(p7);
\draw[thick] (N7)--(p5)--(N6);
\draw[thick] (p7)--(N3);
\draw[thick] (N1)--(p6);
\draw[thick] (p1)--(N4);
\draw[thick] (p6)--(p3)--(p7);
\draw[thick] (p4)--(p2)--(p5);
\draw[thick] (p3)--(p1)--(p2);
\end{tikzpicture}
\end{center}
\end{example}

\section{Future Directions}

We have given a rotationally invariant basis for $S^{(d^3, 1^{n-3d})}$, and a natural question to ask is whether there is a way to generalize these results to $r>3$. Many of the results in Section 5 as well as the crossing reduction rule in Section 6 readily generalize when $r$ is odd (when $r$ is even, a different treatment of signs is needed, as a change in orientation will not give a global change in sign). We can thus obtain spanning set for an $\mathfrak{S}_n$ invariant submodule containing $S^{d^r, 1^{n-rd}}$ and the question remains as to how to prune it down to a basis, as the results of Section 4 do not seem to readily generalize. There are two directions in which one might approach this problem. The first is to begin by looking for a rotationally invariant set of the right size via $WNC(n,d,r)$:
\begin{problem}
\label{bijectgeneral}
Is there a combinatorially nice injection of $WNC(n,d,r)$ for $r>3$ into the set of normal plabic graphs, such that the image is closed under rotation?
\end{problem}
The second approach is to determine skein relations first, and use those to reduce the set of normal plabic graphs to a basis.
\begin{problem}
\label{skeingeneral}
Extend the definitions from Section 5 to $r>3$ for $r$ odd. What are the corresponding skein relations?
\end{problem}
 We expect these questions to likely be quite difficult, as answering both would encompass constructing a rotationally invariant basis of $S^{d^r}$, a question which was only recently answered in the case $r = 4$ by C. Gaetz, O Pechenik, S. Pfannerer, J. Striker, and J. Swanson \cite{GPPSS} and remains open for $r >4$. However, most investigation into this question has been concerned with finding a {\em subset} of $SL_n$ webs which forms a basis. Towards this end, it is perhaps a feature, rather than a bug of our construction that it does not consist of genuine $SL_r$ webs, but rather linear combinations of ones with the same underlying simple graph and its minors, as it gives a new place to search.

If we do consider the difference between our augmented web invariants and $SL_3$ web invariants to be something to be fixed, we can do so by constructing a poset on classical $SL_3$ webs with $W \leq V$ whenever $W$ is a weblike subgraph of $V$, which is equivalent to the graph minor poset restricted to $SL_3$ webs. By Corollary~\ref{weblikecor} can thus write
\[
[W] = \sum_{V \leq W} h(V,W) [V]_{SL_3}
\]
where $h$ is an element of the incidence algebra for this poset which is defined up to sign in Corollary~\ref{weblikecor}, and the sign is defined implicitly in the preceeding exposition. Inverting $h$ would then recover the classical $SL_3$ web invariants. We thus propose the following:
\begin{problem}
Extend the definition of $h$ to the poset of perfectly orientable normal plabic graphs with order given by graph minors. Is there a simple combinatorial description of the inverse of $h$ in the incidence algebra?
\end{problem}

One important property of the $m$-diagram construction of $SL_3$ webs is that, as shown by Petersen, Pylyavskyy, and Rhoades, it intertwines promotion on rectangular tableaux and rotation of webs \cite{PPR}, thereby giving an algebraic proof of the cyclic sieving phenomenon for promotion on three-row rectangles. This is not the case for $n>3d$, however, as promotion for tableaux of shape $(d^3, 1^{n-3d})$ for $n>3d$ is not so well-behaved and the order of promotion does not divide $n$ in general. It may be interesting to investigate if there is a variant of promotion which our bijections do intertwine.
\begin{problem}
Give a combinatorial description similar to promotion of the cyclic action on standard Young tableaux of shape $(d^3, 1^{n-3d})$ given by the pullback of rotation on webs. Does the combinatorial description have a natural extension to other shapes? If so, which shapes have order dividing $n$?
\end{problem}
A combinatorially defined cyclic action with order dividing $n$ for another family of partition shapes would be good evidence for the existence of a web basis for those shapes. It is not clear that the bijection we give is necessarily the most natural, so in answering this question one may want to consider other possible bijections between standard Young tableaux and augmented webs.

\section{Acknowledgements}
We thank Brendon Rhoades, Oliver Pechenik, and Pasha Pylyavksyy for helpful comments and conversations.

\bibliographystyle{amsplain}
\bibliography{augwebs6.bib}

\end{document}